\documentclass[12pt]{amsart}

\usepackage{amsmath,amssymb,amsxtra,anysize,times,tikz,rotating,combelow,xcolor}
\usepackage{wrapfig}
\usepackage[all,cmtip]{xy}
\usepackage{amscd}
\usepackage{stackrel}
\usepackage{extpfeil}
\marginsize{1in}{1in}{1in}{1in}

\numberwithin{equation}{section}

\newtheorem{theorem}{Theorem}[section]

\newtheorem{lemma}[theorem]{Lemma}

\newtheorem{assumption}[theorem]{Assumption}
\newtheorem{proposition}[theorem]{Proposition}
\newtheorem{conjecture}[theorem]{Conjecture}

\newtheorem{corollary}[theorem]{Corollary}

\newtheorem*{Conjecture}{Conjecture \ref{conj:1}}
\theoremstyle{definition}
\newtheorem{definition}[theorem]{Definition}

\theoremstyle{remark}
\newtheorem{remark}[theorem]{Remark}
\newtheorem{example}[theorem]{Example}

\DeclareMathOperator{\Id}{Id}

\DeclareMathOperator{\FHilb}{FHilb}

\DeclareMathOperator{\Span}{Span}

\DeclareMathOperator{\Tr}{Tr}
\DeclareMathOperator{\Hom}{Hom}
\DeclareMathOperator{\Ext}{Ext}
\DeclareMathOperator{\RHom}{{RHom}}

\DeclareMathOperator{\HHH}{HHH}

\DeclareMathOperator{\Cone}{Cone}
\renewcommand{\AA}{\mathbb{A}}
\newcommand{\CC}{\mathbb{C}}
\newcommand{\NN}{\mathbb{N}}
\newcommand{\PP}{\mathbb{P}}
\newcommand{\ZZ}{\mathbb{Z}}

\newcommand{\CB}{\mathcal{B}}
\newcommand{\CA}{\mathcal{A}}
\newcommand{\cat}{\mathcal{C}}
\newcommand{\bcat}{\overline{\cat}}
\renewcommand{\CD}{\mathcal{D}}
\newcommand{\bCD}{\overline{\CD}}
\newcommand{\CE}{\mathcal{E}}

\newcommand{\CL}{\mathcal{L}}
\newcommand{\CM}{\mathcal{M}}
\newcommand{\CN}{\mathcal{N}}
\newcommand{\CO}{\mathcal{O}}
\newcommand{\CP}{\mathcal{P}}
\newcommand{\CF}{\mathcal{F}}
\newcommand{\CQ}{\mathcal{Q}}
\newcommand{\CT}{\mathcal{T}}
\newcommand{\CV}{\mathcal{V}}
\newcommand{\CW}{\mathcal{W}}
\newcommand{\BZ}{\mathbb{Z}}

\newcommand{\bFT}{\overline{\FT}}
\newcommand{\FT}{\mathbf{FT}}
\renewcommand{\1}{\mathbf{1}}

\newcommand{\coh}{\mathrm{Coh}}
\newcommand{\qcoh}{\mathrm{QCoh}}

\newcommand{\cone}{\mathrm{Cone}}
\newcommand{\cyc}{\text{cyc}}
\newcommand{\fl}{\text{Fl}}
\newcommand{\End}{\mathrm{End}}
\newcommand{\FH}{\mathrm{FHilb}}
\renewcommand{\H}{\mathrm{Hilb}}
\newcommand{\oFH}{\mathring{\FH}}
\newcommand{\oH}{\mathring{\H}}
\newcommand{\he}{\mathrm{h.e.}}
\newcommand{\qis}{\mathrm{q.i.s.}}

\newcommand{\point}{\mathrm{point}}
\newcommand{\spec}{\mathrm{Spec}}
\newcommand{\proj}{\mathrm{Proj}}
\newcommand{\supp}{\mathrm{supp}}

\newcommand{\SBim}{\mathrm{SBim}}
\renewcommand{\star}{* \in \{\CC^2, \CC, \point\}}
\newcommand{\torus}{\CC^{*} \times \CC^*}

\newcommand{\fg}{\mathfrak{g}}
\newcommand{\fb}{\mathfrak{b}}

\newcommand{\fm}{\mathfrak{m}}
\newcommand{\fn}{\mathfrak{n}}
\newcommand{\fz}{\mathfrak{z}}
\newcommand{\ofm}{\mathring{\fm}}
\newcommand{\sq}{\square}
\newcommand{\bsq}{\blacksquare}

\newcommand{\oCA}{\mathring{\CA}}

\newcommand{\bA}{\overline{A}}
\newcommand{\bZ}{\overline{Z}}

\newcommand{\Ker}{\text{Ker}}
\newcommand{\Tan}{\text{Tan}}
\newcommand{\fsl}{\mathfrak{sl}}
\newcommand{\fgl}{\mathfrak{gl}}

\newcommand{\dg}{\text{dg}}%{\wedge^\bullet \CT_n^\vee}
\newcommand{\edg}{\emph{dg}}
\newcommand{\bFH}{\overline{\FH}}
\newcommand{\bCT}{\overline{\CT}}
\newcommand{\Rred}{\overline{R}}
\newcommand{\Bred}{\overline{B}}

\newcommand{\tot}{\mathrm{Tot}}
\newcommand{\Adj}{\mathrm{Adj}}
\newcommand{\adj}{\mathrm{adj}}

\newcommand{\expdim}{\mathrm{exp \ dim \ }}
\renewcommand{\he}{\text{h.e.}}
\newcommand{\ehe}{\emph{h.e.}}
\newcommand{\HOM}{\text{HOMFLY--PT}}

\author{Eugene Gorsky}
\author{Andrei Negu\cb t}
\author{Jacob Rasmussen}
\title{Flag Hilbert schemes, colored projectors and Khovanov-Rozansky homology}

\begin{document}

\begin{abstract}

We construct a categorification of the maximal commutative subalgebra of the type $A$ Hecke algebra. Specifically, we propose a monoidal functor from the (symmetric) monoidal category of coherent sheaves on the flag Hilbert scheme to the (non-symmetric) monoidal category of Soergel bimodules. The adjoint of this functor allows one to match the Hochschild homology of any braid with the Euler characteristic of a sheaf on the flag Hilbert scheme. The categorified Jones-Wenzl projectors studied by Abel, Elias and Hogancamp are idempotents in the category of Soergel bimodules, and they correspond to the renormalized Koszul complexes of the torus fixed points on the flag Hilbert scheme. As a consequence, we conjecture that the endomorphism algebras of the categorified projectors correspond to the dg algebras of functions on affine charts of the flag Hilbert schemes. We define a family of differentials $d_{N}$ on these dg algebras and conjecture that their homology matches that of the $\fgl_N$  projectors, generalizing earlier conjectures of the first and third authors with Oblomkov and Shende.
\end{abstract}

%\begin{abstract}
%We define a family of algebras $\CHa(T)$  labeled by standard Young  tableaux which are naturally defined by the local geometry of the flag Hilbert scheme. We conjecture that they coincide with the Hochschild homology of the categorified projectors in the Hecke algebra and hence represent the colored HOMFLY-PT homology of the unknot. Using the recent work of Hogancamp and Abel, we verify this conjecture for the symmetric and antisymmetric projector. Furthermore, we define a family of differentials $d_{N}$ on these algebras and conjecture that their homology match the homology of the $sl(N)$  projectors, generalizing our earlierconjectures with Oblomkov and Shende.
%\end{abstract}

\maketitle

\section{Introduction}

\subsection{} It has been slightly more than ten years since Khovanov and Rozansky defined a triply-graded homology theory \(\HHH\) categorifying the HOMFLY-PT polynomial \cite{KhR2}. We have learned a lot about the structure of this invariant in the intervening time, but there is much that remains mysterious. In \cite{GCatalan}, the third author conjectured  a relation between \(\HHH\) of the \((n,n+1)\) torus knot and the $q,t$-Catalan numbers studied by Haiman and Garsia \cite{GarsiaHaimanCatalan, Hai}. A key feature of this conjecture is that it relates \(\HHH(T(n,n+1))\) to the cohomology of a particular sheaf on the Hilbert scheme of \(n\) points in \(\CC^2\). This idea was developed further in \cite{GORS}, and later in \cite{GN}, which identified the sheaves which should correspond to arbitrary torus knots \(T(m,n)\). This paper grew out of our attempts to understand whether  \(\HHH\) of any closed \(n\)-strand braid in the solid torus can be described as the cohomology of some element of the derived category of coherent sheaves on the Hilbert scheme. 

We conjecture that this is indeed the case (Conjecture~\ref{conj:1} below). More importantly, we introduce a   mechanism which we hope can be used to prove it. Two ideas play an important role in our construction. The first (already present in \cite{GN}) is that one should use the flag Hilbert scheme rather than the usual Hilbert scheme. The second is the notion of categorical diagonalization introduced by Elias and Hogancamp in \cite{EH}.  In Theorem~\ref{th:1}, we give a geometric characterization of categorical diagonalization in terms of the bounded derived category of sheaves on projective spaces.
Using this formulation, we show that Conjecture~\ref{conj:1} would follow from some very specific facts about the Rouquier complex of certain braids. 
Finally, as an application of our ideas, we describe how the homology of colored Jones-Wenzl projectors is related to the local rings at fixed points of the natural torus action on the flag Hilbert scheme. 

\subsection{}
\label{sub:intro hecke}

Recall the Hecke algebra $H_n$ of type $A_n$, whose objects can be perceived as isotopy classes of braids on $n$ strands modulo the relation:
$$
\left(\sigma_k - q^{\frac 12} \right)\left(\sigma_k + q^{-\frac 12} \right) = 0 %\sigma_k - q \sigma_k^{-1} = 1-q
$$ 
where $\sigma_k$ denotes a single crossing between the $k$ and $(k+1)$--th strands. The product in the Hecke algebra corresponds to stacking braids on top of each other, from which the non-commutativity of $H_n$ is manifest. Ocneanu constructed a collection of linear maps:
\begin{equation}
\label{eqn:jo}
\chi : \bigsqcup_{n=0}^\infty H_n \rightarrow \CC(a,q)
\end{equation}
which is uniquely determined by the fact that $\forall \ \sigma,\sigma' \in H_n$ we have $\chi(\sigma \sigma') = \chi(\sigma' \sigma)$, and:
\begin{equation}
\label{eqn:jo2}
\chi(i(\sigma)) = \chi(\sigma) \cdot \frac{1-a}{q^{\frac 12} - q^{-\frac 12}}, \qquad
\chi(i(\sigma) \sigma_n) = \chi(\sigma), \qquad \chi(i(\sigma) \sigma_n^{-1}) = \chi(\sigma) \cdot a
\end{equation}
where $i(\sigma) \in H_{n+1}$ is the braid obtained by adding a single free strand to the right of $\sigma$. Jones (\cite{Jones}, \cite{HOMFLY}) showed that the map \eqref{eqn:jo} is an invariant of the closure $\overline{\sigma}$ of the braid:
\begin{equation}
\label{eqn:jones}
\HOM(\overline{\sigma}) = \chi(\sigma)
\end{equation}
which in fact coincides with the well-known HOMFLY-PT knot invariant. The map $\chi$ factors through a maximal commutative subalgebra \(C_n\):
\begin{equation}
\label{eqn:subalgebra}
C_n \stackrel{\iota^*}\hookrightarrow H_n \stackrel{\iota_*}\longrightarrow C_n \qquad \text{by which we mean that} \qquad \chi: H_n \stackrel{\iota_*}\longrightarrow C_n \stackrel{\int}\longrightarrow \CC(a,q)
\end{equation}
for some linear map $\int$ that will be explained later. As a vector space, the commutative algebra $C_n$ is spanned by the Jones-Wenzl projectors to irreducible subrepresentations of the regular representation of $H_n$. As such, $\dim C_n$ equals the number of standard Young tableaux of size $n$, while $\dim H_n = n!$. Alternatively, one can describe $C_n$ in terms of the \textbf{twists}:
\begin{equation}
\label{eqn:fulltwist}
\FT_k=(\sigma_1\cdots\sigma_{k-1})^{k}
\end{equation}
for all $k\in \{1,...,n\}$. Note that $\FT_1 = \1$, while $\FT_n$ is central in the braid group. The fact that $\FT_1,...,\FT_n$ generate a maximal commutative algebra (precisely our $C_n$) is well-known.

\subsection{}
\label{sub:intro categorification}

The Hecke algebra admits a well-known categorification, namely the monoidal category:
$$
(\SBim_n, \otimes_R) \ \leadsto \ K(\SBim_n) = H_n
$$
of certain bimodules over $R = \CC[x_1,...,x_n]$ called \textbf{Soergel bimodules} (see \cite{Soergel},\cite{Rou}). This category admits three gradings:

\begin{itemize}

\item the {\bf internal grading} given by considering graded bimodules with respect to $\deg x_i = 1$. We write $q$ for the variable that keeps track of this grading.

\item the {\bf homological grading} that arises from chain complexes in the homotopy category $K^b(\SBim_n)$. We write $s$ for the variable that keeps track of this grading.

\item the {\bf Hochschild grading} that appears when considering $D^b(\SBim_n)$, namely the closure of $\SBim_n$ in $D^b(R\text{--mod--}R)$. We write $a$ for the corresponding variable. 

\end{itemize}

Khovanov (\cite{Kh}) used the above structure to construct the functor:
\begin{equation}
\label{eqn:hhh}
\HHH:K^b(D^b(\SBim_n)) \longrightarrow \text{triply graded vector spaces}
\end{equation}
such that:
$$
\text{the Poincar\'e polynomial of }\HHH(\sigma) = \sum_{i,j,k=0}^\infty q^i s^j a^k \cdot \dim \HHH(\sigma)_{i,j,k}
$$
only depends on $\overline{\sigma}$ and specializes to \eqref{eqn:jones} when we substitute $s \mapsto -1$ and $a \mapsto -a$. One of the main goals of this paper is to construct a geometric version of the functor \eqref{eqn:hhh}, by categorifying the maximal commutative subalgebra $C_n$ and the maps of \eqref{eqn:subalgebra}. The natural place to look is the category of coherent sheaves on an algebraic space. In our case, the appropriate choice will be the \textbf{flag Hilbert scheme} $\FH_n(\CC)$ which parametrizes full flags of ideals:
$$
I_n \subset ... \subset I_1 \subset I_0 = \CC[x,y]
$$
such that each successive inclusion has colength $1$ and is supported on the line $\{y=0\}$. For every $k\in \{1,...,n\}$, there is a tautological rank $k$ vector bundle:
\begin{equation}
\label{eqn:def tautological}
\CT_k \ \ \text{ on } \ \ \FH_n(\CC), \qquad \CT_k|_{I_n \subset ... \subset I_1 \subset I_0} = \CC[x,y]/I_k
\end{equation}
which is naturally equivariant with respect to the action:
$$
\torus \curvearrowright \FH_n(\CC) \qquad \text{with }\textbf{equivariant parameters} \quad q \text{ and }t
$$
that is induced by the standard action $\torus \curvearrowright \CC \times \CC$. These parameters are related to the gradings on the category of Soergel bimodules via:
\begin{equation}
\label{eqn:matching}
s = -\sqrt{qt}
\end{equation}
In Subsection \ref{sub:dg scheme} we will introduce a certain dg version of the flag Hilbert scheme, denoted by $\FH_n^{\dg}(\CC)$, which is rigorously speaking a sheaf of dg algebras over $\FH_n(\CC)$. Our main conjecture is the following: \\

%\begin{conjecture}
%\label{conj:0}

%There exists a pair of adjoint functors:
%\begin{equation}
%\label{eqn:conj0}
%D^b(\coh(\FH^\edg_n(\CC)) \xtofrom[\iota_*]{\iota^*} K^b(\SBim_n)
%\end{equation}
%where $\iota^*$ is monoidal and fully faithful. Furthermore, we have: 
%$$
%\iota_* \1_{\SBim_n} = \CO_{\FH^\edg_n(\CC)},
%$$
%and:
%\begin{equation}
%\label{eqn:addendum0}
%\iota_{*}(\FT_k) = ( \det \CT_k ) \otimes \CO_{\FH^\edg_n(\CC)}
%\end{equation}
%for all $k\in \{1,...,n\}$. The tautological vector bundles $\CT_k$ on $\FH_n(\CC)$ are defined in\eqref{eqn:deftaut0}.
%\end{conjecture}

\begin{conjecture}
\label{conj:1}

There exists a pair of adjoint functors which preserve the $q$ and $t$ gradings:
\begin{equation}
\label{eqn:conj1}
K^b(\SBim_n) \xtofrom[\iota^*]{\iota_*} D^b\left(\coh_{\torus} \left( \FH_n^\edg(\CC) \right) \right)
\end{equation}
where $\iota^*$ is monoidal and fully faithful. Furthermore, we have: 
\begin{equation}
\label{eqn:addendum}
\FT_k \quad \xtofrom[\iota_*]{\iota^*} \quad ( \det \CT_k ) \otimes \CO_{\FH^\edg_n(\CC)}
\end{equation}
for all $k\in \{1,...,n\}$. Moreover, the map $\HHH$ of \eqref{eqn:hhh} factors as:
\begin{equation}
\label{eqn:hhh factors}
\HHH: K^b(\SBim_n) \stackrel{\iota_*}\longrightarrow D^b\left(\coh_{\torus} \left( \FH_n^\edg(\CC) \right) \right) \stackrel{\int}\longrightarrow \text{3-graded vector spaces}
\end{equation}
where $\int$ refers to the derived push-forward map in equivariant cohomology. \\

\end{conjecture}

\begin{remark}
\label{rem:a grading}

To account for the $a$ grading in \eqref{eqn:conj1} and \eqref{eqn:hhh factors}, we conjecture that one can lift the setup of Conjecture \ref{conj:1} to functors:
\begin{equation}
\label{eqn:conj2}
K^b(D^b(\SBim_n)) \xtofrom[\widetilde{\iota}^*]{\widetilde{\iota}_*} D^b\left(\coh_{\torus} \left( \tot_{\FH_n^\dg(\CC)} \CT_n[1] \right) \right)
\end{equation}
which preserve the $q,t$ and $a$ gradings, defined by: 
\begin{equation}
\label{eqn:iota wedge}
\widetilde{\iota}_*(\sigma) = \iota_*(\sigma) \otimes \wedge^\bullet \CT_n^\vee
\end{equation}
where $a$ keeps track of the exterior degree in the right hand side. With this in mind, we note that the target of the map $\int$ from \eqref{eqn:hhh factors} can be lifted to quadruply graded vector spaces, since we may separate the derived category grading on $\FH_n^\dg(\CC)$ from the exterior grading $a$.

\end{remark}

\subsection{}
\label{sub:intro homology}

Besides the fact that the category $D^b(\coh_{\torus} (\FH_n^\dg(\CC)))$ and the functors $\iota_*$, $\iota^*$ categorify \eqref{eqn:subalgebra}, one of the main applications of Conjecture \ref{conj:1} is a geometric incarnation of Khovanov's Hochschild homology functor. Indeed, since $\SBim_n$ is a categorification of the Hecke algebra, to any braid $\sigma$ one may associate a homonymous object $\sigma \in K^b(\SBim_n)$ (see Section \ref{sec:Sbim} for an overview). Therefore, we have:
\begin{equation}
\label{eq:HOMFLY homology}
\HHH(\sigma) = \int_{\FH_n^\dg(\CC)} \CB(\sigma) \otimes \wedge^\bullet \CT_n^\vee \qquad \text{where} \qquad \CB(\sigma) :=\iota_*(\sigma)
\end{equation}
is the sheaf on the dg scheme $\FH_n^\dg(\CC)$ that our construction associates to the braid $\sigma$. We tensor with $\wedge^\bullet \CT_n^\vee$ as in Remark \ref{rem:a grading} in order to pick up the $a$ grading on $\HHH(\sigma)$ (if we had not taken this tensor product, we would recover $\HHH(\sigma)|_{a=0}$). While it is difficult to describe at the moment the sheaves $\CB(\sigma)$ for arbitrary braids $\sigma$, properties \eqref{eqn:addendum} and the projection formula \eqref{eqn:projection} imply that:
$$
\CB\left(\prod_{k=1}^n \FT_k^{a_k}\right) = \bigotimes_{k=1}^n ( \det \CT_k)^{\otimes a_k}
$$
Therefore, \eqref{eq:HOMFLY homology} immediately implies the following Corollary for all products of twists:

%In particular, we may consider certain special braids called twists:  
%\begin{equation}
%\label{eqn:fulltwist}
%\FT_k=(\sigma_1\cdots\sigma_{k-1})^{k}
%\end{equation}
%where $\sigma_1,...,\sigma_{n-1}$ are the standard generators of the braid group on $n$ strands. It is known that the objects $\FT_k$ commute in the braid group and the images of $\FT_k$ generate the commutative subalgebra $C_n \subset H_n$. The following conjecture gives a geometric interpretation of these twists. \\

%\noindent \textbf{Addendum to Conjectures \ref{conj:0} and \ref{conj:1}:} {\emph For all $k \in \{1,...,n\}$, one has:} 
%\begin{equation}
%\label{eqn:addendum}
%\CB_{\FT_k} = \det \CT_k
%\end{equation}
%\emph{where the tautological vector bundle $\CT_k$ on $\FH_n(\CC)$ is defined in Section \ref{sub:defflag}.} \\

%The above conjectures imply the following formula for the HOMFLY-PT homology of arbitrary products of twists.

\begin{corollary}
\label{cor:1}
For all $(a_1,\ldots,a_n) \in \ZZ^n$, let us consider the twist braid
$\sigma=\prod_{k} \FT^{a_k}_k$. Assuming Conjecture \ref{conj:1}, the HOMFLY-PT homology of the closure of $\sigma$ is given by:
\begin{equation}
\label{eqn:hhh twist}
\emph{HHH}(\sigma) = \int_{\FH_n^{\edg}(\CC)} \bigotimes_{k=1}^n(\det \CT_k)^{\otimes a_k}\bigotimes \wedge^{\bullet}\CT_n^{\vee} 
\end{equation}
where the integral denotes the derived equivariant pushforward to a point. 
\end{corollary}

%Under certain positivity assumptions, we can combine Corollary \ref{cor:1} with the Thomason localization formula in equivariant $K$--theory, and obtain the following explicit formula for the Khovanov-Rozansky homology of an infinite family of braids. 

%\begin{corollary}
%\label{cor:localization}
%For any collection of non-negative integers $a_1,\ldots,a_n \geq 0$, consider the braid $\sigma=\prod_{i} \FT^{a_k}_k$. Conjecture \ref{conj:1} implies that the Poincar\'e polynomial of the HOMFLY-PT homology of the closure of $\sigma$ is given by:

When the \(a_i\) are sufficiently positive, we expect that the higher cohomology of the sheaf appearing in the the right-hand side of \eqref{eqn:hhh twist} should vanish. If this is the case, the right-hand side of \eqref{eqn:hhh twist} can be computed using  the Thomason localization formula as in \cite{GN} to give:
\begin{equation}
\label{eqn:magic formula}
\HHH(\sigma) =(1-q)^{-n}\sum_{T}\prod_{i=1}^{n}\frac{z_i^{a_i+\ldots+a_n}(1+az_i^{-1})}{1-z_i^{-1}} \prod_{1\leq i<j \leq n} \zeta \left(\frac {z_i}{z_j}\right)
\end{equation}
where the sum goes over all standard tableaux $T$ of size $n$, the variable $z_i$ denotes the $(q,t)$--content of the box labeled $i$ in each such tableau $T$, and:
$$
\zeta(x)=\frac{(1-x)(1-qtx)}{(1-qx)(1-tx)}.
$$
We will explain how to obtain \eqref{eqn:magic formula} in Section \ref{sec:local}, when we discuss the equivariant structure of the flag Hilbert scheme. In Section \ref{sub:knots}, we will explain how to amend Corollary \ref{cor:1} to account for torus knot braids rather than pure braids. Once we will do this, Corollary \ref{cor:1} gives a generalization of one of the main conjectures of \cite{GN} (which dealt with the case when $\sigma$ is a torus knot braid).

\subsection{}

Since $\HHH(\sigma)$ only depends on the closure $\overline{\sigma}$, formula \eqref{eq:HOMFLY homology} might suggest that the coherent sheaf $\CB(\sigma)$ actually only depends on $\overline{\sigma}$. While this cannot be strictly speaking true (after all, $\CB(\sigma)$ lives on $\FH_n^\dg(\CC)$ where $n$ is the number of strands of the braid), we may consider the natural map from the flag Hilbert scheme to the usual Hilbert scheme of $n$ points on $\CC^2$:
$$
\FH_n^\dg(\CC) \stackrel{\nu}\longrightarrow \H_n
$$
\begin{equation}
\label{eqn:map to hilb}
\left( I_n \subset ... \subset I_0 \right) \mapsto I_n
\end{equation}
The composition:
$$
K^b(\SBim_n) \xrightarrow{\iota_*} D^b \left(\coh_{\torus} \left(\FH_n^\dg(\CC) \right) \right) \xrightarrow{\nu_*} D^b \left(\coh_{\torus} \left(\H_n \right) \right)
$$
associates to a braid $\sigma$ a complex of sheaves:
\begin{equation}
\label{eqn:sheaf on hilb}
\CF(\sigma) = \nu_*(\CB(\sigma))
\end{equation}
We may tensor this complex with $\wedge^\bullet \CT_n^\vee$ as in Remark \ref{rem:a grading} if we also wish to encode the $a$ grading. This is the object we conjecture gives rise to the geometrization of \eqref{eqn:jo}.

%The map $\nu$ extends to the affine dg bundles $\tot_{\FH^\dg_n(\CC)} \CT_n[1] \stackrel{\nu'}\longrightarrow \tot_{\H_n} \CT_n[1]$. The diagram:
%\begin{equation}
%\label{eqn:jo}
%\begin{tikzcd}
%K^b(D^b(\SBim_n)) \arrow{r}{\iota_*} & D^b\left(\coh_{\torus} \left( \tot_{\FH^\dg_n(\CC)} \CT_n[1] \right) \right) \arrow{d}{\nu'_*} \\
%&D^b\left(\coh_{\torus} \left( \tot_{\H_n} \CT_n[1] \right) \right)
%\end{tikzcd} \quad \begin{tikzcd}
%B_\sigma \arrow{r} & \CB_\sigma \arrow{d} \\
%& \CF_\sigma
%\end{tikzcd}
%\end{equation}

%$$
%\left( W_n, \CO_{W_n} \right) = \left( \H_n, \wedge^\bullet \CT_n^{\vee} \right) 
%$$
%and note that the map $\nu$ implies the existence of a map:
%$$
%Z_n(\CC) \stackrel{\nu'}\longrightarrow W_n.
%$$

\begin{conjecture}
\label{conj:2}
The objects $\CF(\sigma)$ satisfy the following properties:
\begin{equation}
\label{eqn:nakajima 0}
\CF(\sigma \sigma') \cong \CF(\sigma' \sigma)
\end{equation}
for all braids $\sigma$ and $\sigma'$ on $n$ strands, and:
\begin{equation}
\label{eqn:nakajima 1}
\CF(i(\sigma)) \ = \ \alpha\left(\CF(\sigma)\right)
\end{equation}
where:
\begin{equation}
\label{eqn:nakajima}
\alpha:D^b\left(\coh_{\CC^* \times \CC^*} \left(\H_n \right)\right) \longrightarrow D^b\left(\coh_{\CC^* \times \CC^*} \left(\H_{n+1} \right)\right)
\end{equation}
denotes the simple correspondence of Nakajima and Grojnowski (as in Subsection \ref{sub:nakajima}). 

%\begin{equation}
%\label{eqn:nakajima 2}
%\CF(i(\sigma) \cdot \sigma_n) = \alpha'\left(\CF(\sigma)\right)
%\end{equation}
%\begin{equation}
%\label{eqn:nakajima 3}
%\CF(i(\sigma) \cdot \sigma_n^{-1}) = \alpha''\left(\CF(\sigma)\right)
%\end{equation}

\end{conjecture}

For any braid $\sigma$, the Euler characteristic of $\CF(\sigma)$ at $t = \frac 1q$ coincides with $\chi(\sigma)$ of \eqref{eqn:jo}. %(compare \eqref{eqn:jo2} with Proposition \ref{prop:markov integrals}).

\begin{remark}
\label{rem:oblomkov rozansky}

While the present paper was being written, Oblomkov and Rozansky (\cite{OR}) independently gave an alternative construction of objects very similar to $\CB(\sigma)$ and $\CF(\sigma)$, although in a very different presentation. Specifically, their construction associates to any braid an object in the category of matrix factorizations, which descends to an object on the commuting variety. The authors then show that the corresponding object is actually supported on the Hilbert scheme. We strongly suspect that their objects coincide with ours, and hope that the connection will be elucidated in the near future.

\end{remark}

\subsection{}
\label{sub:intro proof}

We show that  Conjecture \ref{conj:1} would follow from certain computations in the Soergel category, which we believe may be proved using the techniques developed in an upcoming paper of Elias and Hogancamp (see \cite{EH1} for a special case). In the present paper, we develop the geometric machinery necessary to prove such results. Specifically, we outline a strategy for constructing  the functors $\iota^*,\iota_*$ with equation \eqref{eqn:addendum} in mind. The starting point for us is to reinterpret geometrically a concept introduced by Elias and Hogancamp under the name of {\bf categorical diagonalization} (\cite{EH}). 
Suppose that $\cat$ is a graded monoidal category with monoidal unit $\mathbf{1}$, and $F$ is an object in the homotopy category $K^b(\cat)$.
%Suppose we are given an object $F$ in a graded monoidal category $\cat$. 
Elias and Hogancamp call $F$ diagonalizable if there exist grading shifts $\lambda_0,...,\lambda_n$ and morphisms: 
$$
\alpha_i:\lambda_i \cdot \1\to F, \qquad i=0,\ldots,n
$$
satisfying certain conditions (see Definitions \ref{def:EH1} and \ref{def:EH2}). Under these conditions, it is proved in \cite{EH} that there exist objects $P_i \in K(\bcat)$ (a certain completion, whose relation with the original category $K^b(\cat)$ is analogous to the relation between the categories of left unbounded chain complexes and bounded chain complexes) such that tensoring $\Id_{P_i}$ with $\alpha_i$ yields an isomorphism:
\begin{equation}
\label{eqn:eigen0}
\lambda_i \cdot P_i \cong F \otimes P_i, \qquad i=0,\ldots,n
\end{equation}
It is natural to call the $P_i$ {\bf eigenobjects} of $F$ and the $\lambda_i$ the {\bf eigenvalues} of $F$. The maps $\alpha_i$ are called the {\bf eigenmaps} for $F$, and they are a particular feature of the categorical setting. Under mild assumptions on $\cat$ and $F$, we show the following: \\

\begin{theorem}
\label{th:1}
An object $F\in \cat$ is diagonalizable in the sense of \cite{EH} if and only if there is a pair of adjoint functors: 
$$
K^b(\cat) \xtofrom[\iota^*]{\iota_*} D^b(\coh(\PP^n_{A})),
$$
where $A=\End_{\cat}(\1).$ If the category $\cat$ is graded and the maps $\alpha_i$ preserve the grading, then $\iota^*$ and $\iota_*$ can be lifted to the equivariant derived category:
$$
K^b(\cat) \xtofrom[\iota^*]{\iota_*} D^b(\coh_{T}(\PP^n_{A})),
$$
where $T$ is a torus acting on $\PP^n$ with weights prescribed by the eigenvalues of $F$. \\
\end{theorem}

Furthermore, the following result of Elias-Hogancamp provides one of the first proved facts about our conjectural connection between $\SBim_n$ and $\FHilb^\dg_n(\CC)$. 

\begin{theorem}[\cite{EH}]
The full twist $\FT_n$ is diagonalizable in $\SBim_n$, and its eigenvalues agree with the equivariant weights of $\det \CT_n$ at fixed points.
\end{theorem}

The flag Hilbert scheme is more complicated than a projective space, but it turns out to be presented by a tower of projective fibrations. More precisely, the fibers of the natural projection:
$$
\FHilb_n(\CC)\to \FHilb_{n-1}(\CC) \times \CC, \quad (I_n \subset ... \subset I_0) \mapsto (I_{n-1}\subset ... \subset I_0) \times \text{supp}(I_{n-1}/I_n)
$$
are projective spaces. They are rather badly behaved, but we will show in Section \ref{sub:dg scheme} that the corresponding map on the level of our dg schemes:
$$
\pi_n:\FHilb^\dg_n(\CC)\to \FHilb^\dg_{n-1}(\CC) \times \CC
$$
is the projectivization of a two-step complex of vector bundles. The strategy we propose is to use a relative version of Theorem \ref{th:1} (developed in Section \ref{sec:geometry}) in order to construct a commutative tower of functors:
\begin{equation}
\label{eq:tower}
\includegraphics{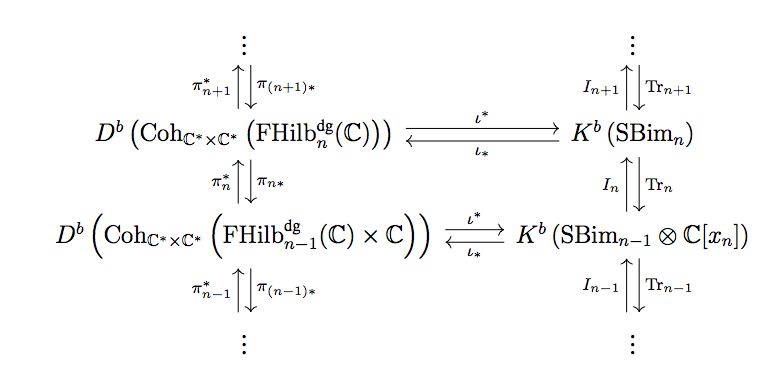}
\end{equation}
Here $I_n:\SBim_{n-1} \otimes \CC[x_{n}] \to \SBim_{n}$ denotes the natural full embedding of categories, while $\Tr_n:\SBim_{n}\to \SBim_{n-1} \otimes \CC[x_{n}]$ is the partial trace map of \cite{Hog} (see Subsection \ref{sub:trace} for details, as well as an overview of the construction of its derived version). We prove that the existence of the horizontal functors in \eqref{eq:tower} is equivalent to the computation of $\Tr_n(\FT_n^{\otimes k})$ for all integers $k$ (see \ref{conj:main} below), together with certain compatibility conditions that must be checked. Assuming these computations, we show how Conjecture \ref{conj:1} follows.

\subsection{} Conjecture \ref{conj:1} implies very explicit facts about the existence of various morphisms and extensions between the twists $\FT_k$ in the Soergel category. The easiest of these conjectures involves the objects $L_k:=\FT_k\otimes \FT_{k-1}^{-1} \in K^b(\SBim_n)$ for all $k\in \{1,...,n\}$:

\begin{conjecture}
\label{conj:T soergel}
There exist objects $T_n,...,T_1 \in K^b(\SBim_n)$ and morphisms $T_n \rightarrow T_{n-1} \rightarrow ... \rightarrow T_1$, which satisfy:
\begin{equation}
\label{eq:def t}
L_k \cong [T_k \rightarrow T_{k-1}]
\end{equation}
for all $k \in \{1,...,n\}$. Furthermore, there exist two commuting morphisms: 
$$
X:q T_k \to T_k \qquad Y: \frac {s^2}q T_k \to T_k
$$
which commute: $[X,Y]=0$ and are compatible with the isomorphisms \eqref{eq:def t}. Moreover, $X|_{L_k}$ is multiplication by the element $x_k \in R$ and $Y|_{L_k} = 0$.
\end{conjecture}

Various matrix elements of products of $X$ and $Y$ can be used to construct morphisms between various $L_k$. See Conjecture \ref{conj:main} for more conjectures of similar kind.

\subsection{} An important role in the geometry of flag Hilbert schemes is played by torus fixed points:
$$
\FH_n(\CC)^{\torus} = \{ I_T \}_{T \text{ is a standard Young tableau of size }n}
$$
While the flag Hilbert scheme is badly behaved, the dg scheme $\FH_n^\dg(\CC)$ is by definition a local complete intersection. As such, the skyscraper sheaves at the torus fixed points are quasi-idempotents in the derived category of coherent sheaves on $\FH_n^\dg(\CC)$:
$$
\CO_{I_T} \otimes \CO_{I_T}\cong  \CO_{I_T} \otimes \wedge^\bullet \left( \text{Tan}_{I_T} \left(\FH_n^\dg(\CC) \right) \right)
$$
where $\text{Tan}$ denotes the tangent bundle (which makes sense for a local complete intersection as a complex of vector bundles). Inspired by the constructions of Elias--Hogancamp (\cite{EH}), we make sense of the objects:
$$
\CP_T \ ``=" \ \left[ \frac {\CO_{I_T}}{\wedge^\bullet \text{Tan}_{I_T} \left( \FH_n^\dg(\CC) \right)} \right] \ \in \ \text{a certain extension of }\coh_{\torus} \left(\FH_n^\dg(\CC) \right)
$$
and conjecture that the functor $\iota^*$ sends this object to the categorified Jones--Wenzl projector:
\begin{equation}
\iota^*\left(\CP_T\right) = P_{T} 
\end{equation}
These projectors are among the main actors of \cite{EH}, where the authors construct them inductively as eigenobjects for the full twists $\FT_n$ following the categorical diagonalization procedure described in \eqref{eqn:eigen0}. In the present paper, we exhibit an affine covering of the flag Hilbert scheme:
$$
\FH_n(\CC) = \bigcup_T \oFH_T(\CC)
$$
If we restrict the structure sheaf $\CO_{\FH_n^\dg(\CC)}$ to these open pieces, we obtain dg algebras:
$$
\CA_T(\CC) = \Gamma \left(\oFH_T(\CC), \CO_{\FH_n^\dg(\CC)} \right)
$$ 
We expect that these dg algebras %categorify the HOMFLY-PT homology 
coincide with the endomorphism algebras of the categorified Hecke algebra idempotent indexed by the standard Young tableau $T$, as in the following conjecture.

\begin{conjecture}
\label{conj:homfly}
The %HOMFLY-PT homology 
endomorphism algebra of the categorified Jones-Wenzl projector $P_{T}$ is isomorphic as an algebra to:
\begin{equation}
\label{eqn:homfly}
%\HHH(P_{T})
\End(P_{T}) =  \CA_T(\CC) \otimes \left( \wedge^{\bullet}\CT_n^{\vee}|_{\oFH_T(\CC)} \right)
\end{equation}
\end{conjecture}

Note that $\CT_n^{\vee}$ is a trivial rank $n$ vector bundle on the affine chart $\oFH_T(\CC)$, and so the exterior power that appears in \eqref{eqn:homfly} is free on $n$ odd generators, whose equivariant weights match the inverse $q,t$--weights of the boxes in the Young tableau $T$. Following recent results of Abel and Hogancamp \cite{AbHog,Hog}, we prove \eqref{eqn:homfly} in the two extremal cases, corresponding to the symmetric and anti--symmetric projectors:
 
\begin{theorem}
\label{thm:symm antisymm}
If $T=(n)$ or $(1,\ldots,1)$ then the endomorphis algebra
%HOMFLY-PT homology 
of the resulting projector is isomorphic to the right hand side of \eqref{eqn:homfly}. Explicitly:
\begin{equation}
\label{eqn:antisym}
\End(P_{(n)})\simeq  \frac {\CC[x_1,\ldots,x_n, y_{i,j}]_{i>j}}{y_{i,j}(x_i-x_j) - (y_{i-1,j} - y_{i,j+1})} \otimes \wedge^{\bullet}(\xi_1,\ldots,\xi_n)
\end{equation}
where $\deg x_i = q$, $\deg y_{i,j} = t q^{j-i}$ and $\deg \xi_i = a q^{1-i}$, while:
\begin{equation}
\label{eqn:sym}
\End(P_{(1,...,1)})\simeq \CC[u_1,\ldots,u_n]\otimes \wedge^{\bullet}(\xi_1,\ldots,\xi_n)
\end{equation}
where $\deg u_i = q t^{1-i}$ and $\deg \xi_i =a t^{1-i}$.  
\end{theorem}

%\begin{remark}
%Note that here $q$ and $t$ denote the equivariant gradings on $\FH_n(\CC)$, and $a$ denotes the exterior grading on $\wedge^\bullet \CT_n^\vee$. The standard gradings $Q,T,A$ of \cite{Kh} coincide with:
%\begin{equation}
%\label{eq:gradings}
%q=Q^2,\ t=Q^{-2}T^{-2},\ a=A^2T.
%\end{equation}
%Here $Q$ is the quantum grading, $T$ is the homological grading and $A$ is the Hochschild grading. Note that \cite{AbHog,Hog} use $T^{-1}$ instead of $T$ and a slightly different convention for $A$.
%\end{remark}

As further evidence for Conjecture \ref{conj:homfly}, we prove that it holds at the decategorified level.

\begin{theorem}
\label{thm:hecke projector}
For all standard Young tableaux $T$ , the Euler characteristic of the algebra:
$$
\oCA_T(\CC) \otimes \left( \wedge^{\bullet}\CT_n^{\vee}|_{\oFH_T(\CC)} \right)
$$
equals the Markov trace of the Hecke idempotent $p_{\lambda}$, where \(\lambda\) is the partition associated to \(T\). %Andrei: isn't this statement imprecise, because the coordinate ring is defined over q,t,a, while the Markov trace has t=1/q?
\end{theorem}

%\begin{remark}
%Topologically, the condition \eqref{eqn: positivity for full twists} says that $\sigma$ is equivalent to a positive braid, that is, after cancellations all of its crossings become positive.
%\end{remark}

\subsection{}
\label{sub:reduced}

One can easily modify the above constructions to describe the {\em reduced} HOMFLY-PT homology. Indeed, it is proven in \cite{RasDiff} that the HOMFLY-PT homology of any braid is a free module over the homology of the unknot, which is isomorphic to a free algebra in one even and one odd variable. Let us explain how these variables arise from the geometry. First, define the {\em reduced} flag Hilbert scheme $\bFH_n(\CC)$ as the subscheme in $\FH_n(\CC)$ cut out by the equation 
$$
\Tr(X)=x_1+\ldots+x_n=0.
$$  
It is not hard to see that there is an isomorphism:
\begin{equation}
\label{eqn:reduction}
r:\FH_n(\CC)\to \bFH_n (\CC)\times \CC %, \qquad r(X,Y,v)=(X-\overline{x}\cdot \Id,Y,v)\times \overline{x},
\end{equation}
%where $\overline{x}=\frac{1}{n}\sum_{i=1}^{n}x_i$.
 We will denote two components of this isomorphism by $r_1$ and $r_2$. As a result, the homology of any sheaf on $\FH_n(\CC)$ is a free module over the polynomial ring in one (even) variable. To identify the odd variable, remark that $\CT_n$ has a nowhere vanishing section given by the polynomial $1 \in \CC[x,y]$. It is not hard to see that this section splits, so we may write:
$$
\CT_n\simeq \CO\oplus \bCT_n \quad \Longrightarrow \quad \CT^{\vee}_n\simeq \CO\oplus \bCT^{\vee}_n \quad \Longrightarrow \quad \wedge^{\bullet}\CT^{\vee}_n\simeq \wedge^{\bullet}(\xi)\otimes \wedge^{\bullet}\bCT^{\vee}_n
$$
To sum up, we get the following corollary analogous to Corollary \ref{cor:1}:

\begin{corollary}
\label{conj:reduced}
Assuming Conjecture~\ref{conj:1}, the reduced HOMFLY-PT homology of any object $\sigma \in K^b(\SBim_n)$ is:
$$
\HHH^{\emph{red}}(\sigma) \cong \int_{\bFH^\edg_n(\CC)} (r_1 \circ \iota)_{*} (\sigma) \otimes \wedge^{\bullet}\bCT^{\vee}_n.
$$
\end{corollary}

\subsection{}

Finally, we give a conjectural geometric description of $\fgl_N$ Khovanov-Rozansky homology \cite{KhR1,KhR2} for all $N$. Recall that in \cite{RasDiff} the third  author constructed a spectral sequence from the HOMFLY-PT homology to the $\fgl_N$ homology of any knot. For any pair of nonnegative integers $N,M$, there is an equivariant section: 
$$
s_{N,M}\in \Gamma \left(\FHilb_n(\CC),\CT_n \right), \qquad s_{N,M}|_{I_n \subset \ldots \subset I_0} = x^Ny^M \in \frac {\CC[x,y]}{I_n} = \CT_n|_{I_n \subset \ldots \subset I_0}
$$

%The way to construct this section is to observe that the fiber of $\CT_{n}$ over a flag $(I_0 \supset \ldots \supset I_n)$ is isomorphic to the quotient $\CC[x,y]/I_n$, and $s_{N,M}(I_0 \supset \ldots \supset I_n)$ is defined as $x^Ny^M$ mod $I_n$.

\begin{conjecture}
\label{conj:diff}
For all braids $\sigma$, the $\fgl_N$ spectral sequence on the homology of $\overline{\sigma}$ is induced by the contraction of:
$$
\wedge^\bullet \CT_n^\vee \quad \text{on} \quad \FH^\edg_n(\CC)
$$
with the section $s_{N,0}$, which induces a differential on the vector space \eqref{eq:HOMFLY homology}.
\end{conjecture}

\begin{remark}
A similar conjecture can be stated for the reduced $\fgl_N$ homology. However, the map \eqref{eqn:reduction} does not commute with the differential, and hence the unreduced homology is no longer a free module over the homology of the unknot.
\end{remark}

We are hopeful that the contraction with more general $s_{N,M}$ may correspond to an (as yet undefined) knot homology theory associated to the Lie superalgebra $\fgl_{N|M}$ (see some conjectural properties in \cite{GGS}). In particular, the differential induced by $s_{1,1}=xy$ should give rise to  a knot homology theory  associated to $\fgl_{1|1}$.  
Recent work of Ellis, Petkova and V\'ertesi \cite{EPV} shows that the tangle Floer homology of \cite{PV} gives  a sort of categorification of the \(\fgl_{1|1}\) Reshitikhin-Turaev invariant. 
In the spirit of the above conjecture,  contraction with $s_{1,1}$ may give rise to a differential on $\HHH$ whose homology is knot Floer homology, as conjectured in \cite{DGR}. 

In an earlier joint work with A. Oblomkov and V. Shende (\cite{GORS}), the first and the third authors gave a precise conjectural description of the stable $\fgl_N$ homology of $(n,\infty)$ torus knots,  which is known (\cite{Cautis,Hog,Rose,Roz}) to be isomorphic to the $\fgl_N$ homology of the categorified projector $\CP_{(1,\ldots,1)}$.
\begin{conjecture}[\cite{GORS}]
\label{conj:GORS}
The spectral sequence from HOMFLY-PT homology (given by \eqref{eqn:sym}) to the $\fgl_N$ homology of $\CP_{(1,\ldots,1)}$ degenerates after the first nontrivial differential $d_N$, which is given by the equation:
\begin{equation}
\label{dN symm}
d_N\left(\sum_{k=1}^{n} z^{k-1}\xi_k\right)=\left(\sum_{k=1}^{n}z^{k-1}u_k\right)^{N}\mod z^n, \qquad d_N(u_i)=0.
\end{equation} 
\end{conjecture}
 This conjecture has been extensively verified against computer-generated data for $N=2$ and $3$ (see \cite{GL,GOR}). We prove that  Conjecture \ref{conj:GORS} immediately follows from Conjecture \ref{conj:diff}.

\subsection{}

This paper is naturally divided into two parts. The first part (Sections \ref{sec:flag}, \ref{sec:Sbim}, \ref{sec:geometry}) presents the non-equivariant picture, which relates the global geometry of the flag Hilbert scheme with the Soergel category. Sections \ref{sec:n=2} and \ref{sec:n=3} present examples of many of our constructions for $n=2$ and $n=3$, respectively. The second part of the paper (Sections \ref{sec:equiv}, \ref{sec:local}, \ref{sec:differentials}) is an equivariant refinement of the previous framework, which relates the local geometry of the flag Hilbert scheme with categorical idempotents in the Soergel category. More specifically:

\begin{itemize}

\item In Section \ref{sec:flag}, we define flag Hilbert schemes and the associated dg schemes, and we realize them as towers of projective bundles. 

\item In Section \ref{sec:Sbim}, we recall the necessary facts about the Hecke algebra and the Soergel category, and formulate the main conjectures.

\item In Section \ref{sec:geometry}, we develop a framework of monoidal categories over dg schemes, which encapsulates the existence of adjoint functors as in \eqref{eqn:conj1}, with all the desired properties. We show what computations one needs to make in order to prove Conjecture \ref{conj:1}.

\item In Section \ref{sec:n=2}, we present examples for $n=2$.

\item In Section \ref{sec:n=3}, we present examples for $n=3$.

\item In Section \ref{sec:equiv}, we show how the categorical setup of Section \ref{sec:geometry} can be enhanced to the equivariant setting. Inspired by the constructions of Elias--Hogancamp, we categorify the equivariant localization formula on projective space.

\item In Section \ref{sec:local}, we work out local equations for flag Hilbert schemes, and connect the structure sheaves of torus fixed points with the categorical projectors of \cite{EH}.

\item In Section \ref{sec:differentials}, we discuss differentials and Conjecture \ref{conj:diff}.

\item In Section \ref{sec:appendix}, we collect certain foundational facts about dg categories and dg schemes.

\end{itemize}

\section*{Acknowledgments}

The authors would like to thank Michael Abel, Ben Elias and Matt Hogancamp for explaining to us their results \cite{AbHog,EH,Hog}, and Mikhail Gorsky, Daniel Halpern-Leistner, Mikhail Khovanov, Ivan Losev, Davesh Maulik, Michael McBreen, Hugh Morton, Alexei Oblomkov, Andrei Okounkov, Claudiu Raicu, Sam Raskin, Raphael Rouquier, Lev Rozansky, Peter Samuelson, Peng Shan and Monica Vazirani for very useful discussions. Special thanks to Ian Grojnowski for explaining Example \ref{ex:not lci} to us. The work of E. G. was partially supported by the NSF grant DMS-1559338 and the Hellman fellowship. The work of A. N. was partially supported by the NSF grant DMS-1600375.
The work of E.G. in sections 3, 5 and 6 was supported by the grant 16-11-10018 of the Russian Science Foundation.

\section{The flag Hilbert scheme}
\label{sec:flag}

\subsection{Definition}
\label{sub:defflag}

Let us recall the usual \textbf{Hilbert scheme} of $n$ points on $\CC^2$:
$$
\H_n = \{\text{ideal } I \subset \CC[x,y], \text{ dim}_{\CC} \ \CC[x,y]/I = n \}
$$
There is a {\bf tautological bundle} of rank $n$ on the Hilbert scheme given by:
$$
\CT_n|_I = \CC[x,y]/I
$$
Similarly, one can define the \textbf{flag Hilbert scheme} $\FH_n(\CC^2)$ of $n$ points on $\CC^2$ \cite{Cheah,Tikh} as the moduli space of complete flags of ideals: 
\begin{equation}
\label{eqn:flagdef}
\FH_n(\CC^2) = \{I_n \subset ... \subset I_1 \subset I_0 = \CC[x,y], \text{ dim } I_{k-1}/I_k = 1, \ \forall k\}
\end{equation}
Clearly, $\FH_n(\CC^2)$ can be thought of as the closed subscheme of $\H_n \times ... \times \H_1 \times \H_0$ cut out by the inclusions $I_k \subset I_{k-1}$ for all $k$. We will not pursue this description, and instead work with an alternative one given in the next Subsection. Meanwhile, let us point out several general features of the flag Hilbert scheme \eqref{eqn:flagdef}. We may pull $\CT_n$ back to $\FH_n(\CC^2)$, where we have a full flag of tautological bundles:

\includegraphics{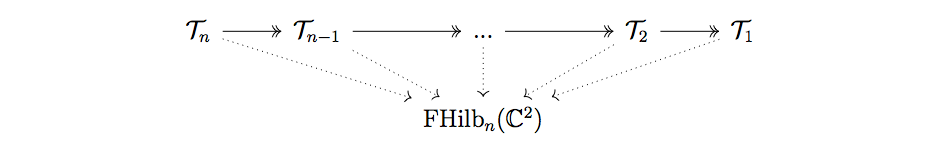}
of ranks $n,...,1$. For any $k \in \{1,...,n\}$, the fibers of $\CT_k$ over flags $I_n \subset ... \subset I_0$ are precisely the quotients $\CC[x,y]/I_k$. We define the tautological line bundles as the successive kernels:
\begin{equation}
\label{eqn:taut line}
\CL_k = \text{Ker} \left( \CT_k \twoheadrightarrow \CT_{k-1} \right)
\end{equation}
Moreover, there is a morphism:
\begin{equation}
\label{eqn:rho}
\rho: \FH_n(\CC^2) \longrightarrow \CC^{2n}=\CC^n\times \CC^n
\end{equation}
$$
(I_n \subset ... \subset I_0) \mapsto (x_1,\ldots, x_n, y_1,\ldots, y_n)
$$
where $(x_k,y_k) = \supp \ I_{k-1}/I_k$. We may consider the various fibers of this map:
$$
\FH_n(\CC)=\rho^{-1}(\CC^n\times \{0\}), \qquad \FH_n(\text{point}) = \rho^{-1}(\{0\}\times \{0\}) 
$$
These will be the moduli spaces of flags of sheaves set-theoretically supported on the line $\{y=0\}$ and at the point $(0,0)$, respectively. The vector bundles $\CT_k$ and $\CL_k$ are defined as before. As a rule, we will write:
$$
\FH_n \quad \text{for any of} \quad \FH_n(\CC^2), \FH_n(\CC) \text{ or } \FH_n(\point)
$$ 
when we will make general statements that apply to all our flag Hilbert schemes. 

\begin{example}
\label{ex:hilb}

It is well-known that $\H_2$ is the blow-up of the diagonal inside $(\CC^2\times \CC^2)/S_2$. It should be no surprise that:
\begin{equation}
\label{eqn:c2}
\FH_2(\CC^2) = \text{Bl}_\Delta\left( \CC^2 \times \CC^2 \right) = \text{Proj} \left( \frac {\CC[x_1,x_2,y_1,y_2,z,w]}{(x_1-x_2)w - (y_1-y_2)z} \right)
\end{equation}
where the variables $x_i,y_i$ sit in degree 0, while $z,w$ sit in degree 1 with respect to the Proj. Setting $y_1=y_2=0$, respectively $x_1=x_2=y_1=y_2=0$, we obtain:
\begin{equation}
\label{eqn:c1}
\FH_2(\CC) = \PP^1 \times \AA^1 \cup \AA^1 \times \AA^1 = \text{Proj} \left( \frac {\CC[x_1,x_2,z,w]}{(x_1-x_2)w} \right)
\end{equation}
\begin{equation}
\label{eqn:c0}
\FH_2(\point) = \PP^1 = \text{Proj} \left( \CC[z,w]\right)
\end{equation}

\end{example}

\subsection{The matrix presentation}
\label{sub:adhm hilbert}

Throughout this section, we fix the Lie groups:
$$
G = GL_n, \qquad B_0 = \text{invertible lower triangular }n \times n\text{ matrices}
$$
and the flag variety $\fl = G/B_0$. We will also consider the Lie algebras:
$$
\fg = n\times n \text{ matrices}, \qquad \fb_0 = \text{lower triangular }n \times n\text{ matrices}
$$
We will also write $\fn_0 \subset \fb_0$ for the nilpotent subgroup of strictly lower triangular matrices, and $V$ for the $n$ dimensional vector space on which all the above matrix groups and algebras act. %We choose to employ the notation $B_0$ for a fixed Borel subgroup, in contrast to the case of a variable Borel subgroup $B$. 

\begin{proposition} 
\label{prop:adhm}

(ADHM construction, \cite{Nak}) \ The Hilbert scheme of $n$ points is given by:
\begin{equation}
\label{eqn:defhilbert}
\H_n = \mu^{-1}(0)^{\emph{cyc}}/G
\end{equation}
where the ``moment map" is given by:
\begin{equation}
\label{eqn:defmoment}
\mu: \fg \times \fg \times V \longrightarrow \fg, \qquad \mu(X,Y,v) = [X,Y]
\end{equation}
and the superscript $\cyc$ stands for the open subset of cyclic triples $(X,Y,v)$, i.e. those for which $V$ is generated by the vectors $\{X^aY^bv\}_{a,b\geq 0}$. Finally, the quotient by $G$ is explicitly given by:
$$
g\cdot (X,Y,v) = \left( g X g^{-1},g Y g^{-1},g v \right) \qquad \forall g \in G
$$
\end{proposition}

\begin{remark}

The reader accustomed to the construction of symplectic varieties via Hamiltonian reduction will recognize that two of the Lie algebras in \eqref{eqn:defmoment} are usually replaced with their duals. Here we tacitly assume the identification of $\fg$ with its dual given by the trace pairing.

\end{remark}

Passing between the ideal description of the Hilbert scheme and the ADHM picture is easy:
$$
I \leadsto \{ V = \CC[x,y]/I, \ X,Y  = \text{multiplication by }x,y, \text{ and } v = 1\text{ mod }I \}
$$
$$
(X,Y,v) \leadsto I = \{f\in \CC[x,y] \text{ such that } f(X,Y)\cdot v = 0\} 
$$
To mimic \eqref{eqn:defhilbert} for the flag Hilbert scheme, one needs to replace the vector space $V$ by a full flag of vector spaces. Then the maps $X,Y$ must preserve these vector spaces, and so are required to lie in the Borel subspace $\fb_0$. In other words, we have:
\begin{equation}
\label{eqn:def flag hilbert 0}
\FH_n(\CC^2) = \bar{\mu}^{-1}(0)^{\cyc}/B_0
\end{equation}
where:
$$
\bar{\mu}: \fb_0 \times \fb_0 \times V \longrightarrow \fn_0, \qquad \bar{\mu}(X,Y,v) = [X,Y]
$$
However, using \eqref{eqn:def flag hilbert 0} as the definition of flag Hilbert schemes leads us into trouble, since there is no general reason why quotients modulo Borel subgroups are good. To remedy this problem, let us consider the following alternative definition of flag Hilbert schemes, built on the observation that one can let the Borel subgroup vary.

\begin{definition}
\label{def:flag} 

Consider the following space, inspired by the Grothendieck resolution:
$$
\fz = \Big\{(X,Y,v,\fb) \in \fg \times \fg \times V \times \fl, \ X,Y \in \fb \Big\}
$$
where we identify the flag variety with the set of Borel subalgebras of $\fg$. Consider the map:
\begin{equation}
\label{eqn:def nu}
\nu : \fz \longrightarrow \Adj_\fn, \qquad (X,Y,v,\fb) \mapsto [X,Y] 
\end{equation}
where the target $\Adj_\fn$ is the affine bundle over the flag variety with fibers given by the nilpotent radicals $\fn$. It is $G$--equivariant with respect to the adjoint action, hence the notation. Define:
\begin{equation}
\label{eqn:def flag hilbert}
\FH_n(\CC^2) = \nu^{-1}(0)^\cyc/G
\end{equation}
where the $G$ action is:
$$
g\cdot (X,Y,v,\fb) = \left( gXg^{-1},gYg^{-1},gv,\text{Ad}_g(\fb) \right)  \qquad \forall g \in G
$$
and the superscript $\cyc$ still refers to the open subset of cyclic triples.

\end{definition}

While mostly a matter of presentation, the definition \eqref{eqn:def flag hilbert} has several advantages. Firstly, note that the map $\nu:\FH_n(\CC^2) \rightarrow \H_n$ is simply given by forgetting the flag $\fb$. Secondly, the set of quadruples $(X,Y,v,\fb)$ which are cyclic is precisely the set of stable points with respect to the action of $G$ on the trivial line bundle on $\fz$ (endowed with the determinant character). Then geometric invariant theory implies that \eqref{eqn:def flag hilbert} is a geometric quotient.

\subsection{DG schemes}
\label{sub:dg schemes}

Because the quotient in \eqref{eqn:defhilbert} is taken in the sense of GIT, the Hilbert scheme is a quasi-projective variety. But let us neglect its interesting structure as a topological space, and describe its ring of functions locally. By definition, the locus of cyclic triples $(\fg \times \fg \times V)^\cyc$ is an open subset of affine space, and the moment map \eqref{eqn:defmoment} gives rise to a section of the trivial $\fg$ bundle:
$$
\mu \in \Gamma \left( \CO_{(\fg \times \fg \times V)^\cyc} \otimes \fg \right)
$$
over $(\fg \times \fg \times V)^\cyc$. We may write down the Koszul complex corresponding to this section:
$$
\left( \wedge^\bullet \fg, \mu \right) := \left[ \CO_{(\fg \times \fg \times V)^\cyc} \otimes \wedge^{\dim G} \fg \stackrel{\mu}\longrightarrow ... \stackrel{\mu}\longrightarrow \CO_{(\fg \times \fg \times V)^\cyc} \otimes \fg \stackrel{\mu}\longrightarrow \CO_{(\fg \times \fg \times V)^\cyc} \right]
$$
Since the Hilbert scheme is smooth, this complex is exact except at the rightmost cohomology group, where it is isomorphic to $\CO_{\mu^{-1}(0)^\cyc}$. Moreover, since all the maps are $G$--equivariant, we may write locally:
$$
\CO_{\H_n} \stackrel{\qis}\cong \left( \wedge^\bullet \adj_\fg, \mu \right) = \left[ \wedge^{\dim G} \adj_\fg \stackrel{\mu}\longrightarrow ... \stackrel{\mu}\longrightarrow \CO_{(\fg \times \fg \times V)^\cyc/G} \right]
$$
where $\adj_\fg$ denotes the vector bundle on $(\fg \times \fg \times V)^\cyc/G$, obtained by descending the trivial vector bundle $\fg$ on $\fg \times \fg \times V$, endowed with the $G$--action by conjugation. One may write down the analogous Koszul complex for the map $\nu$ of \eqref{eqn:def nu}, but observe that:
\begin{equation}
\label{eqn:koszul flag}
\CO_{\FH_n(\CC^2)} \textbf{ is not } \stackrel{\qis}\cong \left( \wedge^\bullet \adj_\fn, \nu \right) := \left[ \wedge^{\dim N} \adj_\fn \stackrel{\nu}\longrightarrow ... \stackrel{\nu}\longrightarrow \CO_{(\fg \times \fg \times V \times \fl)^\cyc/G} \right]
\end{equation}
(recall that $\adj_\fn$ denotes the vector bundle on $(\fg \times \fg \times V \times \fl)^\cyc/G$, obtained by descending the vector bundle $\Adj_\fn$ on $\fl$, endowed with the $G$--action by conjugation). The fact that the Koszul complex \eqref{eqn:koszul flag} is not exact anymore boils down to the fact that $\FH_n(\CC^2)$ is not a local complete intersection, and so we choose to work instead with the dg scheme:
\begin{equation}
\label{eqn:def dg zero}
\CO_{\FH_n^\dg(\CC^2)} := \left( \wedge^\bullet \adj_\fn, \nu \right)
\end{equation}
Note that we think of the left hand side as a sheaf of dg algebras, given precisely by the complex in \eqref{eqn:koszul flag} supported on the smooth scheme $(\fg \times \fg \times V \times \fl)^\cyc/G$, which is nothing but a flag variety bundle over the smooth scheme $(\fg \times \fg \times V)^\cyc/G$. This will allow us to ignore the subtleties of the topology of dg schemes.

\subsection{Explicit matrices}
\label{sub:explicit matrices}

Although the definition of $\fz$ and $\FH_n(\CC^2)$ is given by allowing the Borel subgroup to vary, to keep the presentation explicit we will henceforth fix it to be $B=B_0$. Therefore, points of the flag Hilbert scheme will be triples:
\begin{equation}
\label{eqn:triple matrices}
X = \left( \begin{array}{cccc}
x_1 & 0 & 0 & 0 \\
* & x_2 & 0 & 0 \\
* & * & ... & 0 \\
* & * & * & x_{n} \end{array} \right), \quad Y = \left( \begin{array}{cccc}
y_1 & 0 & 0 & 0 \\
* & y_2 & 0 & 0 \\
* & * & ... & 0 \\
* & * & * & y_{n} \end{array} \right), \quad v = \left( \begin{array}{cccc}
1 \\
0 \\
0 \\
0 \end{array} \right)  
\end{equation}
such that $[X,Y]=0$, and the vectors $\{X^aY^bv\}_{a,b\geq 0}$ generate the space $V$. This latter condition implies that the first entry of $v$ must be non-zero, so we may use the $B=B_0$ action to fix $v$ as in equation \eqref{eqn:triple matrices}. Therefore, we will abuse notation and re-write \eqref{eqn:def flag hilbert} as:
\begin{equation}
\label{eqn:defflaghilbert1}
\FH_n(\CC^2) = \Big\{(X,Y,v), \ X,Y \text{ lower triangular}, \ [X,Y]=0, \ v \text{ cyclic} \Big\}/B 
\end{equation}
In this language, the map:
$$
\FH_n(\CC^2) \stackrel{\rho}\longrightarrow \CC^{2n}
$$
is given by taking the joint eigenvalues of the matrices $X$ and $Y$. Therefore, we conclude that:
\begin{equation}
\label{eqn:defflaghilbert2}
\FH_n(\CC) = \Big\{(X,Y,v) \text{ as in \eqref{eqn:defflaghilbert1}}, \ \ \ Y \text{ strictly lower triangular} \Big\}
\end{equation}
\begin{equation}
\label{eqn:defflaghilbert3}
\FH_n(\point) = \Big\{(X,Y,v) \text{ as in \eqref{eqn:defflaghilbert1}}, \ X , Y \text{ strictly lower triangular} \Big\}
\end{equation}
We may use the descriptions \eqref{eqn:defflaghilbert1}--\eqref{eqn:defflaghilbert3} to obtain the following estimates of the dimensions of flag Hilbert schemes:
$$
\dim \FH_n(\CC^2) \geq \dim \left(\text{affine space of }(X,Y,v)\right) - \#\left(\text{equations }[X,Y]=0\right) - \dim B = 
$$
\begin{equation}
\label{eqn:expdim1}
= n^2+2n - \frac {n(n-1)}2 - \frac {n(n+1)}2 = 2n \ =: \ \expdim \FH_n(\CC^2)
\end{equation}
The right hand side stands for ``expected (or virtual) dimension". Similarly, we have:
\begin{equation}
\label{eqn:expdim2}
\dim \FH_n(\CC) \geq n \ =: \ \expdim \FH_n(\CC)
\end{equation}
\begin{equation}
\label{eqn:expdim3}
\dim \FH_n(\point) \geq n-1 \ =: \ \expdim \FH_n(\point)
\end{equation}
The reason why the expected dimension in \eqref{eqn:expdim3} is $n-1$ rather than 0 is that when $X$ and $Y$ are both strictly lower triangular matrices, the commutator $[X,Y]=0$ is not only strictly lower triangular, but has the first sub-diagonal equal to zero by default. Therefore, the first sub-diagonal entries are $n-1$ equations that need not be placed on $\FH_n(\point)$. 

\begin{example}
\label{ex:not lci}

If the inequalities in \eqref{eqn:expdim1}--\eqref{eqn:expdim3} were equalities, then we would conclude that flag Hilbert schemes were local complete intersections. However, this is not the case. We  give an example of how the bound in \eqref{eqn:expdim3} can fail, which we learned from Ian Grojnowski.  Let $n=10$, and  consider the affine space of matrices $X,Y$ which are lower triangular, and have zero blocks of sizes $1,2,3$ and $4$ on the diagonal:
\begin{equation}
\label{eqn:big matrix}
X,Y = \left( \begin{array}{cccccccccc}
0 & 0 & 0 & 0 & 0 & 0 & 0 & 0 & 0 & 0 \\
* & 0 & 0 & 0 & 0 & 0 & 0 & 0 & 0 & 0 \\
* & 0 & 0 & 0 & 0 & 0 & 0 & 0 & 0 & 0 \\
* & * & * & 0 & 0 & 0 & 0 & 0 & 0 & 0 \\
* & * & * & 0 & 0 & 0 & 0 & 0 & 0 & 0 \\
* & * & * & 0 & 0 & 0 & 0 & 0 & 0 & 0 \\
* & * & * & * & * & * & 0 & 0 & 0 & 0 \\
* & * & * & * & * & * & 0 & 0 & 0 & 0 \\
* & * & * & * & * & * & 0 & 0 & 0 & 0 \\
* & * & * & * & * & * & 0 & 0 & 0 & 0 \\
\end{array} \right)
\end{equation}
The dimension of the affine space consisting of triples $(X,Y,v)$ equals $35+35+10 = 80$. Since the commutator $[X,Y]=0$ must have the $2\times 1$, $3\times 2$ and $4\times 3$ blocks under the diagonal equal to zero by default, the number of equations we need to impose is only $15$. Taking into account the fact that the Borel subgroup has dimension $55$, we conclude that:
$$
\dim \FH_{10}(\point) \geq 80 - 15 - 55 = 10 > 9 = \expdim \FH_{10}(\point)
$$
We may translate this example in terms of flags of ideals inside $\CC[x,y]$. Let $d=4$, $n = \binom{d+1}{2}$, and $\mathfrak{m} \subset \CC[x,y]$ be the maximal ideal of the origin, and let us consider the locus of flags:
$$
L = (I_0\supset I_1\supset \ldots \supset I_{n}) \subset \FH_{n}(\point)
$$
such that:
\begin{equation}
\label{eqn:weird ideal}
I_{\binom{k+1}{2}}=\mathfrak{m}^{k},\quad k = 0,\ldots,d.
\end{equation}
By the defining property of the maximal ideal $\fm$, for each $k \in \{0,...,d-1\}$ the flag of ideals: 
$$
\mathfrak{m}^{k}\supset I_{\binom{k+1}{2}+1}\supset \ldots I_{\binom{k+2}{2}-1} \supset \mathfrak{m}^{k+1}
$$
can be chosen as an arbitrary complete flag of vector subspaces in $\mathfrak{m}^{k}/\mathfrak{m}^{k+1}\simeq \CC^{k+1}$. Since the dimension of the corresponding flag variety is $\binom{k+1}{2}$, we conclude that:
$$
\dim L = \sum_{k=0}^{d-1}\binom{k+1}{2}=\binom{d+1}{3} \gg n-1 = \expdim \FH_n(\point) 
$$
as $d$ becomes large (although the inequality is strict as soon as $d\geq 4$). This construction also shows that the stratum $L$ is non-empty, since there always exist flags of ideals with the property \eqref{eqn:weird ideal}, something which was not immediately apparent from the matrix construction \eqref{eqn:big matrix}.

\end{example}

\subsection{Projective tower construction}
\label{sub:tower}

%Example \ref{ex:not lci} shows that the schemes $\FH_n = \FH_n(*)$ for $* \in \{\CC^2,\CC,\point\}$ are rather badly behaved. Instead, we will make sense of the dimension count in \eqref{eqn:expdim1}--\eqref{eqn:expdim3} in a virtual sense. Specifically, we will define a dg scheme $\FH_n^\dg$ supported on $\FH_n$. 

Let us consider the action:
\begin{equation}
\CC^* \times \CC^* \curvearrowright \FH_n
\end{equation}
which scales the matrices $X,Y$ independently. We denote the basic characters of this action by $q$ and $t$, so the $\torus$ action is explicitly given by:
$$
(z_1,z_2) \cdot (X,Y) = (q(z_1) X ,t(z_2) Y), \qquad \forall \ (z_1,z_2) \in \CC^* \times \CC^*
$$
In the matrix presentation, the tautological bundle $\CT_n$ on $\FH_n$ has fibers consisting simply of the vector spaces $V$ on which the matrices $X,Y$ act. The fact that flag Hilbert schemes are defined as $B$--quotients means that this vector bundle need not be trivial. Therefore, the matrices $X,Y: V \rightarrow V$ give rise to endomorphisms of the tautological bundle on the whole of $\FH_n$, which we will denote by the same letters:
$$
q\CT_n \stackrel{X}\longrightarrow \CT_n, \qquad t\CT_n \stackrel{Y}\longrightarrow \CT_n
$$
In the formulas above, one must twist the tautological bundle by the torus characters $q,t$ in order for the endomorphisms $X,Y$ to be $\CC^* \times \CC^*$ equivariant. Since a point of the flag Hilbert scheme entails the choice of a fixed flag of $V$, there is a full flag of tautological vector bundles:
$$
\CT_n \twoheadrightarrow \CT_{n-1} \twoheadrightarrow ... \twoheadrightarrow \CT_1
$$
on $\FH_n$. Flag Hilbert schemes are easier to work with than usual Hilbert schemes because they can be built inductively. Specifically, we have the maps:
\begin{equation}
\label{eqn:tower}
\includegraphics{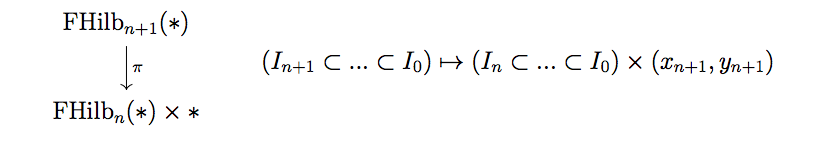}
\end{equation}
for any $\star$. When $* = \CC$ we set $y_{n+1}=0$ and when $* = \point$ we further set $x_{n+1} = y_{n+1} = 0$. What makes \eqref{eqn:tower} manageable is that it is a \textbf{projective bundle}, so we conclude that flag Hilbert schemes are projective towers. Specifically, consider the complexes:
\begin{equation}
\label{eqn:complex}
\includegraphics{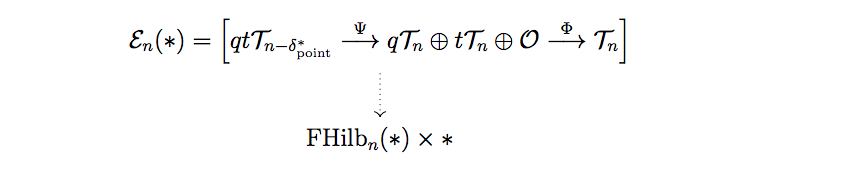}
\end{equation}
for any $\star$, with the maps defined by:
\begin{equation}
\label{eqn:psi}
\Psi(w) = \Big( - (Y-y_{n+1})w, (X-x_{n+1})w,0 \Big) %\left( \begin{array}{c} -(Y-y_{n+1})w \\ (X-x_{n+1})w \\ 0 \end{array} \right)
\end{equation}
\begin{equation}
\label{eqn:phi}
\Phi(w_1,w_2,f) = (X-x_{n+1})w_1 + (Y-y_{n+1})w_2 + fv
\end{equation}
Here, $x_{n+1},y_{n+1}$ are the coordinates on the second factor of $\FH_n(\CC^2) \times \CC^2$, which are specialized to $y_{n+1} = 0$ (resp. $x_{n+1} = y_{n+1} = 0$) when $* = \CC$ (resp. $* = \point$). When $* = \point$, the leftmost bundle in the complex \eqref{eqn:complex} is $\CT_{n-1}$. This implicitly uses the fact that the maps $X,Y:\CT_n \rightarrow \CT_n$ become nilpotent, hence they factor through $\CT_{n} \twoheadrightarrow \CT_{n-1}$. In the next Subsection, we will prove the following inductive description of flag Hilbert schemes (\cite{Negut}):

\begin{theorem}
\label{thm:complex}
The maps $\pi$ of \eqref{eqn:tower} can be written as projectivizations:
\begin{equation}
\label{eqn:towers}
\FH_{n+1} = \PP \left( H^0(\CE_n)^{\vee} \right) := \proj_{\FH_n}\left(S^{\bullet} \left( H^0(\CE_n) \right) \right)
\end{equation}
This holds for each of the three variants $\star$ of flag Hilbert schemes. The line bundle $\CL_{n+1}$ on the left hand side coincides with the tautological sheaf $\CO(1)$ on the right. 
\end{theorem}

\begin{example}
\label{n=1}

Example \ref{ex:hilb} shows that the space $\FH_2$ can be obtained as $\proj$ of an explicit algebra. Let us obtain the same result using Theorem \ref{thm:complex}. Since $\CT_1=\CO$, we have:
$$
\CE_1(\CC^2) = \left[qt \CO\xrightarrow{(-y_1+y_2,x_1-x_2,0)} q\CO \oplus t\CO\oplus\CO\xrightarrow{(x_1-x_2,y_1-y_2,1)}\CO\right]\simeq
$$
$$
\simeq \left[qt \CO\xrightarrow{(-y_1+y_2,x_1-x_2)}q\CO\oplus t\CO\right] \quad \Rightarrow \quad S^\bullet \left(\CE_1 (\CC^2)\right) = \frac{\CC[x_1,x_2,y_1,y_2,z,w]}{(x_1-x_2)w-(y_1-y_2)z}
$$
precisely as in \eqref{eqn:c2}. Here, $z$ and $w$ are the two basis vectors of $q \CO \oplus t \CO$. If we set $y_1 = y_2 = 0$ in the above computation, we obtain the case $* = \CC$ of \eqref{eqn:c1}. Finally, we have:
$$
\CE_1(\point) =\left[q\CO\oplus t\CO\oplus\CO\xrightarrow{(0,0,1)}\CO\right]\simeq \left[q \CO\oplus t \CO\right] \quad \Rightarrow \quad S^\bullet \left( \CE_1(\point) \right) = \CC[z,w]
$$
as expected from \eqref{eqn:c0}.
\end{example}

\begin{example}
\label{ex:n=2}
Let us study Theorem \ref{thm:complex} in the case when $n=2$ and $* = \point$, in which case:
$$
\FH_2(\point) = \PP^1
$$
with respect to which we have $\CT_1 = \CO$ and $\CT_2 = \CO \oplus \CO(1)$. With this in mind, the complex \eqref{eqn:complex} is explicitly given by:
$$
\CE_2(\point) = \Big[ qt\CO \stackrel{\Psi}\longrightarrow q\CO \oplus t \CO \oplus \CO \oplus q \CO(1) \oplus t\CO(1) \stackrel{\Phi}\longrightarrow \CO \oplus \CO(1) \Big]
$$
and the maps are given by:
$$
\Psi = \left( \begin{array}{c}
0 \\ 0 \\ 0 \\ -z_1 \\ z_0 \end{array} \right), \qquad \Phi = \left( \begin{array}{ccccc}
0 & 0 & 1 & 0 & 0 \\
z_0 & z_1 & 0 & 0 & 0 \end{array} \right)
$$
It is clear from the above that the map $\Phi$ is surjective, which is a general phenomenon that follows from the cyclicity of triples $(X,Y,v)$. Therefore, we have:
$$
\CE_2(\point) \stackrel{\text{q.i.s.}}\cong \Big[ qt\CO \xrightarrow{(0,-z_1,z_0)} qt \CO(-1) \oplus  q \CO(1) \oplus t\CO(1) \Big] \stackrel{\text{q.i.s.}}\cong qt\CO(-1) \oplus \CO(2)
$$
Therefore, Theorem \ref{thm:complex} implies that:
\begin{equation}
\label{eqn:pt 3}
\FH_3(\point) = \PP_{\PP^1} \left( \frac {\CO(1)}{qt} \oplus \CO(-2) \right)
\end{equation}
which is a Hirzebruch surface. It is also the resolution of the singular cubic cone, which is nothing but the subvariety of the Hilbert scheme consisting of ideals supported at the origin.

\end{example}

\subsection{Proving Theorem \ref{thm:complex}}
\label{sub:proofcomplex}

Without loss of generality, we will treat the case $* = \CC^2$. We will proceed by induction by $n$, by studying the fibers of the map \eqref{eqn:tower}:
\begin{equation}
\label{eqn:mm}
\includegraphics{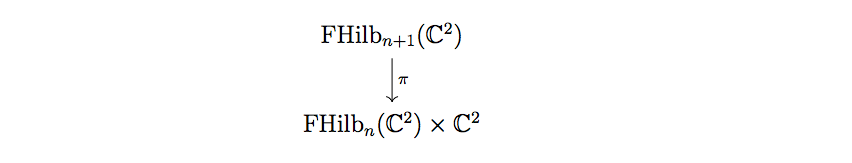}
\end{equation}
Recall that points of $\FH_n(\CC^2)$ are triples $(X,Y,v)$ consisting of two commuting lower triangular matrices (for simplicity, we fix the flag of vector spaces), together with a cyclic vector. Over such a triple, fibers of $\pi$ are completely determined by extending $X,Y,v$ by a bottom row:
$$
\bar{X} = \left( \begin{array}{cc}
X & 0 \\
w_1 & x_{n+1} \end{array} \right), \qquad \bar{Y} = \left( \begin{array}{cc}
Y & 0 \\
w_2 & y_{n+1} \end{array} \right), \qquad \bar{v} = \left( \begin{array}{cc}
v \\
f \end{array} \right)
$$
where $w_1, w_2 \in \CT_n^\vee$ and $f \in \CO$. The triple $(w_1, w_2,f)$ must satisfy the following properties: 

\begin{itemize} 

\item The closed condition $[\bar{X}, \bar{Y}] = 0$ is equivalent to:
\begin{equation}
\label{eqn:enigma1}
w_1 \cdot (Y - y_{n+1}) = w_2 \cdot (X - x_{n+1})
\end{equation}

\item $(w_1, w_2,f)$ is only defined up to conjugation by:
$$\
V \rtimes \CC^* = \text{Ker}(B_{n+1} \twoheadrightarrow B_n) = \left( \begin{array}{cc}
\text{Id} & 0 \\
w & c \end{array} \right) \qquad (w,c) \in V \rtimes \CC^*
$$
In other words, we do not consider the action of the group of $n \times n$ lower triangular matrices $B_n$ because it has already been trivialized locally on $\FH_n(\CC^2)$. In formulas:
\begin{equation}
\label{eqn:enigma2}
(w_1, w_2,f) \sim \left( cw_1 + w \cdot (X-x_{n+1}),cw_2 + w \cdot (Y-y_{n+1}),cf + w \cdot v \right) 
\end{equation}

\item Since we already know that $(X,Y,v)$ is cyclic, the extra condition that $(\bar{X}, \bar{Y}, \bar{v})$ be cyclic is equivalent to the fact that:
\begin{equation}
\label{eqn:equivalent}
\CC^{n+1} \text{ is generated by } \Big\{ \bar{v},\text{Im }(\bar{X}-x_{n+1}), \text{Im }(\bar{Y}-y_{n+1}) \Big\}
\end{equation}
This fails precisely when there exists a linear functional $\lambda : \CC^n \rightarrow \CC$ such that:
$$
\lambda(v) = f, \quad \lambda \left((X-x_{n+1})w \right) = w_1 \cdot w, \quad \lambda \left((Y-y_{n+1})w \right) = w_2 \cdot w
$$
for all $w \in V$. This is equivalent to $(w_1,w_2,f) \sim (0,0,0)$ with respect to \eqref{eqn:enigma2}. 

\end{itemize}

\begin{proof} \emph{of Theorem \ref{thm:complex}:}  The three bullets above establish the fact that the triple $(w_1,w_2,f)$ that determines points in the fibers of $\FH_{n+1}(\CC^2) \rightarrow \FH_n(\CC^2) \times \CC^2$ is a non-zero element in:
\begin{equation}
\label{eqn:dualcomplex}
H^0\left( \frac {\CT^\vee_{n}}{qt} \stackrel{\Psi^\vee}\longleftarrow \frac {\CT^\vee_n}q \oplus \frac {\CT_n^\vee}t \oplus \CO \stackrel{\Phi^\vee}\longleftarrow \CT_n^\vee \right)
\end{equation}
modulo rescaling. Note that \eqref{eqn:dualcomplex} is the dual of \eqref{eqn:complex}, which completes the proof. 

\end{proof}

%\begin{remark}
%\label{rem:geometricquotient}

%Note that the explicit description of the fibers of $\pi$ as the projectivization of $\text{Ker }\Psi^*/ \text{Im }\Phi^*$ from \eqref{eqn:dualcomplex} allows us to show that the quotients \eqref{eqn:defflaghilbert1}--\eqref{eqn:defflaghilbert3} are GIT. Indeed, by the induction hypothesis, it is sufficient to prove that the quotient by:
%$$
%\CC^n \rtimes \CC^* = \text{Ker}(B_{n+1} \twoheadrightarrow B_n) 
%$$
%is GIT. According to the second bullet above, the fibers of the prequotient are given by the coherent sheaf:
%$$
%\text{Ker }\Psi^*
%$$ 
%Taking the quotient by $\CC^n$ amounts to dividing the above coherent sheaf by its vector sub-bundle $\text{Im }\Phi^*$, which is actually a geometric quotient. Finally, the quotient by $\CC^*$ is GIT according to the Proj construction of any coherent sheaf.

%\end{remark}

\begin{remark}
\label{rem:surjective}

Note that the map $\Phi$ of \eqref{eqn:complex} is surjective, according to the equivalent description \eqref{eqn:equivalent} of a point being cyclic. This implies that $\CE_n(*)$ is quasi-isomorphic to a complex:
\begin{equation}
\label{eqn:two step}
\CE_n(*) \stackrel{\qis}\cong \left[ qt \CT_{n-\delta_\point^*} \stackrel{\Psi}\longrightarrow \Ker \ \Phi \right]
\end{equation}
of vector bundles on $\FH_n(*) \times *$, which lie in degrees $-1$ and $0$.

\end{remark}

\subsection{The dg scheme}
\label{sub:dg scheme} 

We will now give an alternative definition of the dg scheme \eqref{eqn:def dg zero}, and we leave it as an exercise to the interested reader to show that the two descriptions are equivalent (we will only use the definition in this Subsection for the remainder of this paper). The idea is to note that the map $\Psi$ of the complex \eqref{eqn:two step} fails to be injective on many fibers, and this will lead to the flag Hilbert scheme misbehaving. To remedy this issue, we replace the middle cohomology sheaf $H^0 (\CE_n)$ in \eqref{eqn:complex} by the entire complex $\CE_n$ (we tacitly suppress the symbol $* \in \{\CC^2,\CC,\point\}$ since the construction applies equally well to all three choices). %To make this into a precise recursive construction of a dg scheme, we must emulate the construction in the previous subsection. 

\begin{proposition}
\label{prop:dg scheme}

There exist dg schemes $\FH_n^\edg$ endowed with flags of objects:
$$
\CT_n \rightarrow \CT_{n-1} \rightarrow ... \rightarrow \CT_1 \quad \in D^b(\coh(\FH_n^\edg))
$$
together with maps $q\CT_n \stackrel{X}\rightarrow \CT_n$, $t\CT_n \stackrel{Y}\rightarrow \CT_n$ that respect the above flag, and $\CO \stackrel{v}\rightarrow \CT_n$ such that:
\begin{equation}
\label{eqn:def dg}
\FH_{n+1}^\edg = \PP_{\FH^\edg_{n}} \left( \CE_n^\vee  \right) := \proj_{\FH^\edg_{n}} \left(S^\bullet\CE_n\right)
\end{equation}
where the complex $\CE_n$ is defined by formula \eqref{eqn:complex}. See Subsection \ref{sub:proj dg} for the definition of the Proj construction of a two-step complex of vector bundles (according to \eqref{eqn:two step}).

\end{proposition}

\begin{proof} Let us write $\CL_{n+1} = \CO(1)$ for the tautological line bundle on the projectivization \eqref{eqn:def dg}, and $\pi:\FH_{n+1} \rightarrow \FH_n$ for the natural map. Take the defining map of projective bundles:
$$
\text{Taut} \in \Hom(\pi^*\CE_n, \CO(1))
$$
and compose it with the natural map $\CT_n[-1] \stackrel{i}\rightarrow \CE_n$. We obtain an object:
$$
i_*(\text{Taut}) \ =: \ \CT_{n+1} \in \Hom \left(\pi^*\CT_n[-1],\CL_{n+1} \right) %\ = \ \Ext^1 \left(\pi^*\CT_n, \CL_{n+1} \right) 
$$
Composing the map $i$ with $q\CT_n \oplus t\CT_n \oplus \CO \xrightarrow{(X,Y,v)} \CT_n$ yields 0, hence:
$$
(X,Y,v)_* \left( \CT_{n+1} \right) \in \Hom \left(\pi^*( q\CT_n \oplus t\CT_n \oplus \CO )[-1], \CL_{n+1} \right) %\Ext^1 \left(\pi^*( q\CT_n \oplus t\CT_n \oplus \CO ), \CL_{n+1} \right)
$$
equals 0 as well. This precisely gives rise to a splitting:
$$
\includegraphics{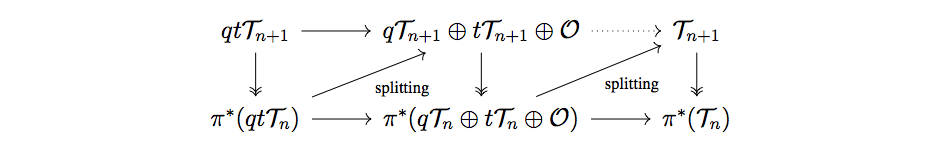}
$$
and the dotted map is the desired extension of the arrows $X,Y,v$ from $\CT_n$ to $\CT_{n+1}$. Note that we may write the above diagram as an equality in the derived category of $\FH_{n+1}^\dg$:
\begin{equation}
\label{eqn:ind}
\left[ qt \CL_{n+1} \xrightarrow{(-y,x)} \underline{q \CL_{n+1} \oplus t \CL_{n+1}} \xrightarrow{(x,y)} \CL_{n+1} \right] \cong \left[ \underline{\CE_{n+1}} \longrightarrow \widetilde{\pi^*(\CE_n)} \right]
\end{equation}
where we have underlined the $0$--th terms of both complexes. In the above equation, we write $x$ and $y$ for the operators of multiplication by $x_n-x_{n+1}$ and $y_n-y_{n+1}$, respectively, and:
\begin{equation}
\label{eqn:switch}
\widetilde{\pi^*(\CE_n)} \ \text{ denotes } \ \pi^*(\CE_n) \text{ with the variables }(x_n,x_{n+1}) \text{ and }(y_n,y_{n+1}) \text{ switched} 
\end{equation}

\end{proof}

\noindent One can run the proof of Proposition \ref{prop:dg scheme} with $\CE_n$ replaced by $H^0\CE_n$. We leave it as an exercise to the interested reader to show that one would obtain the schemes $\FH_n$ of Theorem \ref{thm:complex}.

\subsection{Serre duality}
\label{sub:serre}

As explained in Subsection \ref{sub:proj dg} of the Appendix, we may embed the dg scheme $\FH_{n+1}^\dg$ into an actual projective bundle:
\begin{equation}
\includegraphics{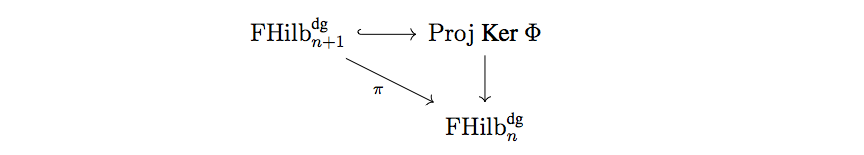}
\label{eqn:triangle}
\end{equation}
where we implicitly use the description \eqref{eqn:two step} of the complex of vector bundles $\CE_n$. This allows us to compute the push-forward $\pi_*$ of sheaves by factoring them through the diagram \eqref{eqn:triangle}. 

\begin{proposition}
\label{prop:serre}

Let $\pi:\FH^{\edg}_{n+1}(\CC) \to \FH_n^{\edg}(\CC)\times \CC$ be the projection. Then:
\begin{equation}
\label{eqn:serre}
\pi_*(\CA)^\vee \cong \pi_*(\CA^\vee\otimes \CL_{n+1}^{-1})
\end{equation}
for any $\CA\in D^b(\coh(\FH^{\edg}_{n+1}(\CC)))$. The functor $\pi_*$ is derived, and $\vee$ denotes Serre duality on the dg scheme $\FH_n^{\edg}(\CC)$, which is defined inductively by Proposition \ref{prop:serre}. 

\end{proposition}

This is a direct application of Proposition \ref{prop:canonical} in the Appendix, together with the fact that the determinant of the complex $\CE_n(\CC)$ of \eqref{eqn:complex} is trivial. Applying formula \eqref{eqn:serre} to $\CA = \CO$ gives us the following formulas for all $k \geq 0$:
\begin{equation}
\label{eqn:push negative}
\pi_{*}(\CL_{n+1}^{-1-k}) = \pi_*(\CL_{n+1}^{k})^{\vee}=S^{k}\CE_n^{\vee} \qquad \text{concentrated in degree 0}
\end{equation}

\begin{remark}

The analogue of \eqref{eqn:serre} when $\CC$ is replaced by $\CC^2$ holds exactly as stated. Meanwhile, when $\CC$ is replaced by point we must replace formula \eqref{eqn:serre} by the following equation:
\begin{equation}
\label{eqn:serre2}
\widetilde{\pi}_*(\CA)^\vee \cong \widetilde{\pi}_*\left(\CA^\vee\otimes \frac {qt \CL_n}{\CL_{n+1}^{2}} \right)[-1]
\end{equation}
where $\widetilde{\pi}:\FH_{n+1}^\dg \rightarrow \FH_n^\dg$ is the standard projection. 

\end{remark}

%\begin{equation}
%\label{eqn:push negative 2}
%\widetilde{\pi}_{*}(\CL_{n+1}^{-1-k}) = qt \CL_n^{-1} \otimes \pi_*(\CL_{n+1}^{k-1})^{\vee} = qt \CL_n^{-1} \otimes S^{k-1}\CE_n^{\vee} \quad \text{concentrated in degree 1}
%\end{equation}
%In particular, $\pi_*(\CL_{n+1}^{-1})=\CO$ and $\widetilde{\pi}_*(\CL_{n+1}^{-1}) = 0$.

%\begin{equation}
%\label{eqn:defcanonical}
%\CK_n = \CL_1^{-1} \otimes ... \otimes \CL_{n}^{-1}
%\end{equation}

%\begin{lemma}
%\label{lem: can bundle}
%Suppose that $\CE$ is a complex of sheaves on a variety $X$ such that both $X$ and $Y:=\PP(\CE^{\vee})$ are locally complete intersections.
%The relative canonical bundle on $Y$ is given by the equation:
%$$
%K_{Y|X}=\det (\CE\otimes \CO(-1)).
%$$
%\end{lemma}
 
%\begin{proof}
%The relative (virtual) tangent bundle to $Y$ is given by 
%$$
%T_{Y|X}=\Hom(\CO(-1),\CE^{\vee}\otimes \CO(-1))=\CE^{\vee}\otimes \CO(1)/\CO,
%$$
%so
%$$
%\det T_{Y|X}=\det (\CE^{\vee}\otimes \CO(1)),
%$$
%and
%$$
%K_{Y|X}=\det T^{\vee}_{Y|X}=\det (\CE\otimes \CO(-1)).
%$$
%\end{proof}

%We apply Lemma \ref{lem: can bundle} inductively to the projective tower \eqref{eqn:complex}.
%The relative canonical bundle of $\FH_{n+1}$ over $\FH_n\times \CC$ is isomorphic to:
%$$
%\det(\CE_n\otimes L_{n+1}^{-1})=L_{n+1}^{-1},
%$$
%so
%$$
%K_{\FH_n|\CC^n}\simeq \otimes_{i=1}^{n}L_i^{-1}=(\det \CT_n)^{-1}.
%$$

\section{The Hecke algebra and Soergel category}
\label{sec:Sbim}

\subsection{The Hecke algebra}

Recall that the Hecke algebra of type $A_n$ has $n-1$ generators:
$$
H_n = \CC(q) \langle \sigma_1,...,\sigma_{n-1} \rangle
$$
modulo relations:
\begin{equation}
\label{eqn:hecke1}
\left(\sigma_i-q^{\frac 12} \right)\left(\sigma_i+q^{-\frac 12} \right) = 0 \ \qquad \ \forall \ i\in \{1,\ldots,n-1\}
\end{equation}
\begin{equation}
\label{eqn:hecke2}
\sigma_i\sigma_{i+1}\sigma_i = \sigma_{i+1}\sigma_i\sigma_{i+1} \qquad \forall \ i\in \{1,\ldots,n-2\}
\end{equation}
\begin{equation}
\label{eqn:hecke3}
\sigma_i\sigma_j = \sigma_j\sigma_i \qquad \quad \ \ \forall \ |i-j|>1.
\end{equation}
The algebra $H_n$ is a $q$-deformation of the group algebra of the symmetric group $\CC[S_n]$.
The irreducible representations $V_{\lambda}$ of $H_n$ at generic parameter $q$ are labeled by partitions of $n$, or, equivalently, by Young diagrams of size $n$. The multiplicity of $V_{\lambda}$ in the regular representation is equal to its dimension, which is itself equal to the number of standard Young tableaux (henceforth abbreviated SYT) of shape $\lambda$. Therefore, the regular representation of $H_n$ splits into a direct sum of irreducible representations labeled by standard tableaux. For each such tableau $T$, let $P_T$ denote the projector onto the irreducible summand in $H_n$ labeled by $T$. By construction, these projectors have the following properties:
\begin{equation}
\label{proj}
P_T P_{T'} = \delta_{T'}^T P_T, \qquad \sum_{T}P_{T}=1.
\end{equation}
The projectors $P_T$ can be written very explicitly in terms of the generators $\sigma_i$, see \cite{AiMor,gyoja} for details. They satisfy the following \textbf{branching rule}:
\begin{equation}
\label{branching}
i(P_T)=\sum_{\sq}P_{T+\sq},
\end{equation}
where $i:H_n\to H_{n+1}$ is the natural inclusion and the summation in the right hand side is over all 
possible SYT obtained from $T$ by adding a single box labeled by $n+1$.

The renormalized \textbf{Markov trace} \(\chi: H_n \to \CC(a,q)\) satisfies  the relations:
\begin{equation}
\chi(\sigma \sigma ') = \chi(\sigma'  \sigma), \qquad
\chi(i(\sigma)) = \chi(\sigma) \cdot \frac{1-a}{q^{\frac 12}-q^{-\frac 12}}, \qquad
\chi(i(\sigma) \sigma_n) = \chi(\sigma).
%, \qquad \chi(i(\sigma) \sigma_n^{-1}) = \chi(\sigma) \cdot a
\end{equation}
There is a natural  pairing \(\langle  \cdot, \cdot \rangle:  H_n \times H_n \to \CC(a,q)\)
 given by \(\langle \sigma, \tau \rangle = \chi(\sigma \tau ^\dagger)\), where \( \sigma^\dagger\) is the ``Hermitian conjugate'' of \(\sigma\) (this is the \(\CC\)-antilinear map   determined by the relations \(q^\dagger = q^{-1}\),  \(\sigma_i^\dagger = \sigma _i^{-1}\), and \((\sigma \tau )^\dagger = \tau^\dagger \sigma ^\dagger\)). With respect to this pairing, the adjoint of the  inclusion \(i: H_n \to H_{n+1}\) is the \textbf{partial Markov trace}:
$$
\Tr : H_{n+1} \to H_n \otimes \CC[a].
$$
It follows easily from the definitions that for all \(\sigma \in H_n\), we have \(\chi(\sigma) = \Tr^n(\sigma).\)

 The Markov trace of a projector $P_T$ only depends on the underlying Young diagram $\lambda$ of the SYT $T$, and is equal to the $\lambda$-colored HOMFLY-PT polynomial of the unknot. Specifically, we have the following result: %Andrei: please make sure that I rephrased this paragraph accordingly

\begin{proposition}(e.g. \cite{Ai}) The Markov trace of $P_T$ equals:
$$
\Tr ^n(P_T)=\prod_{\sq\in \lambda}\frac{1-aq^{c(\sq)}}{1-q^{h(\sq)}},
$$
where $c(\sq)$ and $h(\sq)$ respectively denote the content and the hook length of a square $\sq$ in $\lambda$.
% \textcolor{red}{As defined above, the Markov trace only depends on q, so there should be no a in the above formula. Maybe we want to say that $Tr \circ ... \circ Tr$ is a specialization of Jones-Ocneanu at $a=...$ or something. In any event, something should be changed.}
\end{proposition}

\subsection{The braid group}

The Hecke algebra is a quotient of the braid group on $n$ strands, which is defined by removing relation \eqref{eqn:hecke1}. Specifically, the braid group is generated by $\sigma_1^{\pm 1},...,\sigma_{n-1}^{\pm 1}$ modulo relations \eqref{eqn:hecke2} and \eqref{eqn:hecke3}. By definition, the \textbf{full twist} on $n$ strands is the braid:
$$
\FT_n = (\sigma_1 \cdots \sigma_{n-1})^{n}.
$$
\begin{figure}[ht!]
\begin{tikzpicture}
\begin{scope}[yscale=-1];
\draw (0,0)..controls (0,0.5) and (0.5,0.5)..(0.5,1);
\draw (0.5,0)..controls (0.5,0.5) and (1,0.5)..(1,1);
\draw (1,0)..controls (1,0.5) and (1.5,0.5)..(1.5,1);
\draw [line width=5,white] (1.5,0)..controls (1.5,0.5) and (0,0.5)..(0,1);
\draw  (1.5,0)..controls (1.5,0.5) and (0,0.5)..(0,1);

\draw (0,1)..controls (0,1.5) and (0.5,1.5)..(0.5,2);
\draw (0.5,1)..controls (0.5,1.5) and (1,1.5)..(1,2);
\draw (1,1)..controls (1,1.5) and (1.5,1.5)..(1.5,2);
\draw [line width=5,white] (1.5,1)..controls (1.5,1.5) and (0,1.5)..(0,2);
\draw  (1.5,1)..controls (1.5,1.5) and (0,1.5)..(0,2);

\draw (0,2)..controls (0,2.5) and (0.5,2.5)..(0.5,3);
\draw (0.5,2)..controls (0.5,2.5) and (1,2.5)..(1,3);
\draw (1,2)..controls (1,2.5) and (1.5,2.5)..(1.5,3);
\draw [line width=5,white] (1.5,2)..controls (1.5,2.5) and (0,2.5)..(0,3);
\draw  (1.5,2)..controls (1.5,2.5) and (0,2.5)..(0,3);

\draw (0,3)..controls (0,3.5) and (0.5,3.5)..(0.5,4);
\draw (0.5,3)..controls (0.5,3.5) and (1,3.5)..(1,4);
\draw (1,3)..controls (1,3.5) and (1.5,3.5)..(1.5,4);
\draw [line width=5,white] (1.5,3)..controls (1.5,3.5) and (0,3.5)..(0,4);
\draw  (1.5,3)..controls (1.5,3.5) and (0,3.5)..(0,4);
\end{scope}
\end{tikzpicture}
\caption{The full twist $\FT_4$}
\label{fig:ft}
\end{figure}
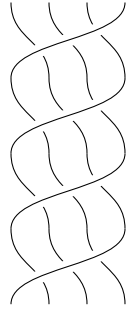
\noindent The full twist is known to be central in the braid group, and hence its image is central in the Hecke algebra. If we interpret the generator $\sigma_i$ as a single crossing between the strands $i$ and $i+1$, then the full twist corresponds to the pure braid where each strand wraps around all the other ones (see Figure \ref{fig:ft}). We may also define the partial twists:
$$
\FT_1,...,\FT_{n-1}
$$ 
where $\FT_k$ is the braid which consists of the full twist on the leftmost $k$ strands, with the rightmost $n-k$ strands simply vertical lines. We will also work with the generalized Jucys-Murphy elements (the name is due to the fact that their images in $H_n$ deform the well-known Jucys-Murphy elements in $\CC[S_n]$):
$$
L_k = \FT_{k-1}^{-1} \cdot \FT_k
$$
which are easily seen to be given by the formula:
$$
L_k = \sigma_{k-1}... \sigma_2 \sigma_1 \sigma_2 ... \sigma_{k-1} .
$$
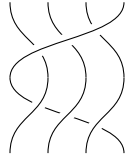
\begin{figure}[ht!]
\begin{tikzpicture}
\begin{scope}[yscale=-1];
\draw (0,0)..controls (0,0.5) and (0.5,0.5)..(0.5,1);
\draw (0.5,0)..controls (0.5,0.5) and (1,0.5)..(1,1);
\draw (1,0)..controls (1,0.5) and (1.5,0.5)..(1.5,1);
\draw [line width=5,white] (1.5,0)..controls (1.5,0.5) and (0,0.5)..(0,1);
\draw  (1.5,0)..controls (1.5,0.5) and (0,0.5)..(0,1);

\draw  (1.5,2)..controls (1.5,1.5) and (0,1.5)..(0,1);
\draw [line width=5,white] (0,2)..controls (0,1.5) and (0.5,1.5)..(0.5,1);
\draw [line width=5,white] (0.5,2)..controls (0.5,1.5) and (1,1.5)..(1,1);
\draw [line width=5,white] (1,2)..controls (1,1.5) and (1.5,1.5)..(1.5,1);
\draw (0,2)..controls (0,1.5) and (0.5,1.5)..(0.5,1);
\draw (0.5,2)..controls (0.5,1.5) and (1,1.5)..(1,1);
\draw (1,2)..controls (1,1.5) and (1.5,1.5)..(1.5,1);
\end{scope}
\end{tikzpicture}
\label{fig: murphy}
\caption{The braid $L_4$}
\end{figure}
Either the braids $\{\FT_k\}_{k=1,\ldots, n}$  or the braids $\{L_k\}_{k = 1, \ldots, n}$ generate a certain commutative subalgebra of the braid group, and hence also of the Hecke algebra, which we will denote by:
$$
C_n \subset H_n.
$$
It is well-known that the projectors $P_T$ lie in this subalgebra for all SYTx $T$. 

\begin{proposition}(e.g. \cite[Theorem 5.5]{AiMor}) The projectors are eigenvectors for twists with the following eigenvalues: 
\begin{equation}
\label{eigen decategorified}
\FT_k\cdot P_T =q^{c(\sq_1)+...+c(\sq_k)}  \cdot P_T \Longrightarrow
\end{equation}
\begin{equation}
\label{eigen decategorified 2}
\Longrightarrow L_k\cdot P_T = q^{c(\sq_k)}\cdot P_T
\end{equation}
where $\sq_k$ denotes the box labeled by $k$ in the standard Young tableau $T$. 
\end{proposition}

In fact, equations \eqref{branching} and \eqref{eigen decategorified} allow one to inductively construct the elements $P_T$, as follows: given $P_T$ for a standard Young tableau $T$ of size $n$, all projectors $P_{T+\sq}$ are eigenvectors for the full twist $\FT_{n+1}$ with different eigenvalues, and hence can be uniquely reconstructed as the projections of $i(P_T)$ onto the corresponding eigenspaces. This is precisely the viewpoint that is categorified in \cite{EH}, and which inspired Section \ref{sec:equiv} of the present paper.

\subsection{Soergel bimodules}
\label{sub:soergel}

The category of Soergel bimodules, which we will denote $\SBim_n$, is a categorification of the Hecke algebra. We will consider $R = \CC[x_1,...,x_n]$ and study graded $R-$bimodules, where $\deg x_i = 1$. We will write $q M$ for the graded module $M$ with the grading shifted by 1. Among the most important such $R-$bimodules are the elementary \textbf{Bott-Samelson} bimodules:
\begin{equation}
\label{eqn:bs}
B_i = q^{-\frac 12} R \otimes_{R^{i,i+1}} R
\end{equation}
for any simple transposition $s_i = (i,i+1)$, where we write $R^{i,i+1}$ for those polynomials which are invariant under $s_i$.  In other words, $R^{i,i+1}$ consists of polynomials which are symmetric in $x_i$ and $x_{i+1}$, and therefore $R$ has rank 2 over $R^{i,i+1}$. Therefore, $B_i$ has rank 2 as an $R-$module. 

\begin{definition}
\label{def:soergel}
The category $\SBim_n$ is the Karoubian envelope of the smallest full subcategory of $R\text{--mod--}R$ that contains the Bott-Samuelson modules $B_i$ and is closed under $\otimes_R$ and grading shifts. Objects of $\SBim_n$ will be called \textbf{Soergel bimodules}.
\end{definition}

The category $\SBim_n$ is monoidal with respect to the operation of tensoring bimodules over $R$. Clearly, the unit object is $\1 := R$, viewed as a bimodule over itself. Note that $\SBim_n$ is neither abelian, nor symmetric. Let:
$$
B_{i,i+1} = q^{-1} R \otimes_{R^{i,i+1,i+2}} R
$$
where  $R^{i,i+1,i+2}$ denotes the set of polynomials which are symmetric in $x_i,x_{i+1},x_{i+2}$. Then one can check the following identities \cite{Kh,Soergel}:
\begin{equation}
\label{eqn:BS1}
B_i^2\simeq q^{\frac 12} B_i\oplus q^{-\frac 12} B_i, \qquad B_iB_j\simeq B_jB_j\ \text{for}\ |i-j|>1,
\end{equation}
\begin{equation}
\label{eqn:BS2}
B_iB_{i+1}B_{i}\simeq B_i\oplus B_{i,i+1}\ \Rightarrow\ B_iB_{i+1}B_{i}\oplus B_{i+1}\simeq B_{i+1}B_{i}B_{i+1}\oplus B_{i}.
\end{equation}
It was shown in \cite{Soergel} that the split graded Grothendieck group of $\SBim_n$ is generated by the classes of $B_i$ and is isomorphic to $H_n$. Indeed, one can identify $[B_i]=\sigma_i+q^{-\frac 12}$ and show that \eqref{eqn:BS1}--\eqref{eqn:BS2} imply \eqref{eqn:hecke1}--\eqref{eqn:hecke3}.

\subsection{From Rouquier complexes to Khovanov-Rozansky homology}
\label{sub:khr}

Since  $\sigma_i=[B_i]-q^{-\frac 12}$, it is clear that $\sigma_i$ does not correspond to any Soergel bimodule. However, Rouquier showed that $\sigma_i$ can be realized in the homotopy category of complexes:
$$
K^b(\SBim_n)
$$
where we use the variable $s$ to keep track of homological degree. Explicitly, objects in the homotopy category of complexes will be denoted by:
$$
\Big[s^k M_k \rightarrow ... \rightarrow s^{k'} M_{k'} \Big]
$$
for some $k \leq k' \in \ZZ$. The variable $s$ may seem redundant when writing down chain complexes, but we keep track of it for two reasons: first of all, it will give rise to the equivariant parameter $t$ of Section \ref{sec:flag} via \eqref{eqn:matching}. Second of all, we think of the object: 
$$
\left[M \rightarrow s M' \right] \in K^b(\SBim_n)
$$
as the cone of a morphism between the objects $M$ and $s M'$, and thus the power of $s$ makes the homological degrees of our formulas manifest. Recall the Bott-Samuelson bimodules \eqref{eqn:bs} and consider the Rouquier complexes:
\begin{equation}
\label{eqn:rouquier}
\sigma_i:= \left[ B_i \xrightarrow{1\otimes 1 \mapsto 1} \frac {sR}{q^{\frac 12}}  \right], \qquad \sigma_i^{-1} := \left[ \frac {q^{\frac 12}R}s \xrightarrow{1 \mapsto x_i \otimes 1 - 1 \otimes x_{i+1}} B_i \right]
\end{equation}
They satisfy the following equations \cite{Kh,Rou} (which can be deduced from \eqref{eqn:BS1} and \eqref{eqn:BS2}):
$$
\sigma_i\otimes \sigma_i^{-1} \cong \sigma_i^{-1} \otimes \sigma_i \cong \1,
$$
$$
\sigma_i \otimes \sigma_j \cong \sigma_j \otimes \sigma_i \ \ \text{for }\ |i-j|>1,
$$
$$
\sigma_i \otimes \sigma_{i+1} \otimes \sigma_i \cong \sigma_{i+1} \otimes \sigma_i \otimes \sigma_{i+1},
$$
and hence categorify the braid group. To any braid $\sigma=\prod_{i=1}^{n-1} \sigma_{i}^{a_i}$ (where $\alpha_{i} \in \{-1,0,1\}$) one can associate a complex of bimodules obtained by tensoring together the various complexes \eqref{eqn:rouquier}. We abuse notation and denote the resulting complex also by $\sigma$. Khovanov \cite{Kh} defined the HOMFLY-PT homology of a braid $\sigma$ as: 
\begin{equation}
\label{eqn:Khovanov}
\HHH(\sigma):=\RHom_{K^b(\SBim_n)}(\1,\sigma).
\end{equation}
The right hand side is a triply graded vector space, endowed with the internal grading $q$, the homological grading $s$ of the complexes \eqref{eqn:rouquier} and their coproducts, and the Hochschild grading $a$ given by taking the $\RHom$. The appropriate derived category formalism can be found in \cite{Hog}. With respect to these three gradings, Khovanov proved that \eqref{eqn:Khovanov} is a topological invariant of the closure of $\sigma$, after a certain renormalization.

%\subsection{Certain useful properties of Rouquier complexes} 
%\label{sub:certain}

%For any braid $\sigma$, we write $w_\sigma \in S_n$ for the underlying permutation. 

%\begin{proposition}(e.g. \cite[Proposition 2.16]{Hog})
%For any braid $\sigma$ and for all $i\in \{1,...,n\}$, the left action of $x_i$ on the corresponding Rouquier complex is homotopic to the right action of $x_{w_{\sigma(i)}}$.
%\end{proposition}

%In short, we will say that the left action $R \curvearrowright \sigma$ is homotopic to the right action $\sigma \curvearrowleft_{w(\sigma)} R$, {\em twisted} by the permutation $w_{\sigma}$. As a consequence, we obtain the following result:

%\begin{corollary}
%\label{cor:support rouquier}
%The $R$--module $\RHom_{K^b(\SBim_n)} (\1,\sigma)$ is supported on the subspace:
%$$
%\left\{x_i=x_{w_{\sigma}(i)},\ i=1,\ldots,n\right\}\subset \CC^n.
%$$
%\end{corollary}

\begin{corollary}
\label{cor:trace twisted identity}

Let $\sigma,\sigma'$ be any two braids. Then:
$$
\HHH(\sigma\sigma') = \RHom_{K^b(\SBim_n)}(\1,\sigma \otimes \sigma') \quad \text{and} \quad \HHH(\sigma'\sigma) = \RHom_{K^b(\SBim_n)}(\1,\sigma' \otimes \sigma)
$$
are isomorphic as $R$-modules, up to a twist by $w_{\sigma}$.

\end{corollary}

The above formula follows from Corollary \ref{cor:inv 2}, which applies to all invertible objects in a monoidal category.
 
%\begin{corollary}
%Let $\sigma$ be any braid, and let $L$ be the Rouquier complex of any pure braid. Then:
%$$
%\RHom_{K^b(\SBim_n)}(\1, \sigma \otimes L) \cong \RHom_{K^b(\SBim_n)}(\1,L \otimes  \sigma)
%$$
%as $R$--modules, since $w_L = \text{Id}$. 
%\end{corollary}

\begin{proposition}
The Soergel bimodule $B_i$ is self biadjoint, for all $i$. The Rouquier complex $\sigma$ for a braid $\sigma$ is biadjoint to $\sigma^{-1}$. 
\end{proposition}

The second statement of the above Proposition also follows from Corollary \ref{cor:inv 1} below, which is quite general, and actually implies the following stronger result:

\begin{corollary}
\label{cor:twisting homs}
For any $A,A'\in \SBim_n$ and any braid $\sigma$ there are canonical isomorphisms:
$$
\RHom_{K^b(\SBim_n)}(A\otimes \sigma,A' \otimes \sigma) \cong \RHom_{K^b(\SBim_n)}(A,A') \cong \RHom_{K^b(\SBim_n)}(\sigma \otimes A,\sigma\otimes A').
$$
\end{corollary}

\subsection{The trace functor}
\label{sub:trace}

We will henceforth write $R_n = \CC[x_1,...,x_n]$ to avoid confusion as to which number $n$ we are considering. For an extra variable $x_{n+1}$, we consider the category:
\begin{equation}
\label{eqn:upgrade}
\SBim_n[x_{n+1}]
\end{equation}
of Soergel bimodules which are equipped with an additional endomorphism denoted by $x_{n+1}$ that commutes with the action of $R_n$.
In other words, $\SBim_n[x_{n+1}]$ is the Karoubian envelope of the smallest full subcategory of $R_{n+1}\text{--mod--}R_{n+1}$ that contains the modules $B_1,\ldots B_{n-1}$ and is closed under $\otimes_{R_{n+1}}$ and grading shifts. 
 It is easy to see that the functors:
$$
\SBim_n[x_{n+1}] \xtofrom{} \SBim_n
$$
that forget the action of $x_{n+1}$, respectively tensor with $\CC[x_{n+1}]$, are adjoint with respect to each other. We will now recall the functors $I$ and $\Tr$ defined in \cite{Hog}, upgraded to the level of the category \eqref{eqn:upgrade}. At the level of additive categories, these functors are quite simple:
$$
I: \SBim_n[x_{n+1}] \longrightarrow \SBim_{n+1}
$$
is the full embedding. Meanwhile:
$$
\Tr: \SBim_{n+1} \longrightarrow \SBim_n[x_{n+1}], \qquad M \mapsto \text{Ker} \left(M \xrightarrow{x_{n+1} \otimes 1 - 1 \otimes x_{n+1}}  M \right)
$$
As shown in \cite{Hog}, these functors can be upgraded to the derived categories:
$$
D^b\left(\SBim_n[x_{n+1}]\right) \xtofrom[I]{\Tr}  D^b\left(\SBim_{n+1}\right)
$$
where the trace functor now encodes the full operation of multiplication by $x_{n+1} \otimes 1 - 1 \otimes x_{n+1}$, instead of simply the kernel:
$$
\Tr(M) = \left[ M \xrightarrow{x_{n+1} \otimes 1 - 1 \otimes x_{n+1}}  M \right].
$$
\begin{remark}
When working in the upgraded category \eqref{eqn:upgrade} rather than $\SBim_n$, one must be careful with Markov invariance, i.e. the statement (\cite{Kh}) that for $M\in \SBim_{n+1}$ one has:
$$
\Tr(M \otimes \sigma_n^{\pm 1})\simeq M \ \in \SBim_n
$$
In the upgraded category, this equation becomes:
\begin{equation}
\label{eqn: R1 upgraded}
\Tr(M \otimes \sigma_n^{\pm 1})\simeq \left[M\otimes \CC[x_{n+1}] \xrightarrow{x_n\otimes 1- 1 \otimes x_{n+1}} M\otimes \CC[x_{n+1}] \right] \ \in \SBim_n[x_{n+1}] %Andrei: I think it's not at all transparent why one needs to use precisely x_n\otimes 1- 1 \otimes x_{n+1} and not some other polynomial. And we should also explain why the RHS does not depend on the sign
\end{equation}
The proof is straightforward and we leave it to the reader. Remark that in the category $\SBim_n$ the complex \eqref{eqn: R1 upgraded} is quasi-isomorphic to $M$, but this is no longer true in $\SBim_n[x_{n+1}]$.

\begin{figure}[ht!]
\begin{tikzpicture}
\draw (0,-0.5)--(0,1);
\draw (0.5,-0.5)--(0.5,1);
\draw (1,-0.5)--(1,1);
\draw (1.5,-0.5)--(1.5,1);
\draw (-0.2,1)--(1.7,1)--(1.7,2)--(-0.2,2)--(-0.2,1);
\draw (0.7,1.5) node {$M$};
\draw (0,2)--(0,3.5);
\draw (0.5,2)--(0.5,3.5);
\draw (1,2)--(1,3.5);
\draw (1.5,2)--(1.5,2.2);
\draw (2,0)--(2,2.2);
\draw (1.5,3)--(1.5,3.5);
%crossing
\draw (2,2.2).. controls (2,2.7) and (1.5,2.5)..(1.5,3);
\draw[line width=5,white] (1.5,2.2).. controls (1.5,2.7) and (2,2.5)..(2,3);
\draw (1.5,2.2).. controls (1.5,2.7) and (2,2.5)..(2,3);
%\draw[line width=5,white] (1.5,2.2).. controls (1.5,2.7) and (2,2.5)..(2,3);
%\draw [line width=5,white] (2,2.2).. controls (2,2.7) and (1.5,2.5)..(1.5,3);
%\draw (2,2.2).. controls (2,2.7) and (1.5,2.5)..(1.5,3);
%loop
\draw (2,3)..controls (2,4) and (2.5,3)..(2.5,1.5);
\draw (2,0)..controls (2,-1) and (2.5,0)..(2.5,1.5);

\draw (3,1.5) node {$\simeq$};

\draw (4,-0.5)--(4,1);
\draw (4.5,-0.5)--(4.5,1);
\draw (5,-0.5)--(5,1);
\draw (5.5,-0.5)--(5.5,1);
\draw (3.8,1)--(5.7,1)--(5.7,2)--(3.8,2)--(3.8,1);
\draw (4.7,1.5) node {$M$};
\draw (4,2)--(4,3.5);
\draw (4.5,2)--(4.5,3.5);
\draw (5,2)--(5,3.5);
\draw (5.5,2)--(5.5,3.5);
\draw (6,-0.5)--(6,3.5);

\draw (7.5,1.5) node {$\xrightarrow{x_n \otimes 1 - 1 \otimes x_{n+1}}$};

\draw (9,-0.5)--(9,1);
\draw (9.5,-0.5)--(9.5,1);
\draw (10,-0.5)--(10,1);
\draw (10.5,-0.5)--(10.5,1);
\draw (8.8,1)--(10.7,1)--(10.7,2)--(8.8,2)--(8.8,1);
\draw (9.7,1.5) node {$M$};
\draw (9,2)--(9,3.5);
\draw (9.5,2)--(9.5,3.5);
\draw (10,2)--(10,3.5);
\draw (10.5,2)--(10.5,3.5);
\draw (11,-0.5)--(11,3.5);

\end{tikzpicture}
\label{fig:R1}
\caption{Markov move in $\SBim_n[x_{n+1}]$}
\end{figure}
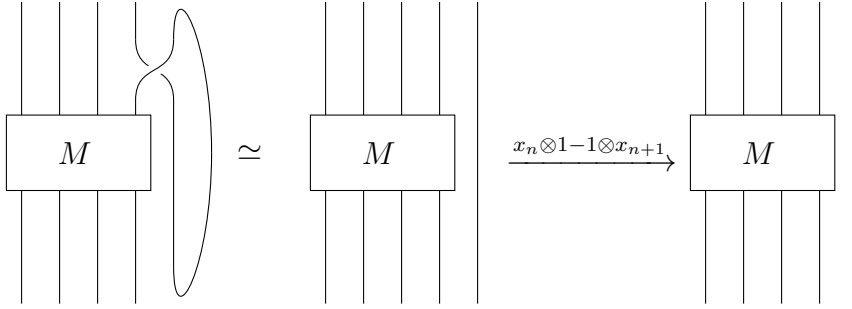

\end{remark}

\subsection{The main conjectures}
\label{sub:mainconj}

For the remainder of this Section, we will write $\FH_n^\dg = \FH_n^\dg(\CC)$ and $\CE_n = \CE_n(\CC)$, in the notation of Section \ref{sec:flag}. Our main Conjecture can be restated more precisely as follows: 

\begin{Conjecture}

There exists a pair of adjoint functors:
\begin{equation}
\label{eqn:funk}
K^b(\SBim_n) \xtofrom[\iota^*]{\iota_*} D^b\left(\coh_{\torus} \left( \FH_n^{\edg} \right) \right)
\end{equation}
where $\iota^*$ is monoidal and fully faithful. Moreover, we have:
\begin{equation}
\label{eqn:sk}
\iota_*(\iota^*N_1 \otimes M \otimes \iota^*N_2) \cong N_1 \otimes \iota_*(M) \otimes N_2
\end{equation}
for all $N_1,N_2 \in D^b\left(\coh_{\torus} \left( \FH_n^{\edg}\right) \right)$ and $M \in K^b(\SBim_n)$. In addition:
\begin{equation}
\label{eqn:nice 1}
\iota_* \1 = \CO \quad \text{and} \quad L_k = \iota^{*}(\CL_k) \qquad \stackrel{\eqref{eqn:sk}}\Longrightarrow
\end{equation}
\begin{equation}
\label{eqn:nice 2}
\stackrel{\eqref{eqn:sk}}\Longrightarrow \qquad \iota_* L_k = \CL_k \quad \forall \ k\in \{1,...,n\},
\end{equation}
where $\CO$ is the structure sheaf of $\FH_n^{\edg}$ and $\CL_k$ is the line bundle \eqref{eqn:taut line}. Finally, the following diagrams of functors commute (we write $\iota = \iota_{(n)}$ to keep track of $n$):
\begin{equation}
\label{eqn:diag1}
\includegraphics{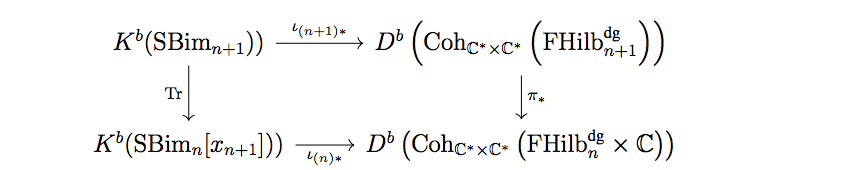}
\end{equation}
\begin{equation}
\label{eqn:diag2}
\includegraphics{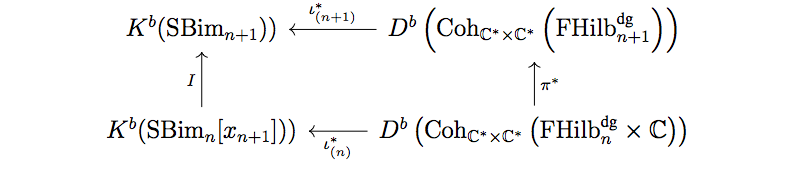}
\end{equation}
where the map $\pi : \FH_{n+1}^\edg \rightarrow \FH_n^\edg \times \CC$ is the particular case of \eqref{eqn:tower} for $* = \CC$.
\end{Conjecture}

%\subsection{Drawing braids}
%\label{sub:braids}

%Objects in $K^-(\SBim_n)$ can be represented by braids on $n$ strands, where the Rouquier complex corresponds to the positive crossing between the strands numbered $i$ and $i+1$, and tensor product corresponds to stacking braids on top of each other. The functor:
%$$
%I: \SBim_n[x_{n+1}] \longrightarrow \SBim_{n+1}
%$$
%corresponds to adding one more strand, while the functor:
%$$
%\Tr : \SBim_{n+1} \longrightarrow \SBim_n[x_{n+1}]
%$$
%corresponds to closing the last strand. The following braid is a visual respresentation of the full twist:
%$$
%\FT_n = \textcolor{red}{draw}
%$$
%and the braid where the last strand wraps around all other ones:
%$$
%L_n = \textcolor{red}{draw}
%$$

In broad strokes, the functor $\iota_*$ is given by sending each object $M \in K^b(\SBim_n)$ to:
\begin{equation}
\label{eqn:general 1}
\iota_* M = \bigoplus_{a_1,...,a_n \in \NN} \text{Hom}_{K^b(\SBim_n)} \left(\1, M \bigotimes_{k=1}^n L_k^{a_k} \right) 
\end{equation}
which is naturally a module for the $\NN^n$--graded dg algebra:
\begin{equation}
\label{eqn:general 2}
A = \bigoplus_{a_1,...,a_n \in \NN} \text{Hom}_{K^b(\SBim_n)} \left(\1, \bigotimes_{k=1}^n L_k^{a_k} \right)
\end{equation}
This algebra is commutative %(which follows from the fact that the objects $L_k$ are invertible, as explained in Subsection \ref{sub:invertible}) 
and %therefore 
$\iota_*M$ gives rise to a coherent sheaf on $(\spec \ A)/(\CC^*)^n$. Our conjecture entails the fact that this sheaf is actually supported on the $n$--fold iterated projectivization $\proj A \hookrightarrow (\spec \ A)/(\CC^*)^n$, and that in fact:
\begin{equation}
\label{eqn:general 3}
\proj \ A = \FH_n^\dg
\end{equation}
To upgrade to the setting of Remark \ref{rem:a grading}, we must replace the Hom spaces by RHom in \eqref{eqn:general 1} and \eqref{eqn:general 2}. We expect that this can be dealt with as in the following conjecture.

\begin{conjecture}
\label{conj:a grading}
Given the setup of Conjecture \ref{conj:1} we consider the object:
$$
T_n  = \iota^*(\CT_n) \in K^b(\SBim_n)
$$
Then we claim that for any object $M \in K^b(\SBim_n)$, we have an isomorphism:
\begin{equation}
\label{eqn:rhom vs hom}
\RHom_{K^b(\SBim_n)}(\1, M) \cong \Hom_{K^b(\SBim_n)}\left(\1, M \otimes \wedge^\bullet T^\vee_n \right)
\end{equation}
which is functorial with respect to the action of the algebra $A \curvearrowright M$ of \eqref{eqn:general 2}.

\end{conjecture}

Assuming Conjecture \ref{conj:a grading}, one may ask if there is a sheaf on the flag Hilbert scheme which is defined by replacing Hom with RHom in \eqref{eqn:general 1}. By \eqref{eqn:rhom vs hom} and \eqref{eqn:sk}, this sheaf would be:
$$
\iota_* \left(M \otimes \iota^*\left(\wedge^\bullet \CT_n^\vee \right)\right) = \iota_*M \otimes \wedge^\bullet \CT_n^\vee
$$
This sheaf should naturally be thought to live on $\tot_{\FH_n^\dg} (\CT_n[1]) = \spec_{\FH_n^\dg} \left(\wedge^\bullet \CT_n^\vee \right)$, as in Remark \ref{rem:a grading}. The entire picture presented in this Subsection will be explained in more detail in Section \ref{sec:geometry}, when we develop the formalism of categories over schemes in general. 

\begin{proof}[Proof of Corollary \ref{cor:1}] The fact that $\iota^*$ is a monoidal functor, together with \eqref{eqn:nice 1}, imply that:
$$
\sigma := \prod_{k=1}^n \FT_k^{a_k} = \prod_{k=1}^n \iota^{*}(\det \CT_k)^{\otimes a_k} = \iota^*\left(\bigotimes_{k} (\det \CT_k)^{\otimes a_k}\right).
$$
Corollary \ref{cor:inv 1} below implies that:
$$
\HHH( \sigma ) := \RHom_{K^b(\SBim_n)}(\1, \sigma)= \RHom_{K^b(\SBim_n)}(\sigma^{-1} , \1)
$$
while \eqref{eqn:rhom vs hom} implies that:
$$
\HHH( \sigma ) = \Hom_{K^b(\SBim_n)}(\sigma^{-1} \otimes \wedge^\bullet T_n, \1) = \Hom_{K^b(\SBim_n)} \left[ \iota^*\left(\bigotimes_{k} (\det \CT_k)^{-a_k} \otimes \wedge^\bullet \CT_n \right) , \1 \right]
$$
The adjunction of $\iota^*$ and $\iota_*$, together with the conjectured fact that $\iota_*\1 = \CO$, imply that:
$$
\HHH( \sigma ) = \RHom_{\FH_n^\dg} \left( \bigotimes_{k} (\det \CT_k)^{-a_k} \otimes \wedge^\bullet \CT_n, \CO \right) 
$$
Dualizing the RHom produces the desired result. 
\end{proof}

\subsection{}
\label{sub:topological computations}

Proposition \ref{prop:dg scheme} describes flag Hilbert schemes as projective towers, which implies that:
$$
\CE_n \cong \pi_*(\CL_{n+1}) \in D^b \left(\coh_{\torus} \left( \FH_n^\dg \right) \right)
$$
Define the following object:
\begin{equation}
\label{eqn:def e}
E_n := \Tr(L_{n+1}) \in K^b(\SBim_n)
\end{equation}
Conjecture \ref{conj:1} implies that:
\begin{equation}
\label{eqn:relation e}
\iota_*(E_n) = \iota_*(\Tr(L_{n+1})) = \pi_*(\iota_*(L_{n+1})) = \pi_*(\CL_{n+1})) \cong \CE_n.
\end{equation}

\begin{conjecture}
\label{conj:main}

The following topological facts hold for all $n\geq 0$.

\begin{itemize}

\item[(a)] $E_n$ is an explicit complex in terms of $I(E_{n-1})$ and $L_n$, as in \eqref{eqn:explicit} below.

\item[(b)] The following equation holds in $K^b(\SBim_n[x_{n+1}])$:
\begin{equation}
\label{eqn:conjmain}
S^k E_n \cong \Tr(L_{n+1}^k) \quad \forall \ k\ge 0. 
\end{equation}

\item[(c)] The Koszul complex 
\begin{equation}
\label{eqn:koszul topology}
\left[... \stackrel{\eta}\longrightarrow I(\wedge^2 E_n) \otimes L_{n+1}^{-2} \stackrel{\eta}\longrightarrow I(E_n) \otimes L_{n+1}^{-1} \stackrel{\eta}\longrightarrow R \right]
\end{equation}
is acyclic, where $I(E_n) \stackrel{\eta}\rightarrow L_{n+1}$ denotes the adjoint map to \eqref{eqn:def e}.

\end{itemize}

\end{conjecture}

Statement (a) implies that $E_n$ lies in the monoidal subcategory of $K^b(\SBim_n)$ generated by $L_1,...,L_n$. Since this subcategory is symmetric and Karoubian, the objects $S^k E_n$ and $\wedge^k E_n$ that appear in (b) and (c) are well-defined: as in \cite{Deligne}, they are simply the projections of $E_n^{\otimes k}$ defined by the symmetric and antisymmetric projectors in the symmetric group $S_k$, respectively. The following result is proved in Section \ref{sub:proof}, and will show how to reduce our main Conjecture \ref{conj:1} to the topological computations of Conjecture \ref{conj:main} (a)--(c).

\begin{theorem}
\label{thm:one implies the other}
Conjecture \ref{conj:main} implies Conjecture \ref{conj:1}.
\end{theorem}

\subsection{$E_n$ as an explicit braid}
\label{sub:explicit braid}

The object $E_n = \Tr(L_{n+1}) \in K^b(\SBim_n[x_{n+1}])$ has a simple topological meaning, represented below.

\begin{figure}[ht!]
\begin{tikzpicture}
\begin{scope}[yscale=-1]
\draw (0,0)..controls (0,0.5) and (0.5,0.5)..(0.5,1);
\draw (0.5,0)..controls (0.5,0.5) and (1,0.5)..(1,1);
\draw (1,0)..controls (1,0.5) and (1.5,0.5)..(1.5,1);
\draw [line width=5,white] (1.5,0)..controls (1.5,0.5) and (0,0.5)..(0,1);
\draw  (1.5,0)..controls (1.5,0.5) and (0,0.5)..(0,1);

\draw  (1.5,2)..controls (1.5,1.5) and (0,1.5)..(0,1);
\draw [line width=5,white] (0,2)..controls (0,1.5) and (0.5,1.5)..(0.5,1);
\draw [line width=5,white] (0.5,2)..controls (0.5,1.5) and (1,1.5)..(1,1);
\draw [line width=5,white] (1,2)..controls (1,1.5) and (1.5,1.5)..(1.5,1);
\draw (0,2)..controls (0,1.5) and (0.5,1.5)..(0.5,1);
\draw (0.5,2)..controls (0.5,1.5) and (1,1.5)..(1,1);
\draw (1,2)..controls (1,1.5) and (1.5,1.5)..(1.5,1);
\end{scope}
\end{tikzpicture} \qquad \qquad 
\begin{tikzpicture}
\begin{scope}[yscale=-1]
\draw (0,0)--(0,1);
\draw (0.5,0)--(0.5,1);
\draw (1,0)--(1,1);
\draw [line width=5,white](-0.5,1)..controls (-0.5,0.5) and (1.5,0.5)..(1.5,1);
\draw (-0.5,1)..controls (-0.5,0.5) and (1.5,0.5)..(1.5,1);

\draw  (-0.5,1)..controls (-0.5,1.5) and (1.5,1.5)..(1.5,1);
\draw [line width=5,white] (0,2)--(0,1);
\draw [line width=5,white] (0.5,2)--(0.5,1);
\draw [line width=5,white] (1,2)--(1,1);
\draw (0,2)--(0,1);
\draw (0.5,2)--(0.5,1);
\draw (1,2)--(1,1);
\end{scope}
\end{tikzpicture}
\caption{The braid $L_4$ and its partial trace $E_3$.}
\end{figure}
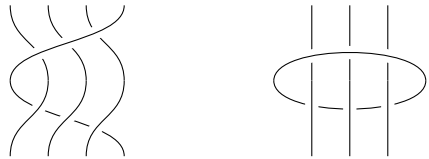

The relation between the tangle $E_n$ and the complex $\CE_n$ is expected to categorify the classical formula for $E_n$ (e.g. \cite{Morton}) in the skein algebra. Specifically, skein relations are topological equalities between knots which only differ near a crossing:

\begin{figure}[ht!]
\begin{tikzpicture}
\draw [<-] (0,1)--(1,0);
\draw [->,white,line width=5] (0,0)--(1,1);
\draw [->] (0,0)--(1,1);
%\draw [->,white,line width=5] (0,1)--(1,0);
%\draw [<-] (0,1)--(1,0);

\draw (1.5,0.5) node {$- q$};

\draw [->] (2,0)--(3,1);
\draw [->,white,line width=5] (2,1)--(3,0);
\draw [<-] (2,1)--(3,0);
%\draw [->,white,line width=5] (2,0)--(3,1);
%\draw [->] (2,0)--(3,1);

\draw (4,0.5) node {$=(1-q)$};

\draw [->] (5.5,0)--(5.5,1);
\draw [->] (6,0)--(6,1);
\end{tikzpicture}
\caption{Skein relation}
\label{fig:skein def}
\end{figure}

In \(K^b(\SBim_n)\) such equalities must be replaced with exact sequences. For example, consider the skein relation applied to the bottom right crossing of the braid $L_{n+1}$.
\begin{figure}[ht!]
\begin{tikzpicture}
\begin{scope}[yscale=-1, shift={(0,-2)}]
\draw (0,0)..controls (0,0.5) and (0.5,0.5)..(0.5,1);
\draw (0.5,0)..controls (0.5,0.5) and (1,0.5)..(1,1);
\draw (1,0)..controls (1,0.5) and (1.5,0.5)..(1.5,1);
\draw [line width=5,white] (1.5,0)..controls (1.5,0.5) and (0,0.5)..(0,1);
\draw  (1.5,0)..controls (1.5,0.5) and (0,0.5)..(0,1);

\draw  (1.5,2)..controls (1.5,1.5) and (0,1.5)..(0,1);
\draw [line width=5,white] (0,2)..controls (0,1.5) and (0.5,1.5)..(0.5,1);
\draw [line width=5,white] (0.5,2)..controls (0.5,1.5) and (1,1.5)..(1,1);
\draw [line width=5,white] (1,2)..controls (1,1.5) and (1.5,1.5)..(1.5,1);
\draw (0,2)..controls (0,1.5) and (0.5,1.5)..(0.5,1);
\draw (0.5,2)..controls (0.5,1.5) and (1,1.5)..(1,1);
\draw (1,2)..controls (1,1.5) and (1.5,1.5)..(1.5,1);
\end{scope}

\draw (2.2,1) node {$-q$};

\begin{scope}[yscale=-1, shift={(0,-2)}]
\draw (3,0)..controls (3,0.5) and (3.5,0.5)..(3.5,1);
\draw (3.5,0)..controls (3.5,0.5) and (4,0.5)..(4,1);
\draw (4,0)..controls (4,0.5) and (4.5,0.5)..(4.5,1);
\draw [line width=5,white] (4.5,0)..controls (4.5,0.5) and (3,0.5)..(3,1);
\draw  (4.5,0)..controls (4.5,0.5) and (3,0.5)..(3,1);

\draw (4,2)..controls (4,1.5) and (4.5,1.5)..(4.5,1);
\draw [line width=5,white]  (4.5,2)..controls (4.5,1.5) and (3,1.5)..(3,1);
\draw  (4.5,2)..controls (4.5,1.5) and (3,1.5)..(3,1);
\draw [line width=5,white] (3,2)..controls (3,1.5) and (3.5,1.5)..(3.5,1);
\draw [line width=5,white] (3.5,2)..controls (3.5,1.5) and (4,1.5)..(4,1);
\draw (3,2)..controls (3,1.5) and (3.5,1.5)..(3.5,1);
\draw (3.5,2)..controls (3.5,1.5) and (4,1.5)..(4,1);
\end{scope}

\draw (5.7,1) node {$=(1-q)$};

\begin{scope}[yscale=-1, shift={(0,-2)}]
\draw (7,0)..controls (7,0.5) and (7.5,0.5)..(7.5,1);
\draw (7.5,0)..controls (7.5,0.5) and (8,0.5)..(8,1);
\draw (8,0)..controls (8,0.5) and (8.5,0.5)..(8.5,1);
\draw [line width=5,white] (8.5,0)..controls (8.5,0.5) and (7,0.5)..(7,1);
\draw  (8.5,0)..controls (8.5,0.5) and (7,0.5)..(7,1);

\draw (8.5,1)--(8.5,2);
\draw  (8,2)..controls (8,1.5) and (7,1.5)..(7,1);
\draw [line width=5,white] (7,2)..controls (7,1.5) and (7.5,1.5)..(7.5,1);
\draw [line width=5,white] (7.5,2)..controls (7.5,1.5) and (8,1.5)..(8,1);
\draw (7,2)..controls (7,1.5) and (7.5,1.5)..(7.5,1);
\draw (7.5,2)..controls (7.5,1.5) and (8,1.5)..(8,1);
\end{scope}
\end{tikzpicture}
\caption{Skein relation for $L_{n+1}$}
\label{fig:skein}
\end{figure}
If one closes the last strand in Figure \ref{fig:skein} and applies a Markov move, one gets the following formula in the Grothendieck group of $\SBim_n$ (which is isomorphic to the Hecke algebra):
\begin{equation}
\label{eqn:explicit0}
\langle E_n\rangle - \langle I(E_{n-1})\rangle = (1-q)\langle L_{n} \rangle
\end{equation}
In the category $K^b(\SBim_n[x_{n+1}])$, the above equality is lifted to an exact sequence:
\begin{equation}
\label{eqn:explicit}
\left[ qt L_{n}\xrightarrow{(0,x_n-x_{n+1})} \underline{q L_{n}\oplus t L_{n}} \xrightarrow{(x_n-x_{n+1},0)}L_{n} \right] \cong \left[ E_n \longrightarrow \widetilde{I(E_{n-1})} \right]
\end{equation}
where $t = s^2/q$ and $\widetilde{I(E_{n-1})}$ refers to the same braid as $I(E_{n-1})$, but with the variables on the last two strands switched (compare with \eqref{eqn:ind}). This is a crucial feature of the category $\SBim_n[x_{n+1}]$, where the variables $x_n$ and $x_{n+1}$ play different roles. Also note that \eqref{eqn:explicit} consists of 4 copies of $L_n$ instead of the two of \eqref{eqn:explicit0}, due to the modified Markov move \eqref{eqn: R1 upgraded}.

\subsection{Geometric Markov invariance}
\label{sub:markov}

%Let us now briefly describe a relative version of \eqref{eqn:general 1}, which will be described in greater generality in Subsection \ref{sub:relative}. Consider the projection map $\pi:\FH_{n+1}^{\dg}  \rightarrow \FH_n^{\dg}  \times \CC$ and recall from \eqref{eqn:def dg} that it is the projectvization of a certain complex of vector bundles $\CE_n$. Thus the sheaf $\iota_{(n+1)*}(M)$ can be reconstructed from $\pi_*(\iota_{(n+1)*}(M) \otimes \CO(k))$ for $k \gg 0$. But \eqref{eqn:sk} and the diagram \eqref{eqn:diag1} imply that:
%$$
%\pi_*(\iota_{(n+1)*}(M) \otimes \CO(k)) = \pi_*(\iota_{(n+1)*}(M \otimes \CO(k))) = \iota_{(n)*}(\Tr(M \otimes L_{n+1}^k))
%$$
%and so we can reconstruct $\iota_{(n+1)*}(M)$ from $\iota_{(n)*}$ and knowledge of $\Tr(M \otimes L_{n+1}^k)$ for $k$ large enough. We will provide more details in Subsection \ref{sub:relative}. In the meantime, we will use this principle to 

In the category of Soergel bimodules, equation \eqref{eqn: R1 upgraded} governs the behavior of objects under Markov moves:
\begin{equation}
\label{eqn:markov}
\alpha \leadsto i(\alpha), \qquad \alpha \leadsto i(\alpha) \cdot \sigma_n, \qquad \alpha \leadsto i(\alpha) \cdot \sigma^{-1}_n
\end{equation}
where $i$ is the operation of adding an extra strand to a braid $\alpha$ on $n$ strands. We will now study how the complexes of sheaves $\CB(\alpha) = \iota_*(\alpha) \in D^b(\coh_{\torus}(\FH_n^\dg(\CC))$ behave under the same moves. Throughout this Subsection, we write $\FH_n^\dg = \FH_n^\dg(\CC)$ and:
$$
\pi:\FH_{n+1}^{\dg}  \rightarrow \FH_n^{\dg} \times \CC
$$
for the standard projection. The following Corollary is an easy consequence of Conjecture \ref{conj:main}, as we will show in Subsection \ref{sub:proof}.

\begin{corollary}
\label{cor:markov 1}
For any braid $\alpha$ on $n$ strands, we have: 
\begin{equation}
\label{eqn:markov 0}
\CB(i(\alpha)) = \pi^*(\CB(\alpha)) .
\end{equation}

\end{corollary}

%Andrei: I think it would be nice to postpone the proofs until we discuss the push-forwards of L_{n+1}^k down the tower (maybe in one of the later sections). Please don't erase the following though

%\begin{proof} The two sides of the equation are sheaves on $\FH_{n+1}^{\dg}$, so we must prove that:
%\begin{equation}
%\label{eqn:adam}
%\pi_*(\CB(i(\alpha)) \otimes \CO(k)) \cong \pi_*(\pi^*(\CB(\alpha))\otimes \CO(k)) = \CB(\alpha) \otimes S^k \CE_n
%\end{equation}
%for all $k\gg 0$. Moreover, the above isomorphism must be compatible with the multiplication $\CE_n \otimes S^k \CE_n \rightarrow S^{k+1}\CE_n$. But note that the left hand side of \eqref{eqn:adam} equals
%$$
%\pi_*(\iota_*(B(i(\alpha))) \otimes \CO(k)) = \pi_*(\iota_*(B(i(\alpha)) \otimes L_{n+1}^k)) = \iota_*(\Tr(I(B(\alpha)) \otimes L_{n+1}^k)) = 
%$$
%$$
%= \iota_*(B(\alpha) \otimes \Tr(L_{n+1}^k)) = \iota_*(B(\alpha) \otimes S^k E_n) = \iota_*(B(\alpha) \otimes \iota^*( S^k \CE_n)) = \CB(\alpha) \otimes S^k\CE_n 
%$$
%where the equalities on the last line follow from \eqref{eqn:conjmain}. 
%\end{proof}

To tackle the second and third Markov moves of \eqref{eqn:markov}, we consider the dg subscheme:
\begin{equation}
\label{eqn:def zn}
Z_n \subset \FH_{n+1}^\dg
\end{equation}
$$
\CO_{Z_n} := \left[\ldots \xrightarrow{y_{n,n+1}} \frac {q^2t\CL_n}{\CL_{n+1}} \xrightarrow{x_n-x_{n+1}} \frac {qt\CL_n}{\CL_{n+1}} \xrightarrow{y_{n,n+1}}  q\CO \xrightarrow{x_n-x_{n+1}} \CO \right],
$$
where $y_{n,n+1}$ denotes the last subdiagonal entry of the matrix $Y$ of \eqref{eqn:triple matrices}, regarded as an endomorphism $t\CL_n \rightarrow \CL_{n+1}$ on $\FH_{n+1}^{\dg}$. The fact that $\CO_{Z_n}$ is a complex follows from:
$$
0 = [X,Y]_{n,n+1} = x_n y_{n,n+1} - y_{n,n+1} x_{n+1}
$$ 

\begin{conjecture}
\label{conj:markov}
For any braid $\alpha$ on $n$ strands, we have: 
\begin{equation}
\label{eqn:markov 1}
\CB(i(\alpha) \cdot \sigma_n) = \pi^*(\CB(\alpha)) \otimes \CO_{Z_n} .
\end{equation}

\end{conjecture}

\begin{corollary}
\label{cor:markov 2}
Conjecture~\ref{conj:markov} implies that for any braid $\alpha$ on $n$ strands: 
\begin{equation}
\label{eqn:markov 2}
\CB(i(\alpha) \cdot \sigma^{-1}_n) = \pi^*(\CB(\alpha)) \otimes \CO_{Z_n} \otimes \frac {\CL_n}{\CL_{n+1}} .
\end{equation}

\end{corollary}

\begin{proof} Note the following the equation in the braid group:
$$
L_{n+1} = \sigma_n \cdot L_n \cdot \sigma_n \Rightarrow \sigma_{n}^{-1} = L_{n+1}^{-1} \cdot \sigma_n \cdot L_{n} \Rightarrow $$
$$
 i(\alpha) \cdot  \sigma_n^{-1} = i(\alpha) \cdot L_{n+1}^{-1}\cdot \sigma_n  \cdot L_{n} = L_{n+1}^{-1}\cdot i(\alpha) \cdot \sigma_n  \cdot L_{n}.
$$
since $L_{n+1}$ commutes with the image of $i$. Applying $\CB(-)$ to the above equation implies:
$$
\CB(i(\alpha)\cdot \sigma_n^{-1}) = \iota_*( i(\alpha)  \otimes  \sigma_n^{-1})  = \iota_*(L_{n+1}^{-1}\otimes i(\alpha)  \otimes \sigma_n  \otimes L_n))
$$
As in Conjecture \ref{conj:1}, we have $L_k = \iota^*(\CL_k)$ for all $k$, and therefore \eqref{eqn:sk} implies \eqref{eqn:markov 2}. 

\end{proof}

Equations~\eqref{eqn:markov 0}--\eqref{eqn:markov 2} are compatible with the stabilization invariance of \(\HHH\) at the level of equivariant Euler characteristic. 
\begin{proposition}
\label{prop:markov integrals}

For any braid $\alpha$ on $n$ strands, we have:
\begin{equation}
\label{eqn:markov int 0}
\chi \left( 
%\int_{\FH_{n+1}^\edg} 
\CB(i(\alpha)) \otimes \wedge^\bullet \CT_{n+1}^\vee  \right) = \frac {1-a}{1-q} \chi \left( 
%\int_{\FH_{n}^\edg}
 \CB(\alpha) \otimes \wedge^\bullet \CT_n^\vee \right)
\end{equation}
Assuming Conjecture \ref{conj:markov}, we further have:
\begin{equation}
\label{eqn:markov int 1}
\chi \left( 
%\int_{\FH_{n+1}^\edg} 
\CB(i(\alpha) \cdot \sigma_n) \otimes \wedge^\bullet \CT_{n+1}^\vee \right)\ \ = \ \ \chi \left(
% \int_{\FH^{\edg}_{n}}
  \CB(\alpha) \otimes \wedge^\bullet \CT_n^\vee \right)
 \end{equation}
\begin{equation}
\label{eqn:markov int 2}
\chi \left(
% \int_{\FH_{n+1}^\edg}
  \CB(i(\alpha) \cdot \sigma^{-1}_n) \otimes \wedge^\bullet \CT_{n+1}^\vee \right) =  \frac a{qt} \chi \left( 
  %\int_{\FH_{n}^\edg}
   \CB(\alpha) \otimes \wedge^\bullet \CT_n^\vee \right)
\end{equation}
\end{proposition}

\begin{proof} 
%Integrals should be interpreted as Euler characteristics, and therefore we may replace the sheaves in \eqref{eqn:markov int 0}--\eqref{eqn:markov int 2} by their $K$--theory classes. Therefore, we may write:
We replace the sheaves in \eqref{eqn:markov int 0}--\eqref{eqn:markov int 2} by their $K$--theory classes and write:
$$
[\CT_{n+1}] = \pi^*\left(  [\CT_n] \right) + [\CL_{n+1}]
$$
and:
\begin{equation}
\label{eqn:k-theory zn}
[\CO_{Z_n}] = (1-q) \left(1-\frac {qt[\CL_n]}{[\CL_{n+1}]} \right)^{-1}
\end{equation}
Since $\int$ is just pushforward to a point, it can be decomposed along the projection map $\pi: \FH_{n+1}^\dg \rightarrow \FH_n^\dg \times \CC$. In other words, for all sheaves $\CA$ one has:
$$
\int_{\FH_{n+1}^\dg} \CA = \int_{\FH_{n}^\dg\times \CC}\pi_*\CA
$$
We will apply this equality for the $K$--theory class: 
$$
[\CA] = [\CB(i(\alpha))] \cdot \wedge^\bullet [\CT_{n+1}^\vee] = \pi^*\left( [\CB(\alpha)] \cdot  \wedge^\bullet [\CT_{n}^\vee] \right) \cdot \left(1- \frac a{[\CL_{n+1}]} \right)
$$
where in the second equality we have used \eqref{eqn:markov 0}. Then we may prove \eqref{eqn:markov int 0} by noting that:
$$
\chi \left( 
%\int_{\FH_{n+1}^\dg} 
[\CB(i(\alpha))] \cdot \wedge^\bullet [\CT_{n+1}^\vee] \right)= \chi \left(  
%\int_{\FH_{n}^\dg\times \CC}
 \pi_* \left[ \pi^*\left( [\CB(\alpha)] \cdot \wedge^\bullet [\CT_{n}^\vee] \right) \cdot \left(1- \frac a{[\CL_{n+1}]} \right) \right] \right)= 
$$
\begin{equation}
\label{eqn:above formula}
= \chi \left( 
%\int_{\FH_{n}^\dg\times \CC}
[\CB(\alpha)] \cdot \wedge^\bullet [\CT_{n}^\vee] \cdot \pi_*\left(1- \frac a{[\CL_{n+1}]} \right) \right)= (1-a) \chi \left( 
%\int_{\FH_{n}^\dg\times \CC}
[\CB(\alpha)] \cdot \wedge^\bullet [\CT_{n}^\vee]\right)
\end{equation}
(the additional factor of $1-q$ in the right hand side of \eqref{eqn:markov int 0} comes from integrating over $\CC$). To establish the last equality in \eqref{eqn:above formula}, we note that it holds at the categorified level:
\begin{equation}
\label{eqn:f1}
\pi_*\left(\CO_{\FH_{n+1}^\dg} \right) = \CO_{\FH_{n}^\dg \times \CC} = \pi_* \left(\CO_{\FH_{n+1}^\dg} \otimes \CL_{n+1}^{-1}  \right)
\end{equation}
where the first equality is a consequence of the fact that $\pi$ is the projectivization of $\CE^\vee_n$, and the second equality follows from the first and \eqref{eqn:serre} for $\CA = \CO$. Similarly, if we assume formula \eqref{eqn:markov 1} (which would also imply \eqref{eqn:markov 2}, according to Corollary \ref{cor:markov 2}), then relations \eqref{eqn:markov int 1} and \eqref{eqn:markov int 2} follow from:
\begin{equation}
\label{eqn:f2}
\pi_*(\CO_{Z_n}) = \left[q\CO \xrightarrow{x_n-x_{n+1}}\CO \right], \qquad \pi_* \left(\CO_{Z_n} \otimes \frac 1{\CL_{n+1}} \right) = 0 %Andrei: I do not see these formulas from \eqref{eqn:def zn}, at least not at the level of sheaves. Are we completely sure they're true?
\end{equation}
\begin{equation}
\label{eqn:f3}
\pi_* \left(\CO_{Z_n} \otimes \frac 1{\CL^2_{n+1}} \right) = \left[\frac 1{t \CL_n} \xrightarrow{x_n-x_{n+1}} \frac 1{qt \CL_n} \right][1]
\end{equation}
We will only prove these equalities at the level of $K$--theory, by using \eqref{eqn:k-theory zn}. Indeed, since the map $\pi$ is $\PP \CE^\vee_n$, the push-forwards of the powers of $\CL_{n+1} = \CO(1)$ are encoded by:
\begin{equation}
\label{eqn:expansion}
\pi_* \left( \delta\left(\frac {\CL_{n+1}}z \right) \right) = S^*_{z \sim \infty} [\CE_n] - S^*_{z \sim 0} [\CE_n] = 
\end{equation}
$$
= \frac {\wedge^*_{z \sim \infty} [qt\CT_n] \wedge^*_{z \sim \infty} [\CT_n]}{(1-z^{-1})\wedge^*_{z \sim \infty} [q\CT_n] \wedge^*_{z \sim \infty} [t\CT_n]} - \frac {\wedge^*_{z \sim 0} [qt\CT_n] \wedge^*_{z \sim 0} [\CT_n]}{(1-z^{-1})\wedge^*_{z \sim 0} [q\CT_n] \wedge^*_{z \sim 0} [t\CT_n]}
$$
where the $\delta$ function is $\delta(z) = \sum_{k=-\infty}^{\infty} z^k$. In the right hand side, we write:
$$
S^*_z[\CV] = \sum_{k=0}^\infty (-z)^{-k} \cdot S^k\CV, \qquad \wedge^*_z[\CV] = \sum_{k=0}^\infty (-z)^{-k} \cdot \wedge^k\CV
$$
and the notations $S^*_{z\sim 0}, \wedge^*_{z\sim 0}$ and $S^*_{z\sim \infty}, \wedge^*_{z\sim \infty}$ refer to expanding the rational functions $S^*_z, \wedge^*_z$ in the domains $z\sim 0$ and $z\sim \infty$, respectively. Applying \eqref{eqn:k-theory zn}, we obtain:
$$
\pi_*\left( [\CO_{Z_n}] \cdot \delta\left(\frac {\CL_{n+1}}z \right) \right) = \pi_*\left( \frac {1-q}{1-\frac {qt[\CL_n]}{[\CL_{n+1}]}} \cdot \delta\left(\frac {\CL_{n+1}}z \right) \right) = \frac {1-q}{1-\frac {qt[\CL_n]}z} \cdot \pi_* \left( \delta\left(\frac {\CL_{n+1}}z \right) \right)
$$
and we can compute the right hand side using \eqref{eqn:expansion}. To obtain \eqref{eqn:f2} and \eqref{eqn:f3}, we must extract the coefficients of $z^0$, $z^{1}$, $z^2$ in the right hand side of the above equality, and it is easy to see that one obtains $1-q$, $0$ and $\frac {q-1}{qt[\CL_n]}$, respectively.
\end{proof}

\subsection{Correspondences}
\label{sub:nakajima}

Formula \eqref{eqn:markov 0} can be expressed in terms of the complexes of sheaves:
$$
\CF ( \sigma ) = \nu_*(\CB(\sigma)) \ \in \ D^b(\coh_{\torus}(\H_n))
$$
of \eqref{eqn:sheaf on hilb}, where $\nu: \FH_n^\dg \rightarrow \H_n$ is the map \eqref{eqn:map to hilb}. Specifically, we have the spaces:
$$
\includegraphics{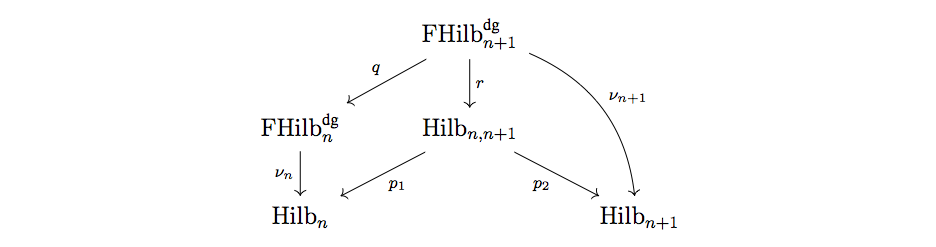}
$$
where $\H_{n,n+1} = \{I \in \H_n, I' \in \H_{n+1}, I \supset I' \text{ with quotient supported on } \{y=0\} \}$ are the correspondences studied by Nakajima and Grojnowski to describe the cohomology groups of Hilbert schemes. At the categorified level, their construction gives rise to a functor:
$$
D^b(\coh_{\torus}(\H_n)) \stackrel{\alpha}\longrightarrow D^b(\coh_{\torus}(\H_{n+1})), \qquad \alpha = p_{2*}p_1^*
$$
To establish \eqref{eqn:nakajima 1}, note that $\CF(i(\sigma))$ equals:
$$
\nu_{n+1 *}(\CB(i(\sigma))) = p_{2*}(r_*(\CB(i(\sigma)))) = p_{2*}(r_*(q^*(\CB(\sigma)))) = p_{2*}(p_1^*(\nu_{n*}(\CB(\sigma)))) = \alpha(\CF(\sigma))
$$
where the second equality follows from \eqref{eqn:markov 0}, and the third equality follows from the fact that the rhombus is cartesian. This latter fact may seem obvious at the level of closed points, but scheme-theoretically it only holds because we have replaced the badly behaved scheme $\FH_n$ with the nicely behaved dg scheme $\FH_n^\dg$. 

%By a similar argument, formulas \eqref{eqn:markov 1} and \eqref{eqn:markov 2} imply formulas \eqref{eqn:nakajima 2} and \eqref{eqn:nakajima 3}, if we define:
%$$
%D^b(\coh_{\torus}(\H_n)) \stackrel{\alpha', \alpha''}\longrightarrow D^b(\coh_{\torus}(\H_{n+1})), \quad \alpha' = \bar{p}_{2*}\bar{p}_1^*, \quad \alpha''= \bar{p}_{2*}\left(\frac {\CL_n}{\CL_{n+1}} \cdot \bar{p}_1^* \right)
%$$
%where the maps $\bar{p}_1$ and $\bar{p}_2$ arise as in the following diagram:
%$$
%\begin{tikzcd}
%& Z \ar{ld}[swap]{\bar{q}} \ar{d}{\bar{r}} \arrow[bend left]{ddr} \\
%\FH^\dg_n \ar{d} & Z_{n,n+1} \ar{ld}{\bar{p}_1} \ar{rd}[swap]{\bar{p}_2} & \\
%\H_n & & \H_{n+1}
%\end{tikzcd}
%$$
%In order for the rhombus to be a fiber diagram, we must define the dg subscheme $Z_{n,n+1} \subset \H_{n,n+1}$ 

\subsection{Mirror braids}
\label{sub:mirror}

In this section, we will relate the operation of mirroring braids (i.e. looking at them from behind) with Verdier duality on the category of coherent sheaves on $\FH_n$.

\begin{proposition}
\label{prop:mirror} 

For any  $\CF\in D^b\coh(\FH^\edg_n)$ one has:
$$
\int_{\FH_n^\edg} \CF \otimes \wedge^\bullet \CT_n^\vee \cong \left[\int_{\FH_n^\edg}\CF^{\vee} \otimes \wedge^\bullet \CT_n \right]^{\vee}
$$
where the $a$-grading in the right hand side is reversed from $i$ to $n-i$.
\end{proposition}

The Proposition is obvious, since it's just stating that a proper push-forward commutes with Verdier duality. It is natural to conjecture, therefore, that mirroring the braid $\sigma$ simply corresponds to dualizing the complex of sheaves $\CB(\sigma)$ on $\FH_n^\dg$:

\begin{conjecture}
\label{eqn:mirror}

For any braid $\sigma$, we have:
$$
\CB(\sigma^{\vee})=\CB(\sigma)^{\vee},
$$
where $\beta^{\vee}$ denotes the mirror of $\beta$.
\end{conjecture}

The following example shows that the computation of a dual sheaf can be nontrivial.

\begin{example}
\label{ex:mirror}

As we will see in Section \ref{sec:n=2} (and also from Section \ref{sub:markov}), the braid $\sigma_1\in \SBim_2$ corresponds to the structure sheaf of $\FH_2(\point)\times \CC\subset \FH_2(\CC)$, while $\sigma_1^{-1}\in \SBim_2$ corresponds to $\CO(-1)$ on  $\FH_2(\point)\times \CC$. The fact that the objects
$$
\CB(\sigma_1) = \CO_{\FH_2(\point)\times \CC} \qquad \text{and} \qquad \CB(\sigma^{-1}_1) = \CO_{\FH_2(\point)\times \CC}(-1)
$$
are dual to each other follows from the fact that the exact sequence:
$$
\CO_{\FH_2(\point)\times \CC}\xleftarrow{} \CO_{\FH_2(\CC)}\xleftarrow{x_1-x_2}\CO_{\FH_2(\CC)}\xleftarrow{w} \CO(-1)_{\FH_2(\point)\times \CC}
$$
is self-dual. 
\end{example}

\subsection{Some remarks on support}
\label{sub:knots}

We now explore what the endpoints of a braid $\sigma$ say about the sheaf $\CB_\sigma$ on $\FH_n^\dg$. For any braid $\sigma$, let $w_\sigma \in S_n$ denote the underlying permutation. 

\begin{proposition}(e.g. \cite[Proposition 2.16]{Hog})
For any braid $\sigma$ and for all $i\in \{1,...,n\}$, the left action of $x_i$ on the complex $\sigma \in K^b(\SBim_n)$ is homotopic to the right action of $x_{w_{\sigma(i)}}$.
\end{proposition}

In short, we will say that the left action $R \curvearrowright \sigma$ is homotopic to the right action $\sigma \curvearrowleft_{w(\sigma)} R$, {\em twisted} by the permutation $w_{\sigma}$. As a consequence, we obtain the following result:

\begin{corollary}
\label{cor:support rouquier}
The $R$--module $\RHom_{K^b(\SBim_n)} (\1,\sigma)$ is supported on the subspace:
$$
\left\{x_i=x_{w_{\sigma}(i)},\ i=1,\ldots,n\right\}\subset \CC^n.
$$
\end{corollary}

Our construction of Conjecture \ref{conj:1} is predicated on the expectation that: 
$$
\Hom_{K^b(\SBim_n)}(\1,\sigma) = R\Gamma(\FH_n^\dg,\CB(\sigma))
$$
and that moreover $\CB(\sigma)$ can be reconstructed from the spaces $\Hom_{K^b(\SBim_n)}(\1,\sigma \cdot \prod_{i=1}^n L_i^{a_i})$ for all sequences of large enough natural numbers $(a_1,...,a_n)$. These Hom spaces in the category $\SBim_n$ are very hard to compute, and all we can say at this stage is that Corollary \ref{cor:support rouquier} still applies to them. Therefore, we obtain the following:

%For a general braid $\alpha$, our method does not give any concrete predictions of the sheaf $\CB(\alpha) := \iota_*(B(\alpha))$: as explained in Section \ref{sub:proof}, one would need to know the Khovanov-Rozansky homology of $\alpha \cdot \prod \FT_k^{a_k}$ for sufficiently many sequences $(a_1,...,a_n) \in \NN^n$. However, one can make several general observations about $\iota_*(\alpha)$ and, in some cases, describe the corresponding sheaf explicitly. For example, note that Corollary \ref{cor:support rouquier} (see also \cite{RasDiff}) immediately implies the following:

\begin{corollary}
\label{cor:support sheaves}
The complex $\CB(\sigma) = \iota_*(\sigma)$ is supported on the subvariety:
$$
\FH^\edg_w := \rho^{-1} \left(\left\{x_i=x_{w_{\sigma}(i)},\ i=1,\ldots,n\right\} \right)\subset \FH_n^\edg = \FH_n^\edg(\CC)
$$
where $\rho:\FH_n^\edg(\CC) \rightarrow \CC^n$ is the map that records the eigenvalues $(x_1,...,x_n)$, akin to \eqref{eqn:rho}. 
\end{corollary}

\begin{corollary}
Suppose that the closure of $\sigma$ is connected. Then $\CB(\sigma)$ is supported on 
$$
\rho^{-1} \left( \left\{x_1=\ldots=x_n \right\} \right)=\FH_n(\point)\times \CC.
$$
\end{corollary}

\begin{remark}
Following Section \ref{sub:reduced}, one can prove that if the closure of $\sigma$ is connected, then the sheaf $\CB(\sigma)$ fibers trivially over $\CC$, i.e.:
$$
\CB(\sigma) = \overline{\CB(\sigma)}\boxtimes \CO_{\CC}
$$
for some sheaf $\overline{\CB(\sigma)}\in D^b\coh(\FH_n(\point))$. Since $\FH_n(\point)$ is projective, the cohomology of this sheaf is expected to be finite-dimensional. Moreover, our conjectures imply the fact that this cohomology matches the {\em reduced} Khovanov-Rozansky homology of $\alpha$.
\end{remark}

In general, $\FH_w^\dg$ may be quite complicated. However,
for certain permutations $w=w_\sigma$ we can describe it explicitly. The baby case is when $w = (j,j+1)$ is a transposition.

\begin{definition}
Define the dg subscheme $Z_j \subset \FH_n^{\dg}$ by the following equation:
\begin{equation}
\label{eqn:z periodic}
\CO_{Z_j} := \left[\ldots \longrightarrow \frac {q^2t^2\CL_j^2}{\CL_{j+1}^{2}} \xrightarrow{y_{j,j+1}} \frac {q^2t\CL_j}{\CL_{j+1}} \xrightarrow{x_j-x_{j+1}} \frac {qt\CL_j}{\CL_{j+1}} \xrightarrow{y_{j,j+1}}  q\CO \xrightarrow{x_j-x_{j+1}} \CO \right]. 
\end{equation}
Here $y_{j,j+1}:t\CL_j\to \CL_{j+1}$ is the map of line bundles induced by the homonymous coefficient of the matrix $Y$ in \eqref{eqn:triple matrices}, and the fact that $y_{j,j+1}(x_j-x_{j+1})=0$ follows from $[X,Y]=0$.

\end{definition}

\begin{remark}
Formula \eqref{eqn:z periodic} implies the following exact sequence:
\begin{equation}
\label{eqn:Z recursion}
\left[ q\CO\xrightarrow{x_j-x_{j+1}}\CO \right] \cong \left[\CO_{Z_j} \xrightarrow{\text{Id}} \frac {qt\CL_j}{\CL_{j+1}} \otimes \CO_{Z_j} [2] \right]
\end{equation}
\end{remark}

Our motivation for defining $Z_j$ is the fact that:
\begin{equation}
\label{eqn:transposition}
\CO_{\FH_{(j,j+1)}^\dg} = \CO_{Z_j}
\end{equation}
for all $j\in \{1,...,n-1\}$. The following proposition follows directly by iterating \eqref{eqn:transposition}.

\begin{proposition}
\label{prop:fh eqiv}
Suppose that $w$ has cycle structure:
$$
(1,...,k_1)(k_1+1,...,k_2),...,(k_{r}+1,...,n)
$$
for some sequence $0<k_1<\ldots<k_{r}<n$. Then the dg structure sheaf of $\FH_w^\edg$ has the following periodic resolution by locally free sheaves on $\FH_n^{\edg}$:
\begin{multline}
\label{eq:mf}
\CO_{\FH_w^{\edg}} \cong \bigotimes_{j\notin \{k_1,...,k_r\}} \left[\ldots \longrightarrow \frac {q^2t\CL_j}{\CL_{j+1}} \xrightarrow{x_j-x_{j+1}} \frac {qt\CL_j}{\CL_{j+1}} \xrightarrow{y_{j,j+1}}  q\CO \xrightarrow{x_j-x_{j+1}} \CO \right].
\end{multline}
\end{proposition}

\begin{conjecture}
\label{conj:knots}
Suppose that $\alpha = \prod_{i=0}^{r}(\sigma_{k_i+1}\cdots\sigma_{k_{i+1}-1})$ is a subword of the Coxeter word $\sigma_1\cdots\sigma_{n-1}$, for any sequence $0<k_1<\ldots<k_{r}<n$ as in Proposition \ref{prop:fh eqiv}. Then:
$$
\CB(\alpha) = \CO_{\FH_w^{\edg}}.
$$
\end{conjecture}

\begin{example}
\label{ex:coxeter}
For $\alpha=1$, the conjecture simply reads $\CB(\alpha) = \CO_{\FH_n^{\dg}}$, as prescribed by Conjecture \ref{conj:1}. For $\alpha = \sigma_1\cdots\sigma_{n-1}$, the conjecture reads $\CB(\alpha) = \CO_{\FH_n^{\dg}(\point)\times \CC}$.
\end{example}

Conjecture \ref{conj:knots} gives a full description of $\CB(\alpha)$ for all braids $\alpha$ on two strands (see Section \ref{sec:n=2} for the explicit construction in this case). Moreover, it completely describes $\CB(\alpha)$ for the braids $\alpha = 1,s_1,s_2,s_1s_2$ on 3 strands, multiplied by arbitrary powers of the twists $\FT_2,\FT_3$. Building upon this, the following conjecture supersedes the main conjecture of \cite{GN}, and it serves as one of the motivating examples of the present work:

\begin{conjecture}
\label{conj:torus knots}
For $\text{gcd}(m,n)=1$, consider the torus braid $\alpha_{n,m}=(\sigma_1\cdots\sigma_{n-1})^m$. Then
\begin{equation}
\label{eq:torus knots}
\CB(\alpha_{m,n}) = \left(\bigotimes_{i=1}^n \CL_i^{\left\lfloor\frac{im}{n}\right\rfloor-\left\lfloor\frac{(i-1)m}{n}\right\rfloor}\right) \otimes \CO_{\FH_n^\edg(\point)\times \CC}
\end{equation}
\end{conjecture}

See Sections \ref{sec:n=2} and \ref{sec:n=3} for detailed computations for two and three-strand torus braids.

\begin{remark}
It was proved in \cite{GN} that the equivariant Euler characteristic of the right hand side of \eqref{eq:torus knots} is equal to the ``refined Chern-Simons invariant'' defined by Aganagic-Shakirov \cite{AS} and Cherednik \cite{Ch}. One can therefore consider Conjecture \ref{conj:torus knots} as a categorification of the conjectures in \cite{AS,Ch} relating the Poincar\'e polynomial of Khovanov-Rozansky homology to these ``refined invariants''.
\end{remark}

%\subsection{Skein relations in the Soergel category}
%\label{sub:skein}

%It is well-known that the HOMFLY-PT polynomials of arbitrary knots can be computed via the skein relations:
%$$
%\text{picture of usual skein}
%$$
%When \cite{GN} was being written, the authors observed that the same thing is true of the three variable Poincar\'e polynomials $HHH(\text{torus knot})$, if one amended the skein relation to:
%\begin{equation}
%\label{eqn:skein1}
%\text{picture of skein on the same component}
%\end{equation}
%\begin{equation}
%\label{eqn:skein2}
%\text{picture of skein on different components}
%\end{equation}
%Is is not hard to see that the skein relations \eqref{eqn:skein1}--\eqref{eqn:skein2} do not define a topological invariant for arbitrary knots, since they depend on the order in which the crossings are resolved. However, we believe that they do make sense for positive braids, in the following sense. As was also observed by \cite{EH}, the skein exact sequences for crossings of braids come with different monomial weights. For example, we have an exact sequence:
%$$
%R \longrightarrow B_{\sigma} \longrightarrow qt B_{\sigma^{-1}} \longrightarrow qt R
%$$
%in the category $\SBim_2$, which categorifies the skein relation \eqref{eqn:skein1}. At the same time, the relation in Figure \ref{fig:skein} gives rise to an exact sequence:

\section{Categories and schemes}
\label{sec:geometry}

\subsection{Motivation: maps to projective space}
\label{sub:linear}

We start by recalling certain classical constructions in algebraic geometry which will guide all subsequent generalizations. Let $X$ be a projective algebraic variety and let $\CL$ be a line bundle (i.e. a rank one locally free sheaf) over $X$. One says that $\CL$ is generated by global sections if the map of sheaves:
$$
\CO_X \otimes \Gamma(X, \CL) \rightarrow \CL 
$$
is surjective. If we choose a basis $s_0,...,s_n$ of the vector space $\Gamma(X,\CL)$, this comes down to requiring that any local section of $\CL$ is a linear combination of the sections $s_0,...,s_n$. Moreover, the above datum gives rise to a map:
\begin{equation}
\label{eqn:embedding}
X \stackrel{\iota}\rightarrow \PP^n, \qquad x \mapsto [s_0(x):...:s_n(x)]
\end{equation}
Global generation implies the fact that the sections $s_0,...,s_n$ cannot all vanish simultaneously. Moreover, while $s_i$ are sections of the line bundle $\CL$, their ratios are well-defined functions on $X$. To this end, we may define the open subset: 
$$
X_i =\{s_i(x)\neq 0\} \subset X
$$
where the ratios $s_j/s_i$ are well--defined. Hence the map \eqref{eqn:embedding} restricts to a map:
$$
X_i \rightarrow U_i = \{z_i \neq 0\} \subset \PP^n
$$
If we let $\CO(1)$ denote the tautological line bundle on $\PP^n$, then we have:
$$
\iota^*(\CO(k)) = \CL^{\otimes k}, \quad \forall \ k\in \BZ
$$ 
The functor $\iota^*$ is monoidal, and is the left adjoint of the direct image functor:
\begin{equation}
\label{eqn:embedding2}
\coh(X) \xtofrom[\iota^*]{\iota_*} \coh(\PP^n) 
\end{equation}
In the remainder of this section, we present a generalization of this construction, where the role of the map $\iota:X \rightarrow \PP^n$ is replaced by an abstract categorical setup inspired by \eqref{eqn:embedding2}. 

%\begin{theorem}
%Let $X$ be a projective algebraic variety and let $L$ be a line bundle over $X$ generated by global sections.
%Then there exists a pair of adjoint derived functors 
%$$
%D^b\coh(X) \xtofrom[f^*]{f_*} D^b\coh(\PP^n) 
%$$
%such that $f^*(\CO(1))=L$.
%\end{theorem}

\begin{remark}
\label{rem:prod cones}

By deriving the functors in question, we may write \eqref{eqn:embedding2} at the level of derived categories. Then the sections can be thought of as complexes:
$$
\left[ \CO_X \stackrel{s_i}\rightarrow \CL \right] \in D^b(\coh(X))
$$
which are supported on $\{ X\setminus X_i \} = \{s_i=0\}$. The product of these complexes: 
\begin{equation}
\label{eqn:tensor}
\bigotimes_{i=0}^{n} \left[ \CO_X \stackrel{s_i}\rightarrow \CL \right]
\end{equation}
is therefore supported on the set where all $s_i$ vanish simultaneously, which by assumption is the empty set. Therefore, \eqref{eqn:tensor} is quasi-isomorphic to 0, and hence it vanishes in $D^b(\coh(X))$. Put differently, the vanishing of \eqref{eqn:tensor} is forced upon us by the vanishing of the Koszul complex:
$$
\bigotimes_{i=0}^{n} \left[ \CO_{\PP^n} \stackrel{z_i}\rightarrow \CO_{\PP^n}(1) \right] \stackrel{\qis}\cong 0 \ \in \ D^b(\coh(\PP^n))
$$
and the fact that the derived version of the functor $\iota^*$ in \eqref{eqn:embedding2} is monoidal.

\end{remark}

\begin{remark} Projective space can be defined more scheme-theoretically as: 
$$
\PP^n = \proj\left(\bigoplus_{k=0}^{\infty} S^k\CC^{n+1} \right)
$$
Then the map \eqref{eqn:embedding} is given by the map %isomorphism EG: it is not an isomorphism!
$\CC^{n+1} \to \Gamma(X,\CL)$ induced by the choice of the sections $s_0,...,s_n$, and in fact global generation translates into:
$$
X = \proj\left(\bigoplus_{k=0}^{\infty} \Gamma(X,\CL^{\otimes k}) \right).
$$

\end{remark}

\subsection{Notations for categories}
\label{def:cat}

In this subsection, we would like to collect all homological algebra notations, definitions and assumptions which will be frequently used below. Let $\cat$ be an additive unital monoidal category with tensor product $\otimes$ and direct sum $\oplus$. The monoidal structure is not necessary symmetric. We will denote the unit object of $\cat$ by $\1_{\cat}$, or $\1$ if the category is clear from context. The endomorphism algebra $\End(\1)$ is always commutative, and we assume that it is Noetherian. For any object $A\in \cat$, the morphism space $\Hom(\1,A)$ is a module over  $\End(\1)$, and we assume that it is finitely generated. We assume that all morphism spaces are positively graded.  We denote by $K^b(\cat)$ the homotopy category of bounded complexes of objects in $\cat$ and by $K^-(\cat)$ the 
homotopy category of bounded above complexes. Unless stated otherwise, we will work with bounded above complexes and abbreviate $K^-(\cat)$ to $K(\cat)$. 

We will consider two types of ``semi-infinite completions" of the category $\cat$. The first type is the homotopy category $K^{-}(\cat)$ of bounded above complexes of objects in $\cat$ (which is well-known to also be a monoidal category). The other type is the category of certain infinite sums of objects in $\cat$, as in the following definition.

\begin{definition}
\label{def:completion}

Assume that $\cat$ is graded, and the grading shift is denoted by $A \mapsto A(1)$. We define its \textbf{graded completion} $\cat^{\uparrow}$ as follows. The objects are given by countable direct sums:
$$
\text{Ob}(\cat^\uparrow) = \left\{ \bigoplus_{i=-\infty}^N A_i(i) \text{ for some }N \in \BZ \right \}
$$ 
and the morphisms $\phi: \oplus A_i(i)\rightarrow \oplus B_j(j)$ are collections of arrows $\{\phi_{ij}:A_i(i)\to B_j(j) \}$ for all $i,j$, such that for each $i$ there are only finitely many $j$ such that $\phi_{ij}\neq 0$. %EG: Is this condition on morphisms automatic? %A: I don't think so.
\end{definition}

One can check that $\cat^{\uparrow}$ and $K^{-}(\cat^{\uparrow})$ inherit the tensor product from $\cat$. Note that $K^{-}(\cat^{\uparrow})$ is endowed with both the grading $(1)$ and the homological degree $[1]$. %that both gradings are bounded by above and in each bidegree one has a finite sum of objects from $\cat$. A: I'm not sure what this statement means

%A: Why do we need K^-(\cat^\uparrow)? We need the derived category of this guy, but I'm just wondering why K^- and not K^b.

Note that the category $\cat$ may have multiple gradings, and the notion of completion depends on a specific choice of grading among these. For example, if $\cat$ is graded by $\mathbb{Z}^r$, this accounts to choosing a one-dimensional direction in $\mathbb{Z}^r$. To clarify homological algebra over $\cat^{\uparrow}$, we present some examples.

\begin{example}
Let $\cat$ be the category of graded finitely generated $\CC[x]$-modules.
Consider the following two-term complex in $K^{-}(\cat^{\uparrow})$:

\begin{figure}[ht!]
%\begin{center}
\includegraphics{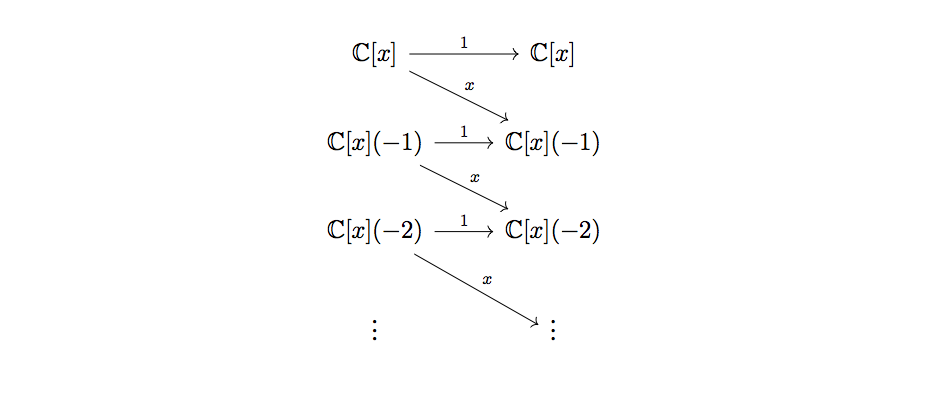}
%\end{center}
\caption{}
\label{fig: model complex}
\end{figure}

We can introduce an auxilary variable $y$ of degree $(-1)$ and rewrite the complex as following:
$$\CC[x,y] \xrightarrow{1+xy}   \CC[x,y].$$
At first glance, one could think that since all horizontal arrows in Figure \ref{fig: model complex} are isomorphisms, the complex is contractible. However, this is not the case, since a homotopy would be:
$$
\CC[x,y]\xleftarrow{H}\CC[x,y] \text{ such that } H(1+xy)=(1+xy)H=1 
$$
A natural choice for $H$ would be:
$$
H(x,y)=\frac{1}{1+xy}=1-xy+x^2y^2-x^3y^3+\ldots,
$$
but this is not a valid morphism in $\cat^{\uparrow}$ since there would be non-zero arrows from the top-most copy of $\CC[x]$ to all infinitely many copies below it.

\end{example}

\begin{remark}
One can check that the homology of the complex in Figure \ref{fig: model complex} is isomorphic to $\CC[x,y]/(1+xy)=\CC[x,x^{-1}]$.
\end{remark}

%\textcolor{red}{EG: the following definition is taken from section 7.3, please adapt it to the above notations}

%\begin{definition}
%\label{def:homological}
%Suppose $\cat$ is a $T$--equivariant category, all of whose Hom spaces are concentrated in finitely many degrees with respect to a distinguished subtorus $\CC^* \subset T$. We define its {\bf homological completion} $\bcat$, whose objects are all complexes of objects in $\cat$ of the form:
%$$
%\left(\bigoplus_{k=0}^{\infty} C_k, \text{with differentials } d_{kl} : C_k \rightarrow C_{l} \  \begin{cases} \text{which are 0 for }k\geq l, \qquad \qquad \text{ and which} \\ 
%\text{have homological degree} \leq l-k \text{ otherwise} \end{cases} \right)
%$$ 
%$$
%\left(\bigoplus_{k=0}^{\infty} C_k, \text{with differentials } d_{lk} : C_l \rightarrow C_{k} \  \begin{cases} \text{which are 0 for }k\geq l, \qquad \qquad \text{ and which} \\ 
%\text{have homological degree} \geq l-k \text{ otherwise} \end{cases} \right)
%$$ 
%\end{definition}

 \subsection{Categories over schemes}
\label{sub:spec}

In this section, we will develop a general setup relating a category $\cat$ with a scheme $X$, with the goal of reducing Conjecture \ref{conj:1} to Conjecture \ref{conj:main}. Though we will not always say this explicitly, $X$ should be thought of as a dg scheme. 

%\begin{equation}
%\label{eqn:coh perf}
%D^-(\coh(X)) \ := \ D^-(\text{Perf}(X)) = \text{the derived category of perfect complexes}
%\end{equation}

\begin{definition}
\label{def:morphism}

A \textbf{morphism} from the category $\cat$ to the scheme $X$, written as:
$$
\cat \stackrel{\iota}\longrightarrow X
$$
consists of a pair of functors:
\begin{equation}
\label{eqn:functors}
\cat \xtofrom[\iota^*]{\iota_*} \coh(X) 
\end{equation}
such that:

\begin{itemize}

\item $\iota^*$ is a monoidal functor 

\item $\iota_*$ is the right adjoint of $\iota^*$ 

\item the following \textbf{projection formula} holds:
\begin{equation}
\label{eqn:projection}
\iota_*(\iota^*M_1 \otimes C \otimes \iota^*M_2) = M_1 \otimes \iota_*(C) \otimes M_2
\end{equation}
for all $M_1,M_2 \in \coh(X)$ and $C \in \cat$. 

\end{itemize}

\end{definition}

The above definition is modeled on the situation when $\cat = \coh(Y)$ for a scheme $Y$, in which case the functors $\iota_*$ and $\iota^*$ play the roles of direct and inverse image functors associated to a map of schemes $\iota:Y \rightarrow X$.

\begin{definition}
\label{def:birational} 

We call the map $\cat \stackrel{\iota}\longrightarrow X$ \textbf{birational} if:
\begin{equation}
\label{eqn:birational}
\iota_* \1 = \CO_X
\end{equation}

\end{definition}

This terminology, albeit imprecise, is motivated by the important case when $\cat = \coh(Y)$ where $Y$ is endowed with a proper birational map to $X$. 

\begin{proposition}
\label{prop: proper hom}
Suppose that $\cat \stackrel{\iota}\longrightarrow X$ is birational. Then $\iota^*$ is fully faithful, and moreover:
\begin{equation}
\label{eqn:RGamma}
\Hom_{\cat}(\1,\iota^* M) = \Gamma(X,M)
\end{equation}
for all $M\in \coh(X)$.
\end{proposition}

\begin{proof} The adjunction implies that:
$$
\Hom_{\cat}(\iota^*M',\iota^*M) = \Hom_X(M',\iota_*\iota^*M) = \Hom_X(M',M)
$$
where the last equality follows from \eqref{eqn:projection} and \eqref{eqn:birational}. When $M'=\CO_X$ we obtain precisely \eqref{eqn:RGamma}.

\end{proof}

Most of the time we will consider a derived version of this construction.

\begin{definition}
 A \textbf{derived morphism} from the category $\cat$ to the scheme $X$, written as:
$$
\cat \stackrel{\iota}{\longrightarrow} X
$$
is a pair of  mutually adjoint functors:
\begin{equation}
\label{eqn:derfunctors}
K(\cat) \xtofrom[\iota^*]{\iota_*} D(\coh(X)) 
\end{equation}
All other properties and requirements remain unchanged. 
\end{definition}

%\begin{remark}
%Given a morphism of algebraic varieties $\iota:Y\to X$, one gets derived functors: 
%$$
%D(\coh( Y)) \xtofrom[L\iota^*]{R\iota_*} D(\coh(X))
%$$
%If $\iota$ is flat, then $L\iota^*=\iota^*$ is exact, and hence we can lift it to a functor $D(\coh (X))\stackrel{\iota^*}\longrightarrow  K(\coh (Y))$. 
%\end{remark}

%and the right hand side of Proposition \ref{prop: proper hom} should be replaced by $R\Gamma$. There is but one catch. To make sure the %functor $\iota^*$ is well-defined, one needs to make sure that it sends quasi-isomorphic complexes of coherent sheaves on $X$ to %isomorphic complexes (modulo homotopy) in $\cat$. 

\subsection{The affine case}
\label{sub:aff}

Let \(\cat\) be an additive monoidal category. Suppose we are given a Noetherian commutative ring $A$ and a ring homomorphism
\begin{equation}
\label{eqn:affine}
A \stackrel{f}\longrightarrow \End_\cat(\1)
\end{equation}
satisfying 
$$
(\#) \quad \Hom_\cat(\1,C) \text{ is finitely generated over }A
$$
for any object $C$ of $ \cat$. 
Then there is a derived morphism:
\begin{equation}
\label{eqn:affine0}
 \cat \stackrel{\iota}\longrightarrow  \spec \ A.
\end{equation}
The functors
$$
K(\cat) \xtofrom[\iota^*]{\iota_*} D(A\text{--mod})
$$
are defined as follows. There is a functor
\(i: \cat \to A \text{--mod}\)  given by:
\begin{equation}
\label{eqn:push0}
\iota_* (C) = \Hom_{\cat}(\1,C).
\end{equation}
This extends in the obvious way to a functor \(i: K(\cat) \to K(A \text{--mod})\), and \(\iota_*\) is defined to be the composition of \(i_*\) with the natural inclusion \( K(A \text{--mod}) \to D(A \text{--mod})\). 

In the other direction, let \(FA \text{--mod}\) be the category of finitely generated free \(A\) modules. 
The inclusion \( K(FA \text{--mod}) \to D(A \text{--mod})\) is an equivalence of categories, so we may as 
well give a functor \( \iota_* : K(FA \text{--mod}) \to K(\cat)\). %\textcolor{red}{JR: is this OK?} 
We define \(\iota_*\) by setting 
\(\iota_*(A) = \1\) and \(\iota_*(a) = f(a)\) for \(a \in A = \Hom(A,A)\). This extends to  \(K(FA \text{--mod})\) in 
the obvious way.  If \(M\) is an object of \(D(A \text{--mod})\), we write
\(\iota_*(M) = M \otimes_A \1\). 

Let us check that the functors \(\iota^*\) and \(\iota_*\) are adjoint, or equivalently, that 
\begin{equation}
\label{eqn:adj0}
\Hom_{K(\cat)}(M \otimes_A \1, C) = \Hom_{D(A\text{--mod})}(M, \Hom_{\cat}(\1,C))
\end{equation}
for all $M \in D(A\text{--mod})$ and $C\in K(\cat)$. If \(C \in \cat\), the right-hand side is by definition 
\(\Ext_A(M,\Hom_{\cat}(1,C))\). The statement that it is equal to the left-hand side reduces to the  well known fact that to compute \(\Ext\) of two modules, it is enough to take a free resolution of one of them. %\textcolor{red}{JR: anything different if \(C \in K(\cat)\)?} 
Properties \eqref{eqn:projection} and \eqref{eqn:birational} also follow directly from the definitions.

%\begin{example}
%Suppose \(\iota: Y \to \Spec A\) is a morphism of (?)schemes. Then in the sense of the previous section, we have a morphism \(\Coh(Y) \to \Spec A\). On the other hand, we have an obvious homomorphism 
%\(A \to \End_{\Coh(Y)} = \Gamma(Y,\CO_Y) \), so the construction above gives a derived morphism
%\end{example}

\begin{example}
%\textcolor{red}{JR:Any modifications needed here?}
Let $Y$ be an algebraic variety, and $\cat=\coh(Y)$. The unit in $Y$ is given by the structure sheaf $\CO_Y$, and indeed $\cat$ is a category over $\spec\ \End_{\cat}(\1)=\spec\ \Gamma(Y,\CO_Y)$. This structure is precisely equivalent with the global section map:
$$
\iota : Y\to \spec\ \Gamma(Y,\CO_Y)
 $$
More generally, a ring homomorphism $A\stackrel{f}\rightarrow \Gamma(Y,\CO_Y)$ corresponds to a map $\spec\ \Gamma(Y,\CO_Y)\to \spec\ A$, and one can use the composed map from $Y$ to $\spec \ A$ to define $\iota_*$ and $\iota^*$.
\end{example}

\subsection{The projective case}
\label{sub:proj}

%$$
%\begin{tikzcd}
%& F^k \otimes F^{k'} \arrow{r}{\alpha \otimes \alpha'} & F^l \otimes F^{l'} \arrow{rd}{=} & \\
%F^{k+k'} \arrow{ru}{=} \arrow{rd}{=} & & & F^{l+l'} \\
%& F^{k'} \otimes F^{k} \arrow{r}{\alpha' \otimes \alpha} & F^{l'} \otimes F^{l} \arrow{ru}{=} &
%\end{tikzcd}
%$$

%by which we mean an invertible object with the property that the following diagram commutes:
%\begin{equation}
%\label{eqn:commrhombus}
%\begin{tikzcd}
%& F^k \otimes F^{l} \arrow{rd}{=}  & \\
%\1 \arrow{ru}{\alpha \otimes \beta} \arrow{rd}[swap]{\beta \otimes \alpha} & & F^{k+l} \\
%& F^{l} \otimes F^{k} \arrow{ru}[swap]{=}  &
%\end{tikzcd}
%\end{equation}
%for all arrows $\alpha:\1 \rightarrow F^k$ and $\beta : \1 \rightarrow F^{l}$. As a consequence, the following algebra:
%\begin{equation}
%\label{eqn:commalg}
%\Hom_{\cat}(\1,F^{\bullet}) := \bigoplus_{k=0}^\infty \Hom_{\cat}(\1,F^k)
%\end{equation}
%is commutative.

In the previous Subsection, we showed that any category can be realized over the spectrum of the endomorphism ring of its unit. We may upgrade this construction if we are given an \textbf{invertible} object $F \in K(\cat)$, i.e. one which is endowed with isomorphisms:
\begin{equation}
\label{eqn:def invertible}
F \otimes F^{-1} \cong F^{-1} \otimes F \cong \1
\end{equation}
%In Subsection \ref{sub:invertible}, we will give a sufficient condition for an object to be invertible. %For any such object $F$, 

\begin{assumption}
\label{ass:commute}
We assume that the graded algebra:
\begin{equation}
\label{eqn:commalg}
\Hom_{K(\cat)}(\1,F^{\bullet}) := \bigoplus_{k=0}^\infty \Hom_{K(\cat)}(\1,F^k)
\end{equation}
is commutative. %Indeed, any arrow $\1 \stackrel{\alpha}\rightarrow F^k$ is isomorphic to an arrow:
%$$
%F^{-k} \stackrel{\alpha'}\longrightarrow \1
%$$
%The fact that this commutes with any arrow $\1 \stackrel{\beta}\rightarrow F^l$ is a consequence of the fact that any object naturally commutes with $\1$. 
\end{assumption}

\begin{remark}
Recall that $\cat$ was a graded category, so for every $k$ the space $\Hom_{K(\cat)}(\1,F^k)$ is graded. 
The algebra $\Hom_{K(\cat)}(\1,F^{\bullet})$ has an extra grading which equals $k$ on $\Hom_{K(\cat)}(\1,F^k)$.
\end{remark}

In this setting, there exists a tautological derived map:
\begin{equation}
\label{eqn:projective0}
\cat \stackrel{\iota}\longrightarrow (\spec \ R )/\CC^*
\end{equation}
for any Noetherian graded commutative ring $R$ and graded ring homomorphism: 
\begin{equation}
\label{eqn:projective}
R \stackrel{f}\longrightarrow  \Hom_{K(\cat)}(\1,F^{\bullet})
\end{equation}
The functors \eqref{eqn:functors} are explicitly given by:
\begin{equation}
\label{eqn:functors1}
\iota_*(C) = \Hom_{K(\cat)} \left(\1, F^{\bullet} \otimes C \right)
\end{equation}
\begin{equation}
\label{eqn:functors2}
{\iota}^*(M) = \left( M \otimes_R \bigoplus_{k=-\infty}^{\infty} F^k \right)^0
\end{equation}
for all graded $R-$modules $M$ and all $C \in K(\cat)$. The Hom space in \eqref{eqn:functors1} is an $R-$module via \eqref{eqn:projective}. It is straightforward to show that the analogue of \eqref{eqn:adj0} holds, and that the above datum makes $\cat$ into a category over the stack $(\spec \ R )/\CC^*$:
\begin{equation}
\label{eqn:yoy}
K(\cat) \xtofrom[\iota^*]{\iota_*}D(R\text{--grmod})
\end{equation}
Note that one needs the analogue of condition $(\#)$  on the category $\cat$ to ensure that the above functors are well-defined (in particular, that the right hand side of \eqref{eqn:functors1} is a finitely generated $R$-module). But given this, the map $\iota$ is birational if and only if the map $f$ of \eqref{eqn:projective} is an isomorphism. %The derived version of \eqref{eqn:yoy} is defined analogously:
%\begin{equation}
%\label{eqn:zoz}
%K^b(\cat) \xtofrom[\iota^*]{\iota_*}D^b(R\text{--grmod}) 
%\end{equation}
%but one must require that $\iota^*$ sends quasi-isomorphic complexes to isomorphic (modulo homotopy) objects in $\cat$.

%\begin{equation}
%\label{eqn:functors2}
%{\iota}^*\left( [ ... \rightarrow \oplus_{i=1}^n R(k_i) \stackrel{\phi}\longrightarrow \oplus_{j=1}^m R(l_j) \rightarrow ... ] \right) = [ ... \rightarrow \oplus_{i=1}^n F^{k_i} \stackrel{f\circ \phi}\longrightarrow \oplus_{j=1}^m F^{l_j} \rightarrow ... ]
%\end{equation}

\begin{example}
\label{ex:projective} 

Let us consider the case where $R = A[z_0,...,z_n]$, for a ring $A$ equipped with a homomorphism $A \rightarrow \End_\cat(\1)$. Then the datum of the homomorphism \eqref{eqn:projective} boils down to giving $n+1$ morphisms:
\begin{equation}
%\label{eqn:datum}
z_i \quad \rightsquigarrow \quad \Big\{ \1 \xrightarrow{\alpha_i} F \Big\}_{i=0,...,n}
\end{equation}
This makes $\cat$ into a category over the stack:
$
\cat \xrightarrow{\iota} \AA^{n+1}_A/\CC^*.
$
The natural question is when does $\iota$ factor through projective space:
$$
\includegraphics{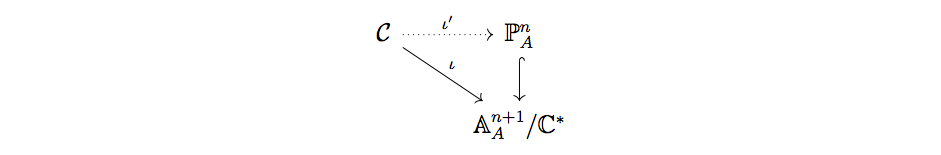}
$$
which amounts to factoring \eqref{eqn:yoy} through functors:
\begin{equation}
\label{eqn:func}
K(\cat) \xtofrom[{\iota'}^*]{\iota'_*} D(\coh(\PP^n_A)) 
\end{equation}
It is clear that ${\iota'}^*$ and $\iota'_*$ must be given by the same formulas as in \eqref{eqn:functors1}--\eqref{eqn:functors2}, but one needs to impose a certain relation. Because the zero section of $\AA^{n+1}_A/\CC^*$ is removed when defining projective space, the structure sheaf of the zero section becomes quasi-isomorphic to 0. Since this structure sheaf can be expressed via the following Koszul complex:
$$
\left[ ... \longrightarrow \CO(-2)^{\oplus \binom{n+1}{2}} \longrightarrow \CO(-1)^{\oplus \binom{n+1}{1}} \xrightarrow{(z_0,...,z_n)} \CO \right] = \bigotimes_{i=0}^n \left[\CO(-1) \xrightarrow{z_i} \CO \right]
$$
we conclude that the functors \eqref{eqn:func} are well-defined only if:
\begin{equation}
\label{eqn:conecondition}
\left[ \1 \xrightarrow{\alpha_0}  F \right] \otimes ... \otimes \left[ \1 \xrightarrow{\alpha_n} F \right] \stackrel{\he}\cong 0 \ \in K(\cat).
\end{equation}
%\textcolor{red}{JR: to be compatible with Matt+Ben's definition, should we assume that \(\cat\) is triangulated, and this equation holds in \(\cat\), rather than in \(K^b(\cat)\)?}

It is not hard to see that this condition is also sufficient, by invoking Beilinson's description \cite{Be,BGG} of the derived category of projective space as equivalent to the homotopy category of complexes of finite direct sums of free $A[x_0,...,x_n]$--modules with degree shifts $\in \{0,...,n\}$.

\end{example}

\begin{remark}
If $F = \CL$ is a line bundle in $\cat=\coh(X)$, then $\alpha_i$ are nothing but sections of $\CL$.
By Remark \ref{rem:prod cones}, equation \eqref{eqn:conecondition} is equivalent to the fact that $\alpha_i$ generate $\CL$, and indeed this is a necessary and sufficient condition for the existence of $X \rightarrow \PP^n$, as we saw in Subsection \ref{sub:linear}.
\end{remark}

\subsection{The relative case}
\label{sub:relative}

The situation of Example \ref{ex:projective} captures a very interesting problem, namely when can we factor a map from a category to a scheme through another scheme:
\begin{equation}
\label{eqn:iota0}
\includegraphics{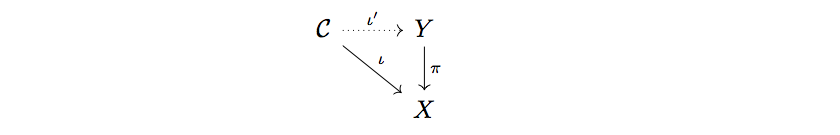}
\end{equation}
More precisely,  $\iota'$ should satisfy the equations $\iota^*=\iota'^*\circ \pi^*$ and  $\iota_*=\pi_*\circ \iota'_*$ and all the functors should be derived from now on. The situation we will study in this paper is when:
$$
Y = \PP \CV^{\vee} := \proj_X (S^{\bullet}\CV)
$$
where $\CV$ is a coherent sheaf on $X$ of projective dimension 0 or 1. Let us first study the case of projective dimension zero, so assume that $\CV$ is a vector bundle.  

%and assume that the map $\iota$ in \eqref{eqn:iota} is constructed. Then we may ask what is needed to extend the map $\iota$ to $\iota'$ as in \eqref{eqn:iota}.

%We say that $F \in \cat$ is \textbf{locally free with respect to} $X$ if it is invertible, and \eqref{eqn:commrhombus} holds for any:
%$$
%\alpha: \iota^*(\CE) \rightarrow F^k, \quad \beta: \iota^*(\CE') \rightarrow F^l \qquad \forall \ \CE,\CE' \in \coh(X)
%$$

\begin{proposition}
\label{prop:lift0}
Suppose that $Y=\PP \CV^{\vee}$ and that the map $\iota$ in \eqref{eqn:iota0} is constructed. The datum of the extension $\iota'$ is equivalent to an invertible object $F \in \cat$ together with an arrow:
\begin{equation}
\label{eqn:datum0}
\iota^*\CV \stackrel{\alpha}\longrightarrow F
\end{equation}
in $\cat$. This gives $\cat$ the structure of a category over $Y$ if and only if:
\begin{equation}
\label{eqn:koszul0}
\left[ ... \stackrel{\alpha}\longrightarrow \iota^*\left(\wedge^k \CV\right) \otimes F^{-k} \stackrel{\alpha}\longrightarrow ... \right] \stackrel{\ehe}\cong 0 \ \in \ K^b(\cat)
\end{equation}
The map $\iota'$ is birational if and only if $\iota$ satisfies:
\begin{equation}
\label{eqn:properproj0}
S^k\CV \cong \iota_*(F^k) \qquad \forall \ k\geq 0
\end{equation}

\end{proposition}

\begin{proof} All notations $\CO$ or $\CO(k)$ will refer to invertible sheaves on $\PP \CV^\vee$. If $\iota'$ exists and has all the expected properties, then set $F = {\iota'}^*(\CO(1))$. In this case, the map \eqref{eqn:datum0} is simply ${\iota'}^*$ applied to the tautological morphism:
$$
\pi^*\CV \longrightarrow \CO(1)
$$
on $Y$. The fact that the complex \eqref{eqn:koszul0} is quasi-isomorphic to 0 follows by applying ${\iota'}^*$ to the Koszul complex of $Y$. The birationality of $\iota'$ implies that $\iota'_*\1 \cong \CO$, from which the projection formula implies $\iota'_*(F^k) \cong \CO(k)$. Applying $\pi_*$ to this relation implies precisely \eqref{eqn:properproj0}. Conversely, suppose that we are given a morphism \eqref{eqn:datum0} which satisfies \eqref{eqn:koszul0}, and let us construct the map $\iota'$ that makes the diagram \eqref{eqn:iota0} commute. Note that \eqref{eqn:datum0} gives us an arrow:
$$
\iota^*\left(\CV^{\otimes k} \right) \longrightarrow F^k
$$
for all $k\geq 0$. Because $F$ is invertible, this arrow factors through:
\begin{equation}
\label{eqn:algebras0}
\iota^*\left(S^k\CV \right) \longrightarrow F^k
\end{equation}
for all $k\geq 0$ (since $F$ is invertible, so is $F^k$, and hence has no nontrivial endomorphisms; this implies that the anti-symmetric projector is zero, hence $S^kF = F^k)$. This allows us to define:
$$
{\iota'}^*(M) = \left(\pi_*(M) \bigotimes_{S^{\bullet}\CV} \bigoplus_{k=-\infty}^{\infty} F^k \right)^0
$$
A priori, this only determines the functor ${\iota'}^*$ on the level of the homotopy category of coherent sheaves on $\PP\CV^\vee$. To check that it descends to a functor on the derived category, we must show that ${\iota'}^*$ takes quasi-isomorphic complexes to isomorphic complexes. The fact that this statement is true for the Koszul complex is precisely the assumption \eqref{eqn:koszul0}. The fact that this is sufficient is due to Theorem 2.10 of \cite{DanHL} (see also \cite{BFK}), which asserts that:
$$
D^b \left(\coh(\PP \CV^\vee ) \right) \cong \text{homotopy category of complexes of} \left( \bigoplus_{i=0}^{\text{rank }\CV-1} \CE_i(i) \right)_{\CE_0,\CE_1,... \in D^b(\coh(X))} 
$$ 
As for the right adjoint functor, we set:
$$
\iota'_*(C) = \iota_* \left( \bigoplus_{k=0}^\infty F^k \otimes C \right)
$$
as a graded $\CO_X-$module. To realize the right hand side as a sheaf on $Y$, we need to endow it with an action of $S^*\CV$, namely with an associative homomorphism of graded algebras:
$$
S^*\CV \otimes_{\CO_X} \iota_* \left( \bigoplus_{k=0}^\infty F^k \otimes C \right) \longrightarrow \iota_* \left( \bigoplus_{k=0}^\infty F^k \otimes C \right)
$$
The above morphism is obtained via adjunction and \eqref{eqn:algebras0}. 
\end{proof}

\subsection{Projective dimension one}
\label{sub:one}

For the setting of this paper, we will need a version of Proposition \ref{prop:lift0} when the vector bundle $\CV$ is replaced by the quotient:
$$
0 \longrightarrow\CW \stackrel{\psi}\longrightarrow \CV \longrightarrow \CQ \longrightarrow 0
$$ 
where $\CW$ is another vector bundle. More precisely, we are interested in the case when:
$$
Y \hookrightarrow \PP \CV^\vee
$$
is the (derived) zero locus of the section:
\begin{equation}
\label{eqn:section}
s : \pi^*(\CW) \stackrel{\psi}\longrightarrow \pi^*(\CV) \longrightarrow \CO(1)
\end{equation}
where $\pi$ is the map in the following diagram:
\begin{equation}
\label{eqn:iota}
\includegraphics{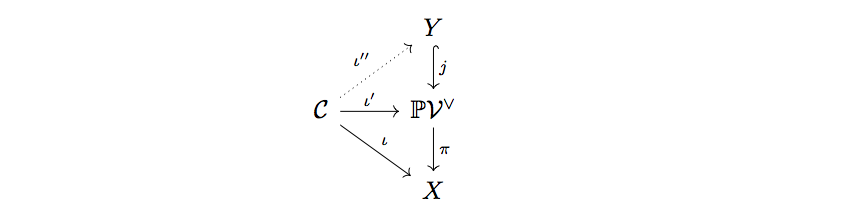}
\end{equation}
To simplify the geometry, we make the following very important assumption:
\begin{equation}
\label{eqn:avramov}
\text{the ideal of }Y \stackrel{j}\hookrightarrow \PP\CV^\vee \text{ is generated by a regular sequence in Im }s
\end{equation}
which entails that the embedding $\psi$ cuts out $Y$ as a complete intersection in $\PP\CV^\vee$. One could do without this assumption, but that would require one to replace $Y$ with the dg scheme determined by the exterior power of the section $s$. In other words, we must require the following quasi-isomorphism in the derived category of $\PP \CV^\vee$:
\begin{equation}
\label{eqn:rivers}
\CO_Y \cong \left[ ... \stackrel{s}\longrightarrow \wedge^k \pi^*(\CW) \otimes \CO(-k) \stackrel{s}\longrightarrow ... \stackrel{s}\longrightarrow \CO \right]
\end{equation}
In order to construct the lift $\iota''$ in \eqref{eqn:iota}, we must first construct the arrow $\iota'$, and for this we invoke Proposition \ref{prop:lift0}. Then the following Proposition says precisely when the arrow $\iota'$ thus defined factors through $Y$. 

\begin{proposition}
\label{prop:lift}
Suppose that $Y \stackrel{j}\hookrightarrow \PP \CV^{\vee}$ as in \eqref{eqn:iota} and that the map $\iota$ is constructed. The datum of the extension $\iota''$ is equivalent to an invertible object $F \in \cat$ together with an arrow:
\begin{equation}
\label{eqn:datum}
\iota^* \CQ \stackrel{\beta}\longrightarrow F
\end{equation}
in $\cat$. This gives $\cat$ the structure of a category over $Y$ if and only if:
\begin{equation}
\label{eqn:koszul}
\left[ ... \stackrel{\beta}\longrightarrow \iota^*\left(\wedge^k \CQ \right) \otimes F^{-k} \stackrel{\beta}\longrightarrow ... \right] \stackrel{\ehe}\cong 0
\end{equation}
The map $\iota''$ is birational if and only if $\iota_*$ gives rise to an isomorphism:
\begin{equation}
\label{eqn:properproj}
S^k \CQ \cong \iota_*(F^k) \qquad \forall \ k\geq 0
\end{equation}
\end{proposition}

Note that if we interpret $Y$ as a dg scheme whose structure sheaf is the dg algebra in the right hand side of \eqref{eqn:rivers}, we must replace $\CQ$ in \eqref{eqn:datum}, \eqref{eqn:koszul},  \eqref{eqn:properproj} with the two term complex $[\CW\rightarrow\CV]$. Making sense of the symmetric and exterior powers of such a complex is rather straightforward homological algebra, which we relegate to the Appendix.

\begin{proof} As we have seen in Proposition \ref{prop:lift0}, the existence of a monoidal functor:
$$
{\iota'}^*: D^b(\coh(\PP \CV^\vee)) \rightarrow K^b(\cat)
$$
implies the datum of an invertible object $F \in \cat$ (the image of $\CO(1)$) together with an arrow $\iota^*\CV \rightarrow F$ in $\cat$ (the image of the tautological morphism). The question is when does the functor ${\iota'}^*$ factor through:
$$
D^b(\coh(\PP \CV^\vee)) \stackrel{j^*}\longrightarrow D^b(\coh(Y))
$$
$$
M \mapsto M \otimes_{\CO_{\PP \CV^\vee}} \CO_Y =  \left[ ... \stackrel{s}\longrightarrow \wedge^k \pi^*(\CW) \otimes M(-k) \stackrel{s}\longrightarrow ... \stackrel{s}\longrightarrow M \right]
$$
where in the last equality we have used the assumption \eqref{eqn:avramov}. In particular, we have: 
$$
\includegraphics{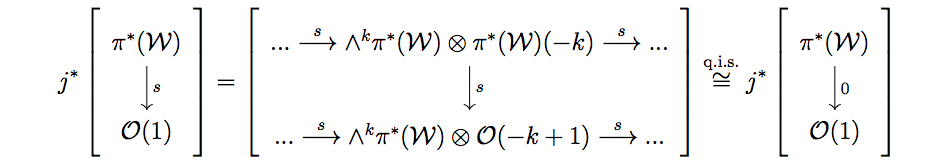}
$$
This implies that the functor ${\iota''}^*$ must take the composition $\pi^*(\CW) \stackrel{\psi}\hookrightarrow \pi^*(\CV) \rightarrow \CO(1)$ to zero, and hence the map $\alpha$ of \eqref{eqn:datum0} must factor through a map $\beta$ as in \eqref{eqn:datum}. Sending the Koszul complex of $\beta$ though the functor $j^*$ gives rise to the Koszul complex of $\alpha$, which must be sent to 0 by \eqref{eqn:koszul0}. Therefore, we conclude that the existence of the extension ${\iota''}^*$ requires \eqref{eqn:koszul}. Finally, recall that being birational is equivalent to $\iota''_*\1 \cong \CO_Y$. The projection formula implies that $\iota''_*(F^\bullet) \cong \CO_Y(k)$, and applying $j_*$ to this isomorphism yields:
$$
\iota'_*(F^\bullet) \cong \left[ ... \stackrel{s}\longrightarrow \wedge^k \pi^*(\CW) \otimes \CO(\bullet-k) \stackrel{s}\longrightarrow ... \right]
$$
Applying $\pi_*$ to the above isomorphism implies:
\begin{equation}
\label{eqn:luch}
\iota_*(F^\bullet) \cong \Big[ \wedge\CW \otimes S \ \CV \Big]^\bullet 
\end{equation}
where the differential in the right hand side of \eqref{eqn:luch} is given by the map $\psi:\CW \rightarrow \CV$. As in Example \ref{ex:important}, the right hand side is a resolution of $S^\bullet \CQ$, hence we obtain \eqref{eqn:properproj}.

\end{proof}

\subsection{Deducing Conjecture \ref{conj:1} from Conjecture \ref{conj:main}}
\label{sub:proof} 

The categorical setup presented in this section allows one to deduce the main conjecture from Conjecture \ref{conj:main} (a)--(c). We will proceed by induction on $n$, so let assume that the functors \eqref{eqn:funk} are well-defined for some fixed $n$. Our task is to construct functors:
$$
K^b(\SBim_{n+1}) \xtofrom[{\iota}_{n+1}^*]{\iota_{n+1*}}   D^b(\coh(\FH^\dg_{n+1}))
$$
given the functors:
$$
K^b(\SBim_{n})[x_{n+1}] \xtofrom[\iota_n^*]{\iota_{n*}} D^b(\coh(\FH^\dg_{n} \times \CC) )
$$
obtained from the inductive hypothesis and tensoring with the extra variable $x_{n+1}$. We define the composed functors:
$$
\iota_* \ : \ K^b(\SBim_{n+1}) \xtofrom[I]{\Tr}  \SBim_{n}[x_{n+1}] \xtofrom[{\iota_n}^*]{\iota_{n*}} D^b(\coh(\FH^\dg_{n} \times \CC)) \ : \ \iota^*
$$
According to Proposition \ref{prop:dg scheme}, we have $\FH^\dg_{n+1} = \PP \CE^{\vee}_n$, where $\CE_n$ is the complex on $\FH^\dg_n \times \CC$ from \eqref{eqn:complex}. Relation \eqref{eqn:two step} states that this complex has projective dimension 1, and we can therefore apply Proposition \ref{prop:lift}. To do so, we must exhibit an invertible object $F \in \SBim_{n+1}$ and a morphism:
$$
\iota^* \CE_n \stackrel{\beta}\longrightarrow F
$$
in $K^b(\SBim_n)$. We will choose $F = L_{n+1}$ and take the morphism $\beta$ to be the adjoint of \eqref{eqn:conjmain}:
$$
\CE_n = \iota_*(L_{n+1}) = \iota_{n*} \left( \Tr(L_{n+1}) \right)
$$
The full statement of \eqref{eqn:conjmain} allows one to prove that $S^k(\CE_n) = \iota_*(L_{n+1}^k)$, which establishes the fact that $\SBim_{n+1}$ is birational over $\FH_{n+1}$ by \eqref{eqn:properproj}. To complete the proof of Conjecture \ref{conj:1} one needs to also check that \eqref{eqn:koszul} holds, which is part (c) of Conjecture \ref{conj:main}.

\subsection{Invertible objects in monoidal categories}
\label{sub:invertible}

We summarize several important properties of invertible objects \eqref{eqn:def invertible} in arbitrary monoidal categories. The proofs are straightforward, and left as exercises to the interested reader.

\begin{proposition}
\label{prop:inv}
For any invertible object $F\in \cat$ and two arbitrary objects $C,C'\in \cat$, there exist canonical isomorphisms:
$$
\Hom_{\cat}(F \otimes C, F \otimes C') \cong \Hom_{\cat}(C,C') \cong \Hom_{\cat}(C \otimes F, C' \otimes F)
$$
\end{proposition}

\begin{corollary}
\label{cor:inv 1}

Tensoring with an invertible object and with its inverse yield biadjoint functors, that is, we have canonical isomorphisms:
$$
\Hom_{\cat}(C, F \otimes C') \cong \Hom_{\cat}(F^{-1} \otimes C,C') \qquad \Hom_{\cat}(C, C' \otimes F) \cong \Hom_{\cat}(C \otimes F^{-1},C')
$$
\end{corollary}

\begin{corollary}
\label{cor:inv 2}

For any invertible $F\in \cat$ and any object $C\in \cat$, we have:
$$
\Hom_{\cat}(\1, F\otimes C) \cong \Hom_{\cat}(1, C \otimes F)
$$
\end{corollary}

\section{Example: the case of two strands}
\label{sec:n=2}

\subsection{The geometry of $\FH_2$}
\label{sub:fh 2}

In this section, we will always write $\FH_2 = \FH_2(\CC)$. In this section we construct explicitly the functors $\iota^*$ and $\iota_*$ between the category of sheaves on $ \FH_2 $ and the category of Soergel bimodules $\SBim_2$. We have the matrix presentation:
$$
\FH_2 = \frac {\left\{ X = \left( \begin{array}{cc}
x_1 & 0 \\
z & x_2 \end{array} \right), Y = \left( \begin{array}{cc}
0 & 0 \\
w & 0 \end{array} \right),[X,Y]=0, v = \left( \begin{array}{cc}
1 \\
0 \end{array} \right) \text{ cyclic} \right\} }{ \text{conjugation by } g = \left( \begin{array}{cc}
1 & 0 \\
0 & c \end{array} \right)}
$$
Note that in the presentation above, we fixed the vector $v$ (and this fixes the first column of the conjugating matrix) to eliminate some coordinates. Unwinding the above gives us:
\begin{equation}
\label{eqn:matrix fh 2}
\FH_2 = \frac {\{(x_1,x_2,z,w), \ (x_1-x_2)w = 0, \ z,w \text{ not both zero} \}}{(x_1,x_2,z,w) \sim (x_1,x_2,cz,cw)} = \proj(A)
\end{equation}
where $x_1,x_2,z,w$ have degrees $0,0,1,1$ in the graded algebra:
\begin{equation}
\label{eqn:functions on fh 2}
A = \frac{\CC[x_1,x_2,z,w]}{(x_1-x_2)w} 
\end{equation}
Recall the complex \eqref{eqn:complex}:
\begin{equation}
\label{eqn:complex e1}
\CE_1 = \left[qt \CO \xrightarrow{(0,x_1-x_2,0)} q \CO \oplus t \CO \oplus \CO \xrightarrow{(x_1-x_2,0,1)^T} \CO \right]
\end{equation}
on $\FH_1(\CC) \times \CC = \CC^2$, from which it is clear the the leftmost map is injective and the rightmost map is surjective on all fibers. Therefore, we have $H^0(\CE_1) \cong \CE_1$ and hence:
$$
\FH_2 \cong \FH_2^{\dg}
$$ 
Moreover, letting $z$ and $w$ be coordinates on the first two summands of the middle space of \eqref{eqn:complex e1}, we observe that $H^0(\CE_1) = (\CO  z \oplus \CO w)/(x_1-x_2)w$, which matches the algebra \eqref{eqn:functions on fh 2}. The irreducible components of the flag Hilbert scheme are:
\begin{equation}
\label{eqn:irreducible components n=2}
\FH_2 = Z_1 \cup Z_2
\end{equation}
where:
\begin{equation}
\label{eqn:z1 n=2}
Z_1 = \overline{\{x_1\neq x_2\}} = \{b=0\} = \CC^2 \text{ with coordinates } (x_1,x_2) = \proj(A/wA)
\end{equation}
\begin{equation}
\label{eqn:z2 n=2}
Z_2 = \{x_1=x_2\} = \CC \times \PP^1 \text{ with coordinates } (x,[z:w]) = \proj(A/(x_1-x_2)A)
\end{equation}
The intersection of these two irreducible components is:
$$
Z_1 \cap Z_2 = \CC \times [1:0] = \CC \times \{ I_{(2)} \}
$$
while the other torus fixed point $I_{(1,1)}$ satisfies:
$$
Z_1 \not \ni I_{(1,1)} \in Z_2, \ I_{(1,1)} = (0,[0:1])
$$

\begin{figure}[ht!]
\begin{tikzpicture}
\draw (2,2) circle (1);
\draw (0,1)--(4,1);
\draw (2,1) node {$\bullet$};
\draw (2,3) node {$\bullet$};
\draw (2,0.5) node {$I_{(2)}$};
\draw (2,3.5) node {$I_{(1,1)}$};
\draw (4,0.7) node {$Z_1$};
\draw (0.7,2) node {$Z_2$};
\end{tikzpicture}
\caption{Flag Hilbert scheme of two points}
\label{fig:fh2}
\end{figure}
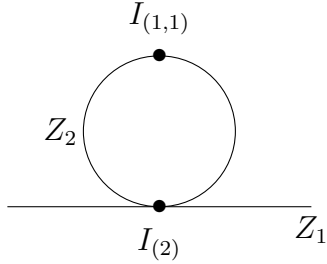

\subsection{Cohomology of sheaves on $\FH_2$}
\label{sub:coh 2}

On the projectivization \eqref{eqn:matrix fh 2}, the line bundles of importance for us are $\CL_1 \cong \CO$ and $\CL_2 \cong \CO(1)$, where the latter denotes the Serre twisting sheaf. Note that:
\begin{equation}
\label{eqn:taut 2}
\CT_2 \cong \CO \oplus \CO(1)
\end{equation}
We will now compute the cohomology groups of certain line bundles on $\FH_2$. To simplify our computations by removing a factor of $\CC$, we will work with the reduced version all the schemes and dg schemes in question (see Subsection \ref{sub:reduced}). Specifically, this means:
\begin{equation}
\label{eqn:bar a}
\bFH_2 = \proj(\bA) \qquad \text{where } \bA = \frac{\CC[x,z,w]}{xw}  
\end{equation}
where we set $x_1+x_2=0$ and $x_1-x_2 = x$. The irreducible components of this variety are:
$$
\bZ_1 = \CC \qquad \text{and} \qquad \bZ_2 = \PP^1 = \FH_2(\point)
$$
Note that $\bCT_2 = \CO(1)$. The following cohomology computations are well-known:
$$
H^i(\bZ_1,\CO(k)) = \frac {q^{k}}{1-q} \cdot \delta_{i,0}
$$
because $\bZ_1 = \CC$, while:
$$
H^i(\bZ_2,\CO(k))=\begin{cases}
t^k+\ldots+q^k\ & \text{if}\ i=0 \text{ and } k\ge 0\\
(qt)^{-1}(t^{k+2}+\ldots+q^{k+2})\ & \text{if}\ i=1 \text{ and } k\le -2 \\
0 & \text{otherwise}
\end{cases}
$$
because $\bZ_2 = \PP^1$ with equivariant weights $q$ and $t$. Consider the short exact sequence:
$$
\xymatrix{
0 \ar[r] & q\CO_{\bZ_1} \ar[r]^-{x} & \CO_{\bFH_2} \ar[r] & \CO_{\bZ_2} \ar[r] & 0}
$$
which is induced by \eqref{eqn:irreducible components n=2}. Because the cohomology of sheaves on $\bZ_1$ is concentrated in degree 0, we have the following equality of $(q,t)$--equivariant vector spaces:
$$
H^i(\bFH_2,\CO(k)) = q H^i(\bZ_1,\CO(k)) + H^i(\bZ_2,\CO(k)) =
$$
\begin{equation}
\label{eqn:coh n=2}
= \begin{cases}
t^k+\ldots+q^k+\frac{q^{k+1}}{1-q} & \text{if}\ i=0 \text{ and } k\ge 0\\
\frac{q^{k+1}}{1-q} & \text{if}\ i=0 \text{ and } k<0\\
(qt)^{-1}(t^{k+2}+\ldots+q^{k+2}) & \text{if}\ i=1 \text{ and } k\le -2 \\
0 & \text{otherwise}
\end{cases}
\end{equation}
The analogous equalities for the non-reduced version $\FH_2$ are obtained by dividing the right hand sides of \eqref{eqn:coh n=2} by $1-q$.

\subsection{Soergel bimodules for $n=2$}
\label{sub:soergel n=2}

The category of Soergel bimodules is generated by two objects: $R = \CC[x_1,x_2]$ and $B = R \otimes_{R^{(12)}} R$. With our grading conventions, we have: 
\begin{equation}
\label{eqn:square}
B^2 = B \otimes_R B \cong q^{\frac 12}B\oplus q^{- \frac 12}B
\end{equation}
In the reduced category, we can set $x_1+x_2=0$ and $x_1-x_2 = x$, and write $\Rred=\CC[x]$ and:
$$
\Bred=\Rred\otimes_{\Rred^s}\Rred=\CC[x]\otimes_{\CC[x^2]}\CC[x].
$$
This object also satisfies property \eqref{eqn:square}, and moreover:
$$
\Hom(\1,\Bred)\simeq \Ext^1(\1,\Bred)=\Rred
$$
are rank 1 modules over $\Rred$. In terms of grading, note that $\Ext^1$ differs from $\Hom$ by a shift by the equivariant weight $a^{-1}q^{-1}$, which is an incarnation of the wedge product in \eqref{eqn:iota wedge}. Thus:
\begin{equation}
\label{eqn:basic}
\RHom^{\bullet}(\1,\Bred)=\wedge^{\bullet} \left( \frac {\xi}{qa} \right)\otimes \Rred %\qquad \Longrightarrow \qquad \RHom^{\bullet}(\1,B)=\wedge^{\bullet}(q\xi_0, q^2\xi)\otimes R.
\end{equation}
for a formal variable $\xi$. The object in the Soergel category which corresponds to a single positive crossing $\sigma$ is the following complex:
$$
\sigma = \left[\Bred \xrightarrow{1\otimes 1 \mapsto 1} \frac {s\Rred}{q^{\frac 12}} \right]  %\sigma = \left[q^{\frac 12} \Bred \xrightarrow{1\otimes 1 \mapsto 1} t^{\frac 12} \Rred \right] 
$$
The powers of $s$ mark homological degree, and so they are always consecutive integers in a complex. We mainly use them to pinpoint the 0--th term of a complex, and to compare with formulas from geometry. Similarly, the object in the Soergel category which corresponds to a single negative crossing $\sigma^{-1}$ is:
$$
\sigma^{-1} = \left[\frac {q^{\frac 12}\Rred}s \xrightarrow{1 \mapsto x \otimes 1 + 1 \otimes x} \Bred \right]  % \sigma^{-1} = \left[t^{-\frac 12} \Rred \xrightarrow{1 \mapsto x \otimes 1 + 1 \otimes x} q^{-\frac 12} \Bred \right] 
$$
Let us write $\bFT = \bFT_2$ for the image of the full twist in the reduced Soergel category, and note that $\bFT = \sigma^2$. Therefore, formula \eqref{eqn:square} allows us to write:
$$
\bFT = \left[q^{\frac 12} \Bred \to \frac {s \Bred}{q^{\frac 12}} \to \frac {s^2 \Rred}q \right]  % \bFT = \left[q^{\frac 12} \Bred \to q^{\frac 12} t^{\frac 12} \Bred \to t \Rred \right]  
$$
Recall that the connection between the parameters $s$ and $t$ is given by $s^2 = qt$. The two eigenmaps that span the space $\Hom(\1, \bFT)$ are described in the following diagram:
\begin{equation}
\label{eqn:eigenmaps2}
\includegraphics{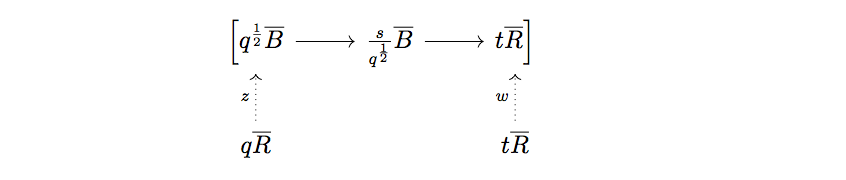}
\end{equation}
where $z = (1\mapsto x \otimes 1 + 1 \otimes x)$ and $w = \text{Id}$. As we will see in more examples in the next Subsection, it is no coincidence that the only maps from $\Rred$ into non-negative powers of the full twist have integer $q,t$--weights: this is called the ``parity miracle" by \cite{EH}.

\begin{proposition}
\label{prop:koszul=0 n=2}
We have the following relation in the category $\SBim_2$:
$$
0 \stackrel{\ehe}\cong \left[ ... \xrightarrow{\alpha_1} q \bFT^{-2} \oplus q \bFT^{-2} \xrightarrow{\alpha_2}  \bFT^{-2} \oplus q \bFT^{-1} \xrightarrow{\alpha_1} \bFT^{-1} \oplus \bFT^{-1} \xrightarrow{(z,w)} \1 \right]
$$
where the maps $\alpha_1$ and $\alpha_2$ are given by:
\begin{equation}
\label{eqn:alphas}
\alpha_1 = \left( \begin{array}{cc}
w & 0 \\
-z & x \end{array} \right) \qquad \text{and} \qquad  \alpha_2 = \left( \begin{array}{cc}
x & 0 \\
z & w \end{array} \right) 
\end{equation}

\end{proposition}

\begin{proof}
Remark that this complex is filtered by complexes:
\begin{multline}
\label{eqn:small koszul}
\left[\bFT^{-k-2}  \xrightarrow{(w,-z)} \bFT^{-k-1} \oplus \bFT^{-k-1} \xrightarrow{(z,w)} \bFT^{-k} \right] = \\
= \bFT^{-k-2}\otimes \Cone\left[\1\xrightarrow{w}\bFT\right]\otimes \Cone\left[\1\xrightarrow{z}\bFT\right],
\end{multline}
so it is sufficient to prove that $\Cone(z)\otimes \Cone(w)\simeq \Cone(w)\otimes \Cone(z) \simeq 0$ (indeed, this would imply that the complexes in the left hand side of \eqref{eqn:small koszul} are contractible). Since:
$$
\Cone\left[\1\xrightarrow{w}\bFT\right]=[\Bred \xrightarrow{} \Bred],
$$
and:
$$
\Bred \otimes \Cone(z)\simeq \Bred \otimes [\Rred \to \Bred \to \Bred \to \Rred]\simeq [\Bred \to \Bred \oplus \Bred \to \Bred \oplus \Bred \to \Bred]\simeq 0,
$$
we conclude that $\Cone(w) \otimes \Cone(z) \simeq 0$. The case of $\Cone(z) \otimes \Cone(w)$ is analogous. \end{proof}

\subsection{$\proj$ construction}
\label{sub:proj construction}

The purpose of this Subsection is to construct the functors:
\begin{equation}
\label{eqn:functors n=2}
D^b\left(\coh_{\torus} \left( \FH_2\right) \right) \xtofrom[\iota_*]{\iota^*} K^b(\SBim_2)
\end{equation}
and prove Conjecture \ref{conj:1} for $n=2$. To keep our notation simple, we will perform the computation for the reduced versions of the above categories. As was shown in Section \ref{sec:geometry}, in order to construct $\iota_*$ one needs to prove the following isomorphism of graded algebras:
\begin{equation}
\label{eqn:volo}
\bigoplus_{k=0}^\infty \Hom \left(\1,\bFT^{k}\right) \cong \bigoplus_{k=0}^\infty \Hom \left( \bFH_2,\CO(k) \right)  
\end{equation}
To compute the left hand side of \eqref{eqn:volo}, recall from \cite{Kh} that we have the following identity in $\SBim_2$ for all $k>0$:
$$
\bFT^{k}\simeq \left[\underbrace{q^{k-\frac 12}\Bred \rightarrow q^{k-\frac 32} s \Bred \rightarrow \cdots \rightarrow \frac {s^{2k-3} \Bred}{q^{k-\frac 52}} \rightarrow \frac {s^{2k-2} \Bred}{q^{k-\frac 32}} \rightarrow \frac {s^{2k-1} \Bred}{q^{k-\frac 12}}}_{2k} \rightarrow \frac {s^{2k} \Rred}{q^k} \right]  % \bFT^{k}\simeq \left[\underbrace{q^{k-\frac 12} \Bred \rightarrow q^{k-\frac 12} t^{\frac 12}  \Bred \rightarrow \cdots \rightarrow q^{\frac 32}  t^{k-\frac 32} \Bred \rightarrow q^{\frac 12}  t^{k-1} \Bred \rightarrow q^{\frac 12}  t^{k-\frac 12} \Bred}_{2k} \rightarrow t^k \Rred \right] 
$$
where the maps alternate between $\frac {x \otimes 1 - 1 \otimes x}2$ and $\frac {x \otimes 1 + 1 \otimes x}2$. Since $s=-\sqrt{qt}$, we have:
$$
\Hom(\1,\bFT^{k}) \simeq \left[\underbrace{q^k \Rred \stackrel{0}\rightarrow q^{k-1} t^{\frac 12} \Rred \xrightarrow{x} \cdots \stackrel{0}\rightarrow q  t^{k-\frac 32} \Rred \xrightarrow{x} q t^{k-1} \Rred \stackrel{0}\rightarrow t^{k-\frac 12} \Rred}_{2k} \xrightarrow{x} t^k \Rred \right]   % \Hom(\1,\bFT^{k}) \simeq \left[\underbrace{q^k \Rred \stackrel{0}\rightarrow q^k t^{\frac 12} \Rred \xrightarrow{x} \cdots \stackrel{0}\rightarrow q^{2} t^{k-\frac 32} \Rred \xrightarrow{x} q t^{k-1} \Rred \stackrel{0}\rightarrow q t^{k-\frac 12} \Rred}_{2k} \xrightarrow{x} t^k \Rred \right]  
$$
\begin{equation}
\label{eqn:hom}
\cong z^k \CC[x] \bigoplus_{i=1}^{k}w^{i} z^{k-i} \frac {\CC[x]}x  = \bA^k %\Longrightarrow \Hom(\1,\FT_2^k) =  H^0 \left(\FH_2(\CC), \CO(k) \right)
\end{equation}
One can think of $z$, $w$ as formal variables of degrees $q$, $t$, but they actually correspond to the maps of \eqref{eqn:eigenmaps2} under the required isomorphism \eqref{eqn:volo}. This establishes \eqref{eqn:volo} as an isomorphism of $\CC[x]$--modules. We claim that this isomorphism also preserves the algebra structures, and therefore the functor $\iota_*$ is well-defined. By construction:
$$
\iota_*(\bFT^k) = \CO(k)
$$
for all $k\geq 0$. As for the functor $\iota^*$ of \eqref{eqn:functors n=2}, we require:
$$
\iota^*(\CO(k)) := \bFT^k
$$
and:
$$
\iota^* \left(q\CO \stackrel{z}\longrightarrow \CO(1) \right) \ \text{ and } \ \iota^* \left(q\CO \stackrel{w}\longrightarrow \CO(1) \right) \  = \text{ the maps \eqref{eqn:eigenmaps2}}
$$
However, note that this assignment simply defines a functor:
$$
\coh \left( \spec \ \bA / \CC^* \right) \xrightarrow{\iota^*} \SBim_2
$$
since $\bA$ is the homogeneous coordinate ring of $\bFH_2$. We wish to show that this functor factors through $D^b(\coh (\proj \ \bA )) $. To do so, we must prove that the object: 
\begin{equation}
\label{eqn:quiero}
0 \stackrel{\qis}\cong \bA_0 = \frac {\bA}{(z,w)} \text{ on } \bFH_2 \qquad \xrightarrow{\text{goes to}} \qquad \iota^*(\bA_0) \stackrel{\he}\cong 0 \text{ in } \SBim_2
\end{equation}
To compute the image of $\bA_0$ under $\iota^*$, we need to resolve this object in terms of free $\bA$ modules. The standard choice is the Koszul resolution, which is infinite because $\bFH_2$ is singular:
$$
0 \stackrel{\qis}\cong \left[ ... \xrightarrow{\alpha_1} q\bA(-2) \oplus q\bA(-2) \xrightarrow{\alpha_2}  \bA(-2) \oplus q\bA(-1) \xrightarrow{\alpha_1}\bA(-1) \oplus \bA(-1) \xrightarrow{(z,w)} \bA \right]
$$
where the maps alternate between those of \eqref{eqn:alphas}. Then \eqref{eqn:quiero} follows from Proposition \ref{prop:koszul=0 n=2}. 

\begin{remark}
\label{rem:a grading n=2}

By analogy with \eqref{eqn:hom}, we have:
$$
\Ext^1(\1,\bFT^{k}) \cong \left[\underbrace{q^{k-1} \Rred \stackrel{0}\rightarrow q^{k-1} t^{\frac 12} \Rred \xrightarrow{x} \cdots \stackrel{0}\rightarrow q t^{k-\frac 32} \Rred \xrightarrow{x}  t^{k-1} \Rred \stackrel{0}\rightarrow t^{k-\frac 12} \Rred}_{2k} \stackrel{1}\rightarrow t^k \Rred \right] \cong % \cong H^0(\bFH_2(\CC),\CO(k-1)) %\Ext^\bullet\left(\1,\overline{\FT}_2^{k}\right) \cong H^0 \left(\bFH_2(\CC), \wedge^\bullet \CO(k-1) \right)
$$
\begin{equation}
\cong \ z^{k-1} \CC[x] \bigoplus_{i=1}^{k-1} w^{i} z^{k-1-i} \frac {\CC[x]}x = \bA^{k-1}
\end{equation}
and therefore:
$$
\RHom \left(\1,\bFT^k \right) \cong \Hom\left(\1,\bFT^{k-1} \right)
$$
This is precisely \eqref{eqn:rhom vs hom} for $M = \bFT^k$ and $\overline{T}_2 = \iota^*(\bCT_2) = \iota^*(\CO(1)) = \bFT$.

\end{remark}

%Note that for $k\ge 0$ the sheaves in the right hand side of \eqref{eqn:ft2 all a} have no higher cohomology, which corresponds to the fact that on the Soergel side, all the cohomology is concentrated in even homological degree (i.e. integer $t$-degree). 

\subsection{Sheaves for two-strand braids}
\label{sub:sheaves n=2}

To construct the sheaf $\iota_*(M)$ for any object $M \in \SBim_2$, one needs to consider the module $\Hom(\1, M \otimes \bFT^{\bullet})$ over the graded algebra $A = \Hom(\1,\bFT^{\bullet})$. In the previous subsection, we have studied the case $M = \bFT^k$ for positive integers $k$, and we found that $\iota_*(M) = \CO(k)$. The computation for negative $k$ is more interesting:
$$
\bFT^{-k} \cong  \left[t^{-k} \Rred  \rightarrow \underbrace{q^{-\frac 12} t^{\frac 12-k}  \Bred \rightarrow q^{-\frac 12} t^{1-k} \Bred \rightarrow q^{-1} t^{\frac 32-k} \Bred \rightarrow \dots \rightarrow q^{\frac 12-k} t^{-\frac 12} \Bred \rightarrow q^{\frac 12-k} \Bred}_{2k}  \right]  %\left[ t^{-1}\Rred \longrightarrow t^{-\frac 12} q^{-\frac 12} \Bred \longrightarrow q^{-\frac 12} \Bred \right]
$$
for any $k\geq 0$, where the maps alternate between $\frac {x \otimes 1 + 1 \otimes x}2$ and $\frac {x \otimes 1 - 1 \otimes x}2$. Therefore, we have:
$$
\Hom(\1,\bFT^{-k}) \cong \left[t^{-k} \Rred  \xrightarrow{1} \underbrace{ t^{\frac 12-k} \Rred \xrightarrow{0} t^{1-k} \Rred \xrightarrow{x}  q^{-1} t^{\frac 32-k} \Rred  \xrightarrow{0} \dots \xrightarrow{x} q^{1-k} t^{-\frac 12} \Rred  \xrightarrow{0}  q^{1-k} \Rred}_{2k}  \right] 
$$
\begin{equation}
\label{eqn:negative hom n=2}
\Longrightarrow \quad \Hom(\1,\bFT^{-k}) \cong t^{\frac 12} H^1\left(\bFH_2,\CO(-k)\right) 
\end{equation}
according to \eqref{eqn:coh n=2}. The case of general $a$ follows by analogy with the previous subsection, so we conclude the following formula that extends \eqref{eqn:volo} to negative integers:
\begin{equation}
\label{eqn:formula}
\RHom^\bullet_{\SBim_2}(\1,\bFT^{-k}) \cong R\Gamma \left( \bFH_2,\CO(-k)\otimes \wedge^{\bullet} \CO(-1) \right)
\end{equation}

\begin{remark}
\label{rem:different gradings}

Let us observe the fact that the derived functors in the two sides of the above equation are very different. In the left hand side, we have the derived Hochschild homology functor, whose degree is measured by $a$. In the right hand side, we have derived direct image of sheaves, whose degree is measured by $t^{\frac 12}$, and the $a$ grading comes from $\wedge^\bullet \CO(-1)$.

\end{remark}

To complete the discussion for $n=2$, let us compute $\CB(\sigma) := \iota_*(\sigma)$ where $\sigma$ denotes a single positive crossing. Together with the projection formula \eqref{eqn:projection}, this implies that:
$$
\CB(\sigma^{2k+1}) := \iota_*(\sigma^{2k+1}) = \iota_*(\sigma\otimes \bFT^k) = \iota_*(\sigma) \otimes \CO(k) = \CB(\sigma)\otimes \CO(k)
$$
for all integers $k$. In fact, we have:
$$
\Hom(\1,\sigma^{2k+1}) \cong \left[\underbrace{q^{k+\frac 12} \Rred \xrightarrow{x} \cdots \xrightarrow{0} q^{2} t^{k-1} \Rred \xrightarrow{x}  q t^{k-\frac 12} \Rred \xrightarrow{0} q t^{k} \Rred}_{2k+1} \xrightarrow{x}  t^{k+\frac 12} \Rred \right]  
$$
and therefore:
$$
\bigoplus_{k\ge 0}\Hom(\1, \sigma \otimes \FT^{k})=\bigoplus_{k\ge 0}\Hom(\1, \sigma^{2k+1})=t^{1/2}\frac{\CC[x,z,w]}{x}
$$
where recall that $z$ and $w$ are the maps of \eqref{eqn:eigenmaps2}. We conclude that $\CB(\sigma)$ is the structure sheaf of the subscheme $\{x=0\} \subset \bFH_2$, which is nothing but the connected component $\overline{Z_2} = \FH_2(\point) \cong \PP^1$ of \eqref{eqn:z2 n=2}. The periodic resolution \eqref{eq:mf} takes the form:
$$
\CB(\sigma) \cong \CO_{\PP_1} \stackrel{\qis} \cong \left[ ...  \stackrel{w}\longrightarrow q^2t \CO(-1) \stackrel{x}\rightarrow qt \CO(-1) \stackrel{w}\longrightarrow q \CO(-1) \stackrel{x}\rightarrow \CO\right]
$$
where $\CO$ denotes the structure sheaf of $\bFH_2$. In the non-reduced category, one needs to replace $x$ by $x_1-x_2$ everywhere. Finally, let us compute $\iota_*(B)$, where recall that $B = R \otimes_{R^s} R$. Since $z\otimes \Id_B$ is an isomorphism between $B$ and $\FT_2\otimes B$, we have:
$$
\bigoplus_{k\ge 0}\Hom(\1,B\cdot\FT^{k})=\bigoplus_{k\ge 0}z^k\Hom(\1,B)=\CC[x_1,x_2,z].
$$
Therefore $\iota_*(B)$ is the structure sheaf of the irreducible component $Z_1 \subset \FH_2$ cut out by the equation $w=0$ (see \eqref{eqn:z1 n=2}), which is isomorphic to $\CC^2$ with coordinates $x_1$ and $x_2$. 

\section{Example: the case of three strands}
\label{sec:n=3}

\subsection{The geometry of $\FH_3$}
\label{sub:fh 3}

We will now study the variety $\FH_3 = \FH_3(\CC)$ and formulate a precise conjecture about the sheaf $\iota_*(\text{figure eight knot})$. Recall the matrix presentation:
$$
\FH_3  = \left\{ X = \left( \begin{array}{ccc}
x_1 & 0 & 0 \\
a & x_2 & 0 \\
\alpha_1 & \alpha_2 & x_{3} \end{array} \right), \ Y = \left( \begin{array}{ccc}
0 & 0 & 0 \\
b & 0 & 0 \\
\beta_1 & \beta_2 & 0 \end{array} \right), \ [X,Y]=0, \right.
$$
\begin{equation}
\label{eqn:matrix fh 3}
\left. v = \left( \begin{array}{ccc}
1 \\
0 \\
0 \end{array} \right) \text{ cyclic} \right\} \Big/ \text{conjugation by } g = \left( \begin{array}{ccc}
1 & 0 & 0 \\
0 & c & 0 \\
0 & t & d \end{array} \right)
\end{equation}
Note that in the presentation above, we fixed the vector $v$ (and this fixes the first column of the conjugating matrix) to eliminate certain coordinates. Note that the map $\FH_3 \rightarrow \FH_2$ is given by only retaining the top $2 \times 2$ corners of the matrices in question. If one is given the eigenvalues $x_1,x_2,x_3$ and the point $[a:b] \in \PP^1$, then the datum one needs to construct a point in $\FH_3$ is the vector:
\begin{equation}
\label{eqn:dat}
(\alpha_1, \alpha_2, \beta_1, \beta_2) \in \CT_2^\vee \oplus \CT_2^\vee
\end{equation}
To ensure that the equation $[X,Y]=0$ is satisfied, we need to ensure that: 
$$
(x_1-x_3)\beta_1 = \alpha_2b - \beta_2 a \quad \text{and} \quad (x_2-x_3)\beta_2 = 0
$$
Moreover, the fact that we quotient out by conjugation matrices implies that we must identify:
$$
(\alpha_1, \alpha_2, \beta_1, \beta_2) \sim (\alpha_1 + ta, \alpha_2+t(x_2-x_3), \beta_1 + tb, \beta_2) \quad \text{and} \quad (\alpha_1,\alpha_2,\beta_1,\beta_2) \sim d(\alpha_1,\alpha_2,\beta_1,\beta_2)
$$
for $t \in \CC$ and $d\in \CC^*$. Unwinding these facts, one sees that the datum \eqref{eqn:dat} corresponds to a vector in $H^0(\CE^\vee_2)$, where $\CE_2$ is the complex in \eqref{eqn:complex} when $* = \CC$. It is elementary to prove that $\CE_2$ and $\CE_2^\vee$ are quasi-isomorphic to their zero-th cohomology, so we conclude that $\FH_3 = \FH_3^\dg$. The irreducible components of the flag Hilbert scheme $\FH_3$ are:
$$
\FH_3(\CC) = Z_1 \cup Z_2 \cup Z_3 \cup Z_4 \cup Z_5
$$
where $Z_1,\ldots,Z_5$ are determined by which eigenvalues $x_1,x_2,x_3$ are equal to each other:
$$
Z_1 = \overline{\{x_1 \neq x_2 \neq x_3 \neq x_2\}}, \qquad Z_2 = \{x_1=x_2=x_3\}
$$
$$
Z_3 = \overline{\{x_1 = x_2 \neq x_3\}}, \ Z_4 = \overline{\{x_3 = x_1 \neq x_2\}}, \ Z_5 = \overline{\{x_2 = x_3 \neq x_1\}}
$$
On $Z_1$, because the eigenvalues are generically distinct, the commutation relation $[X,Y]=0$ forces $Y=0$. Then the cyclicity of the vector $v$ implies $a,\alpha \neq 0$, and so conjugation by $g$ allows one to set $a = \alpha = 1$ and $e=0$. We conclude that:
\begin{equation}
\label{eqn:z1}
Z_1 = \CC^3
\end{equation}
As for $Z_2$, note that one can always subtract a constant matrix from $X$ without changing any of the other properties of \eqref{eqn:matrix fh 3}. By \eqref{eqn:pt 3}, we see that:
\begin{equation}
\label{eqn:z2}
Z_2 = \FH_2(\point) \times \CC = \PP_{\PP^1} \left( \frac {\CO(1)}{qt} \oplus \CO(-2) \right) \times \CC
\end{equation}

%We conclude that the variety $\FH_3$ can be described as the projectivization, over $\FH_2$, of the middle cohomology of the following complex:
%\begin{equation}
%\label{eqn:ldr}
%\CO(-1) \xrightarrow[\left( \begin{array}{ccc}
%a & x_2-x_3 \end{array} \right) \oplus \left( \begin{array}{cc}
%b & 0 \end{array} \right)]{\psi^\vee} \CT_2^\vee \oplus \CT_2^\vee \xrightarrow[\left( \begin{array}{cc}
%0 & 0 \\
%-b & 0  \end{array} \right) \oplus \left( \begin{array}{cc}
%x_1-x_3 & 0 \\
%a & x_2-x_3  \end{array} \right)]{\phi^\vee} \CT_2^\vee
%\end{equation}
%Note that the above complex is simply the dual of $\CE_2$ (up to quasi-isomorphism, because we have chipped off a direct summand of $\CO$ from the middle and rightmost space of \eqref{eqn:complex}). Since $\CT_2 = \CO \oplus \CO(1)$, we dualize \eqref{eqn:ldr} and consider the map of complexes:
%$$
%\begin{tikzcd}
%\CO \oplus \CO(1) \arrow{r}{\phi} & \CO \oplus \CO \oplus \CO(1) \oplus \CO(1) \arrow{r}{\psi} & \CO(1) \\
%? \arrow{r}{\nu} \arrow{u}{i} & ? \arrow{u}{j} & 
%\end{tikzcd}
%$$
%where:
%$$
%\phi = \left( \begin{array}{cc}
%0 & 0 \\
%x_1-x_3 & 0 \\
%-b & 0 \\
%a & x_2-x_3 \end{array} \right) \qquad \psi = \left( \begin{array}{cccc}
%a & b & x_2-x_3 & 0 \end{array} \right)
%$$
%The above computation is therefore an example of Theorem \ref{thm:complex}. 

\subsection{Torus braids}
\label{sec:n=3 torus}

In this section, we compare our conjectures to the ones of \cite{GORS,ORS} for three-strand torus braids. The remainder of this Section provides explicit computations that follow from Conjectures \ref{conj:1} and \ref{conj:knots}.

\begin{proposition}
\label{prop:3 strand}

The sheaves on $\FH_3$ associated to torus braids on 3 strands are:
\begin{equation}
\label{eqn:3 strand}
\iota_{*}(\sigma_1\sigma_2)^{k} = \iota_{*}(\sigma_2\sigma_1)^{k} = 
\begin{cases}
\CL_2^{m}\CL_3^{m} & k=3m, \\
\CL_2^{m}\CL_3^{m}\otimes \CO_{Z_2}& k=3m+1,\\
\CL_2^{m+1}\CL_3^{m}\otimes \CO_{Z_2}& k=3m+2.\\
\end{cases}
\end{equation}
Here $m$ (and hence $k$) is allowed to be either positive or negative.
\end{proposition}

\begin{proof}
Clearly, $(\sigma_1\sigma_2)^3=(\sigma_2\sigma_1)^3=\FT_3= \iota^*(\CL_2\CL_3)$, so in virtue of the projection formula \eqref{eqn:sk} it is sufficient to consider the cases $k=0,1,2$. For $k=0$, Conjecture \ref{conj:1} states that $\iota_*(\1)=\CO_{\FH_3}$, which is precisely the content of \eqref{eqn:3 strand}. For $k=1$, Conjecture \ref{conj:knots} implies $\iota_*(\sigma_1\sigma_2)=\CO_{Z_2}$. Furthermore, for all $a, b \in \NN$ one has:
\begin{equation}
\label{eqn:equal traces}
\Tr(\sigma_2\sigma_1\CL_2^{a}\CL_3^{b})=\Tr(\sigma_1\sigma_2\CL_2^{a}\CL_3^{b}),
\end{equation}
since $\sigma_1$ commutes with both $\CL_2$ and $\CL_3$, and the trace map enjows the property $\Tr(\sigma\sigma') = \Tr(\sigma'\sigma)$. By virtue of the definition \eqref{eqn:general 1} of the sheaves associated to the braids $\sigma_1\sigma_2$ and $\sigma_2\sigma_1$, formula \eqref{eqn:equal traces} implies that $\iota_*(\sigma_2\sigma_1) = \iota_*(\sigma_1\sigma_2)$. The case $k=2$ of \eqref{eqn:3 strand} follows analogously, because:
$$
(\sigma_1\sigma_2)^2=\CL_2\sigma_2\sigma_1,\ (\sigma_2\sigma_1)^2=\sigma_1\sigma_2\CL_2.
$$
\end{proof}

To compute the Khovanov-Rozansky homology of torus braids, one needs to compute the homology of the resulting line bundles either on $\FH_3$, or on $Z_2=\FH_3(\point)\times \CC$. For simplicity, we will consider only the latter case, which corresponds to knots:

\begin{proposition}
\label{prop:3 strand torus}

The following equations hold: $\qquad H^{i}(\FH_3(\point),\CL_2^{a}\CL_3^{b}) =$
\begin{equation}
\label{eqn:3 strand torus}
= \begin{cases}
H^i(\PP^1,\CO(a)\otimes S^b(\CO(2)\oplus qt\CO(-1))& \text{if}\ b\ge 0,\\
0& \text{if}\ b=-1, \\
H^{i+1} \left(\PP^1,\frac {\CO(a-1)}{qt} \otimes S^{-b-2}\left(\CO(-2) \oplus \frac {\CO(1)}{qt} \right) \right)& \text{if}\ b\le -2.
\end{cases}
\end{equation}
\end{proposition}

\begin{proof} Let $\pi:\FH_3(\point)\to \FH_2(\point)=\PP^1$ be the natural projection.
By \eqref{eqn:z2} we have $\FH_3(\point) = \proj \left( S^*_{\PP^1} (\CO(2)\oplus qt\CO(-1)) \right)$. The following properties hold:
$$
R^{i}\pi_* \left( \CL_3^{b} \right)=\begin{cases}
S^b(\CO(2)\oplus qt\CO(-1)) & \text{if} \ i=0 \text{ and } b\ge 0, \\
\frac {\CO(-1)}{qt} \otimes S^{-b-2}\left(\CO(-2) \oplus \frac {\CO(1)}{qt} \right) &\text{if} \ i=1 \text{ and } b\le -2, \\
0 & \text{otherwise}
\end{cases}
$$
Indeed, the second formula follows from the first and Serre duality. This completes the proof.
\end{proof}

\begin{corollary}
\label{cor:3 strand torus}
Putting together \eqref{eqn:3 strand}, \eqref{eqn:3 strand torus} and the well-known formula for the cohomology of line bundles on $\PP^1$, we have the following formulas for all $m \geq 0$. 
\begin{equation}
\label{eqn:hhh 3 strand 1}
\HHH \left( (\sigma_1\sigma_2)^{3m+1} \right) = H^*(\FH_3(\point),\CL_2^{m}\CL_3^{m}) = 
\end{equation}
$$
= H^0 \left(\PP^1,  \bigoplus_{i=0}^m (qt)^i \CO(3m-3i)\right) = \sum_{i=0}^m \sum_{j=0}^{3m-3i} q^{i+j} t^{3m-2i-j}
$$
\begin{equation}
\label{eqn:hhh 3 strand 2}
\HHH \left( ( \sigma_1\sigma_2)^{3m+2} \right) = H^*(\FH_3(\point),\CL_2^{m+1}\CL_3^{m})=
\end{equation}
$$
= H^0 \left(\PP^1,  \bigoplus_{i=0}^m (qt)^i \CO(3m-3i+1)\right) = \sum_{i=0}^m \sum_{j=0}^{3m-3i+1} q^{i+j} t^{3m-2i-j+1}
$$
\end{corollary}

This agrees with the $a=0$ part of the Khovanov-Rozansky homology of $(3,3m+1)$ and of $(3,3m+2)$ torus knots, conjectured in \cite[Section 5.2]{GORS}. To recover the full $a$ dependence, we need to twist the right hand sides of \eqref{eqn:hhh 3 strand 1} and \eqref{eqn:hhh 3 strand 2} by the exterior power:
$$
\wedge^\bullet \CT_3^\vee = \wedge^\bullet (\CL_3 ``\oplus" \CL_2 \oplus \CO)^\vee
$$
where the symbol $``\oplus"$ refers to the fact that $\CT_3$ is a non-trivial extension of $\CL_2 \oplus \CO$ by $\CL_3$. Note that all of our computations can be easily extended to ``twisted torus knots"  in the sense of \cite{CDW}, which are presented by the braids $(\sigma_1\sigma_2)^{k}\otimes \iota^* (\CL_2^{a} )$. We leave the corresponding computation to the interested reader.

\subsection{The longest word}
\label{sub:longest}

Let us describe the sheaf for the positive lift $\sigma_1\sigma_2\sigma_1$ of the longest word in $S_3$.  Remark that the following equation holds for all $a$ and $b$:
\begin{equation}
\label{eqn:conj for 121}
\Tr(\sigma_1\sigma_2\sigma_1\CL_2^{a}\CL_3^{b})=\Tr(\sigma_2\sigma_1\sigma_1\CL_2^{a}\CL_3^{b}),
\end{equation}
since $\sigma_1$ commutes both with $\CL_2$ and $\CL_3$ and the trace satisfies $\Tr(\sigma \sigma') = \Tr(\sigma'\sigma)$.  By Corollary \ref{cor:trace twisted identity}, these traces are isomorphic up to a twist by a permutation $(1\ 2)$. In particular, the left hand side of \eqref{eqn:conj for 121} is supported on $\{x_1=x_3\}$, while the right hand side is supported on $\{x_2=x_3\}$. Furthermore, 
$$
\iota_*(\sigma_2\sigma_1\sigma_1) = L_2\otimes \iota_*(\sigma_2)=L_2\otimes \CO_{\FH(2\sim 3)}.
$$
Note that in the notations of Section \ref{sec:n=3}, $\FH(2\sim 3)=Z_2\cup Z_5$.
There is a natural involution $j_{12}$ on $\FH_3$ which exchanges $x_1$ and $x_2$ in $Z_1$, acts trivially on $Z_2$ and $Z_3$ and permutes $Z_4$ and $Z_5$. We arrive at the following conjecture:

\begin{conjecture}
One has $\iota_*(\sigma_1\sigma_2\sigma_1)=j_{12}^{*}(L_2\otimes \CO_{Z_2\cup Z_5}).$
\end{conjecture}

\subsection{The figure eight knot}
\label{sub:figure eight}

In this section we describe a sheaf for the braid $\beta=\sigma_1\sigma_2^{-1}\sigma_1\sigma_2^{-1}$ representing the figure eight knot.
There is a skein exact sequence relating  $B_{\beta}$ with the following objects in $\SBim_3$:
$$
\sigma_1\sigma_2\sigma_1\sigma_2^{-1}=\sigma_2\sigma_1, \sigma_1\sigma_1\sigma_2^{-1}=\CL_2\sigma_2^{-1}.
$$
More precisely, there is an exact sequence:
\begin{equation}
\label{skein fig 8}
0\xleftarrow{}\sigma_2\sigma_1\xleftarrow{}\Cone\left[\CL_2\sigma_2^{-1}\xleftarrow{x_1-x_2} \CL_2\sigma_2^{-1}\right]\xleftarrow{}B_{\beta}\xleftarrow{} 0.
\end{equation}
\begin{proposition}
\label{prop:cone}
The following identity holds:
$$
\iota_*\Cone\left[\sigma_2^{-1}\xleftarrow{x_1-x_2} \sigma_2^{-1}\right]\simeq \left[\CL_2\CL_3^{-1}\oplus qt\CL_3^{-1}\right]_{Z_2}.
$$
\end{proposition}
\begin{proof}
By \eqref{eq:mf}  one has:
$$
\CO_{\FH(1\sim 2)}\simeq [\CO_{\FH_3(\CC)}\xleftarrow{x_1-x_2}q\CO_{\FH_3(\CC)}\xleftarrow{y_{21}}qt\CL_2^{-1}|_{\FH(1\sim 2)}]
$$
(note that this is also a skein exact sequence for $\sigma_1^{-1},\1, \sigma_1$) and 
$$
\iota_*(\sigma_2^{-1})=\CL_2\CL_3^{-1}\otimes \CO_{\FH(2\sim 3)}.
$$
Since $\CO_{\FH(2\sim 3)}\otimes \CO_{\FH(1\sim 2)}=\CO_{Z_2}$, one has an exact sequence:
$$
0\leftarrow{} \CL_2\CL_3^{-1}|_{Z_2}\xleftarrow{} \iota_*\Cone\left[\sigma_2^{-1}\xleftarrow{x_1-x_2} \sigma_2^{-1}\right] \xleftarrow{}qt\CL_3^{-1}|_{Z_2}
\xleftarrow{} 0.
$$
It remains to notice that 
$$
\Ext_{Z_2}(\CL_2\CL_3^{-1},\CL_3^{-1})=H^{*}(Z_2,L_2^{-1})=H^{*}(\PP^1,\CO(-1))=0.
$$
\end{proof}

\begin{proposition}
Consider the braid $\beta=\sigma_1\sigma_2^{-1}\sigma_1\sigma_2^{-1}$ representing the figure eight knot.
Then, assuming Conjecture \ref{conj:1} and \ref{conj:knots}, one has
$$
\iota_*(\beta)=\CO_{\PP^1}\oplus qt\CL_2\CL_3^{-1}.
$$
\end{proposition}

\begin{proof}
By \eqref{skein fig 8} and Proposition \ref{prop:cone} one has:
$$
0\xleftarrow{}\CO_{Z_2}\xleftarrow{\alpha}\left[\CL_2^2\CL_3^{-1}\oplus qt\CL_2\CL_3^{-1}\right]_{Z_2}\xleftarrow{}qt(\iota_*B_{\beta})\xleftarrow{} 0.
$$
Let us compute the map $\alpha$. Remark that:
$$
\Hom_{Z_2}(\CL_2\CL_3^{-1},\CO)=H^0(Z_2,\CL_2^{-1}\CL_3)=H^0(\PP^1,\CO(1)\oplus qt\CO(-2)),
$$
$$
\Hom_{Z_2}(\CL_2^2\CL_3^{-1},\CO)=H^0(Z_2,\CL_2^{-2}\CL_3)=H^0(\PP^1,\CO\oplus qt\CO(-3)).
$$
Therefore $\alpha$ is the unique degree 1 map $\CL_2^2\CL_3^{-1}\to \CO$ and vanishes on $\CL_2\CL_3^{-1}$, so
$$
\iota_*B_{\beta}\simeq \CL_2\CL_3^{-1}\oplus q^{-1}t^{-1}\Cone[\CO\xleftarrow{\alpha} \CL_2^2\CL_3^{-1}]\simeq \CO_{\PP^1}\oplus \CL_2\CL_3^{-1}.
$$
\end{proof}

Using this result, we can compute the reduced homology of $\beta\cdot \FT_2^{a}\FT_3^{b}$ by computing the homology of each summand individually. Since $\FH_3(\point)$ is a blowup of the punctual Hilbert scheme of 3 points, and $\PP^1$ is the exceptional divisor, the tautological bundle is trivial on $\PP_1$: $\overline{\CT}_3\otimes \CO_{\PP^1}\simeq (q+t)\CO_{\PP^1}$. Similarly, $\FT_3\otimes \CO_{\PP^1}\simeq qt\CO_{\PP^1}$.  We get the following equation:
\begin{equation}
\label{eq: 3 strand square}
\int_{\FH_3(\point)}\CO_{\PP^1}\otimes \FT_2^{a}\FT_3^{b}\otimes \wedge^{\bullet}\overline{\CT}_3^{\vee}=
(1+aq^{-1})(1+at^{-1})(qt)^{b}\int_{\PP^1}\CO(a).
\end{equation}

Equations \eqref{eq: 3 strand square} and \eqref{eqn:3 strand torus} can be used to compute the homology of $\beta\cdot \CL_2^{a}\CL_3^{b}$ for all $a$ and $b$. In particular:
$$
H^*(\FH_3(\point),\CL_2\CL_3^{-1})=H^*(\FH_3(\point),\CL_3^{-1})=0,
$$
$$
H^*(\FH_3(\point),\CL_3^{-2})=H^{*+1}(\PP^1,\CO(-1))=0,
$$
$$
H^*(\FH_3(\point),\CL_2\CL_3^{-2})=H^{*+1}(\PP^1,\CO)=\CC[1],
$$
so 
$$
\int_{\FH_3(\point)}\CL_2\CL_3^{-1}\otimes \wedge^{\bullet}\overline{\CT}_3^{\vee}=a\CC[1],
$$
and
$$
\HHH(\beta)=(1+aq^{-1})(1+at^{-1})+a\sqrt{qt}.
$$
One can compare this with \cite[Table 5.7]{DGR}.

\section{Categorical idempotents and equivariant localization}
\label{sec:equiv}

\subsection{Categories over equivariant schemes}
\label{sub:equivcategories}

We will now enhance the setup of Section \ref{sec:geometry} to schemes endowed with a torus action $T \curvearrowright X$. %Specifically, recall that a map from a monoidal category to a scheme:
%\begin{equation}
%\label{eqn:aaa}
%\cat \stackrel{\iota}\longrightarrow X
%\end{equation}
%is the datum of a pair of adjoint functors:
%\begin{equation}
%\label{eqn:bbb}
%\cat \stackrel[\iota^*]{\iota_*}\rightleftarrows D^b(\coh(X)) 
%\end{equation}

\begin{definition}
\label{def:equivariantcategory} 

A $T$--equivariant category $\cat$ is one which the Hom spaces are representations of $T$. If the category is monoidal, we require the tensor product to preserve the $T$ action. 

\end{definition}

\begin{definition}
\label{def:equivariant} 

Given a $T$--equivariant category $\cat$, we will say that a map $\iota:\cat \rightarrow X$ is $T$--equivariant if the defining functors:
$$
\cat \xtofrom[\iota^*]{\iota_*} \coh_T(X)
$$
preserve the action of $T$ on all $\Hom$ spaces. The derived version is defined analogously. 
\end{definition}

\begin{example}
\label{ex:equivspec} 
Suppose that $X = \spec \ A$ with $A$ being a $T$--graded ring. Recall from Subsection \ref{sub:aff} that realizing $\cat$ as a category over $X$ amounts to giving a ring homomorphism:
$$
A \stackrel{f}\longrightarrow \End_\cat(\1)
$$ 
It is easy to see that $\cat \rightarrow X$ is $T$--equivariant if and only if $f$ is $T$--equivariant.
\end{example}

\begin{example}
\label{ex:equivproj} 

Going one step further, suppose $A$ is a $T$--graded ring. Define:
$$
X = \PP_A^n
$$ 
where the $n+1$ coordinate directions of the projective spaces have $T$--equivariant characters $\lambda_0,...,\lambda_n$. As in Example \ref{ex:projective}, the map $\cat \stackrel{\iota}\longrightarrow X$ is the same datum as a ring homomorphism:
$$
A \stackrel{f}\longrightarrow \End_\cat(\1)
$$
together with an object $F \in K^b(\cat)$ and $n+1$ arrows:
$$
\left[ \lambda_0 \cdot \1 \stackrel{\alpha_0}\longrightarrow F \right], ..., \left[ \lambda_n \cdot \1 \stackrel{\alpha_n}\longrightarrow F \right]
$$
whose tensor product is homotopic to 0. Then $\iota$ is $T$--equivariant if the homomorphism $f$ is $T$--equivariant, and moreover the arrows $\alpha_i, i\in \{0,...,n\}$ are all homogeneous with respect to the structure of $T$--modules of the vector spaces $\Hom_{K^b(\cat)}(\lambda_i \cdot \1, F)$. 

\end{example}

\begin{example}
\label{ex:equivrelative} 

Finally, let us treat the relative case of Subsection \ref{sub:relative}. Suppose we have a $T$--equivariant map:
$
\cat \stackrel{\iota}\longrightarrow X
$
and we wish to upgrade it to a $T$--equivariant map:
$$
\cat \stackrel{\iota'}\longrightarrow \PP \CV^{\vee}
$$
where $\CV$ is a $T$--equivariant vector bundle on $X$. As we saw in Subsection \ref{sub:relative}, the existence of the map $\iota'$ is equivalent to the choice of an object $F \in \cat$ together with an arrow:
$$
\iota^* \CV \stackrel{\alpha}\longrightarrow F
$$
in $\cat$, whose Koszul complex is quasi-isomorphic to 0. It is easy to see that the map $\iota'$ is $T$--equivariant if and only if the map $\alpha$ is $T$--equivariant. The same picture applies when $\CV$ is replaced by a coherent sheaf $\CQ$ of homological dimension 1, as in Subsection \ref{sub:one}.

\end{example}

\subsection{Categorical diagonalization}
\label{sub:cat diag}

In \cite{EH}, Elias and Hogancamp developed a theory of categorical diagonalization, which we will now recall. Assume we are given an equivariant monoidal category $T \curvearrowright \cat$, which can be taken to be triangulated or dg. %The main example for us will be $T = \CC^*$ and $\cat = \SBim_n$ or $\coh(\PP^n)$. 

\begin{definition}(\cite{EH})
\label{def:EH1}
Fix an object $F \in K^b(\cat)$. An arrow:
\begin{equation}
\label{eqn:eigenmap}
\lambda\cdot \1 \stackrel{\alpha}\longrightarrow F
\end{equation}
is called an {\bf eigenmap} of $F$, and the grading shift $\lambda \in T^\vee$ is called an {\bf eigenvalue} of $F$.  
\end{definition}

\begin{definition}(\cite{EH})
\label{def:EH2}
An object $F \in K^b(\cat)$ is called {\bf diagonalizable} if it has a collection of eigenvalues $\lambda_0,...,\lambda_n \in T^\vee$ and eigenmaps:
$$
\Big\{ \lambda_i \cdot \1 \stackrel{\alpha_i}\to F \Big\}_{i\in \{0,...,n\}}
$$
such that $\otimes_{i=0}^n \cone(\alpha_i) \simeq 0$. 
\end{definition}

The intuition behind the above terminology comes about by considering the Grothendieck group $[\cat]$, which is an algebra because the category $\cat$ is monoidal. Multiplication by the class of the object $[F]$ induces an operator on $[\cat]$, and the datum of Definition \ref{def:EH2} amounts to:
\begin{equation}
\label{eqn:charpoly}
\prod_{i=0}^n \left( [F]-\lambda_i \right) = 0
\end{equation}
In other words, the condition that the product of the cones of the eigenmaps is 0 amounts to requiring the operator $* \leadsto * \cdot [F]$ to solve its characteristic polynomial. In Lemma \ref{lem:universal}, we establish the fact that categorical diagonalization is universally represented by the category:
$$
\CD = D^b(\coh_T(\PP^n_A))
$$
where $A$ is any commutative ring and $T \curvearrowright \PP_A^n$ acts via: 
\begin{equation}
\label{eqn:action}
t \cdot [z_0:...:z_n] \mapsto \left[ \frac {z_0}{\lambda_0(t)}:...: \frac {z_n}{\lambda_n(t)} \right]
\end{equation}
where $\lambda_0,...,\lambda_n \in T^\vee$. An immediate generalization of Example \ref{ex:projective} yields the following:

\begin{lemma}
\label{lem:universal} 

The datum of a diagonalizable object $F \in \cat$ as in Definition \ref{def:EH2} is equivalent to the existence of a $T$--equivariant map: 
$$
\iota:\cat\to \PP^n_{A}
$$
such that $F = \iota^*\left(\CO(1)\right)$, where $A = \End_\cat(\1)$. 
\end{lemma}

\subsection{Eigenobjects}
\label{sub:eigenobjects}

In Definition \ref{def:EH1} we have recalled the categorical version of eigenvalues. In \cite{EH}, the authors complete the picture by categorifying eigenvectors: 
 
\begin{definition}
\label{def:EH3}
If for some $P \in \cat$ the arrow: 
\begin{equation}
\label{eqn:eigenobject}
\alpha \otimes \Id_P : \lambda\cdot P \stackrel{\cong}\longrightarrow F\otimes P
\end{equation}
is an isomorphism, then we call $P$ an \textbf{eigenobject} for the datum of Definition \ref{def:EH1}.
\end{definition}

In the decategorified world, the eigenvectors of the operator of multiplication by $[F]$ of \eqref{eqn:charpoly} can be computed explicitly, essentially by the Lagrange interpolation formula:
\begin{equation}
\label{eqn:idempotent0}
[P_i] := \prod_{0\leq j \neq i \leq n} \frac {\lambda_j - [F]}{\lambda_j-\lambda_i} %\frac {1 - \lambda_j/[F]}{1 - \lambda_j/\lambda_i}
\end{equation}
The reason why we divide by $\lambda_j-\lambda_i$ is to ensure that the elements $[P_i]$ are idempotents. However, this comes at the cost of enlarging the algebra to account for such denominators. One of the main constructions in \cite{EH} is the categorify formula \eqref{eqn:idempotent0} in a way which keeps track of the eigenmaps. %To do so, one must first categorify idempotents.

%\begin{definition}(\cite{BD}, \cite{Hog}) An object $P \in \cat$ is called a categorical idempotent if there exists a morphism $\nu:\1\to P$ such that:
%$$
%\nu\otimes \Id_P:P\to P\otimes P, \qquad \Id_P\otimes \nu:P\to P\otimes P
%$$
%are isomorphisms. 
%\end{definition}

%We refer to \cite[Section 4]{Hog} for a detailed discussion of categorical idempotents. Indeed, $P$ is isomorphic to $P\otimes P$, but in this definition $\nu$ gives a preferred isomorphism. This allows one to prove the following general result:

%\begin{theorem}(\cite{BD}, \cite{Hog})
%\label{th:projectors}
%Let $P$ be a categorical idempotent. Then the following hold:
%\begin{enumerate}
%\item (Theorem 4.9) The category $P\cat P$ is monoidal with monoidal identity $P$.
%\item (Corollary 4.15) Precomposition with $\nu$ gives an isomorphism $\e(P)\simeq \Hom(\1,P)$.
%\item (Theorem 4.19) $\e(P)$ is a commutative $\e(\1)$--algebra with unit given by $\nu$.
%\end{enumerate}
%\end{theorem}

%\subsection{Homological completions}
%\label{sub:hom}

The main difficulty, which we will shortly address, is how to lift the denominators of \eqref{eqn:idempotent0} from the Grothendieck group to the category $\cat$. The idea spelled out in \cite{EH} is that in \eqref{eqn:idempotent0} one should expand:
$$
\frac {\lambda_j - [F]}{\lambda_j-\lambda_i} = \left(1 - \frac {[F]}{\lambda_j} \right)\left(1 + \frac {\lambda_i}{\lambda_j} + \frac {\lambda_i^2}{\lambda_j^2} + ... \right)
$$
if $j<i$ and:
$$
\frac {\lambda_j - [F]}{\lambda_j-\lambda_i} = \left(\frac {[F]}{\lambda_i} - \frac {\lambda_j}{\lambda_i} \right)\left(1 + \frac {\lambda_j}{\lambda_i} + \frac {\lambda_j^2}{\lambda_i^2} + ... \right) 
$$
if $j>i$. To understand the above as an expansion of geometric series, we assume that there exists a distinguished subtorus $\CC^* \subset T$ which we will be called \textbf{homological}, such that:
\begin{equation}
\label{eqn:increasing}
\lambda_0|_{\CC^*} > ... > \lambda_n|_{\CC^*}
\end{equation}
To categorify these geometric series, \cite{EH} replace the category $\cat$ by its {\bf homological completion} $\cat^{\uparrow}$,
as in Section \ref{def:cat}.

% This will be properly defined in Definition \ref{def:homological} below, but it boils down to requiring $\bcat$ to contain certain infinite direct sums in the direction of decreasing degree $\lambda|_{\CC^*}$. Under these assumptions, we present the following slight reinterpretation of a construction from \cite{EH}: 

%$$
%\alpha_i \otimes \Id_{P_i} : \lambda_i \cdot P_i \stackrel{\cong}\rightarrow F\otimes P_i
%$$ 

\begin{theorem}(\cite{EH})
\label{thm:EH diagonalization}
Let $F$ be a diagonalizable object, with eigenmaps $\alpha_i$ and eigenvalues $\lambda_i$ satisfying \eqref{eqn:increasing}. Then there exist a collection of eigenobjects $P_i$ as in \eqref{eqn:eigenobject}, explicitly given by:
\begin{equation}
\label{eqn:idempotent}
\includegraphics{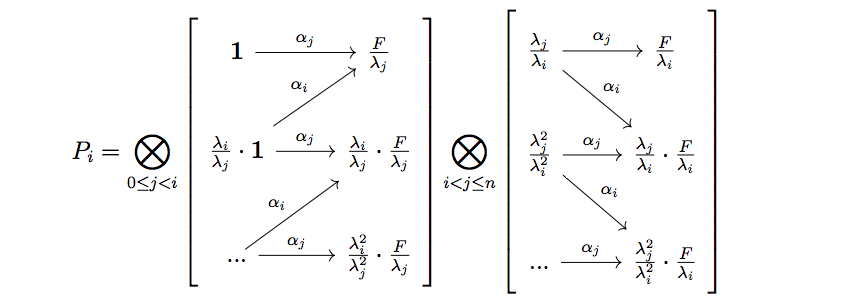}
\end{equation}
The objects should be added $\oplus$ along columns, with differentials according to the arrows. The collection $\{P_0,...,P_n\}$ yields a semi-orthogonal decomposition of $\bcat$:
\begin{equation}
\label{eqn:semi orthogonal}
\1 \cong \Big[P_0 \oplus ... \oplus P_n, \text{ a certain differential} \Big]
\end{equation}
and $\Hom_{\bcat}(P_i,P_j)=0$ if $i>j$. Furthermore, $P_i\otimes P_j\simeq 0$ for $i\neq j$ and $P_i\otimes P_i\simeq P_i$.
\end{theorem}
 
%We briefly outline the construction of \cite{EH}. Let $\alpha_i:\1\to F$ be the eigenmaps for $F$. Consider the complex 
%\begin{equation}
%\label{eq: def k}
%K_i:=\bigotimes_{j\neq i} \Cone(\alpha_j).
%\end{equation}
%One can prove that $K_i$ is an eigenobject for $F$ and a quasi-idempotent. The projector $P_i$ is an infinite complex made of $K_i$ with a certain connecting differential. On the level of $K$-group, \eqref{eq: def k} just descends to the standard formula for the projector to an eigenspace:
%$$
%[\Cone(\alpha_j)]=\lambda_j-[F],\ 
%[K_i]=\prod_{j\neq i}(\lambda_j-[F]),\ [P_i]=\prod_{j\neq i}\frac{\lambda_j-[F]}{\lambda_j-\lambda_i}.
%$$

%\begin{proposition}(\cite{EH})
%\label{prop:Koszul}
%The  algebra $\e(K_i)$ is a module over the exterior algebra with $r-1$ generators. The algebra $\e(P_i)$ (described in Theorem \ref{th:projectors}) is a Koszul dual module over the polynomial algebra with $r-1$ generators. 
%\end{proposition}

The main application of \cite{EH} is when $\cat = K^b(\SBim_n)$ is replaced by $\bcat = K^-(\SBim_n)$, and the homological $\CC^*$ action is by homological degree of chain complexes. We may generalize this particular case to the following setup.

\subsection{The geometric realization over a fixed base}
\label{sub:geometric diagonalization}

As we saw in Lemma \ref{lem:universal}, any categorical diagonalization in a category $\cat$ comes from a $T$--equivariant map:
$$
\cat \rightarrow \PP^n_A \qquad \text{i.e.} \qquad K^b(\cat) \xtofrom[\iota^*]{\iota_*} \CD
$$
where $\CD = D^b(\coh_T(\PP^n_A))$, and the action $T \curvearrowright \PP^n_A$ is given in \eqref{eqn:action}. The above functors extend to functors on the homological completions:
$$
K(\cat^{\uparrow}) \xtofrom[\iota^*]{\iota_*} \CD^{\uparrow}
$$
which are given by the same formulas, but allow infinite direct sums of objects in decreasing homological degree. Therefore, we have:
$$
P_i = \iota^*(\CP_i)
$$
where $\CP_i \in \CD^{\uparrow}$ are   given by formula \eqref{eqn:idempotent} with $F$ replaced by $\CO(1)$ and $\alpha_i$ replaced by multiplication with the homogeneous coordinate $z_i$:
\begin{equation}
\label{eqn:idempotent geometry}
\includegraphics{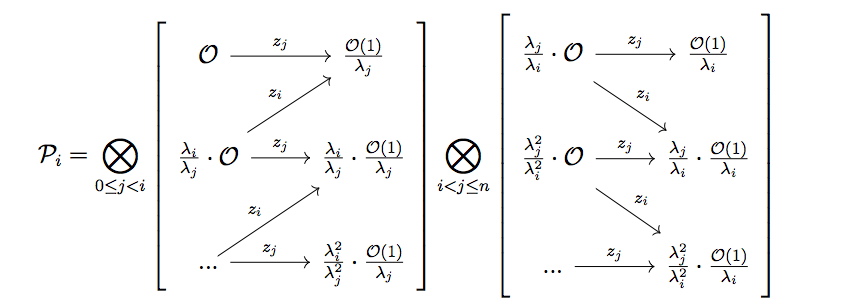}
\end{equation}
The rows in the above diagram make up for the expansion of the geometric series $(\lambda_j - \lambda_i)^{-1}$. Meanwhile, observe that the top row is precisely;
\begin{equation}
\label{eqn:quasi}
\text{top row of }\CP_i = \bigotimes_{j<i} \left[\CO \stackrel{z_j}\longrightarrow  \CO(1) \lambda_j^{-1}\right] \bigotimes_{j>i} \left[\lambda_j \lambda_i^{-1} \stackrel{z_j}\longrightarrow \CO(1)\lambda_i^{-1} \right] \stackrel{\qis}\cong \CO_{p_i} \prod_{j<i} \frac {\lambda_i}{\lambda_j}
\end{equation}
Here, $\CO_{p_i}$ is the structure sheaf of the torus invariant subscheme $p_i = [0:...:0:1:0:...:0] \in \PP^n_A$, which is a closed point if and only if $A$ is a field. The quasi-isomorphism in \eqref{eqn:quasi} is the standard one between the structure sheaf of $p_i$ and its Koszul complex. We conclude that the full idempotent \eqref{eqn:idempotent geometry} is a way to make sense of the denominators in the object:
\begin{equation}
\label{pi as a fraction}
\CP_i = \frac {\CO_{p_i}}{\prod_{0 \leq j \neq i \leq n} \left(1-\frac {\lambda_j}{\lambda_i} \right)} \in \bCD
\end{equation}
Recall from \eqref{eqn:semi orthogonal} that $\CP_0,...,\CP_n$ give a decomposition of the unit object in $\bCD$. This statement categorifies the fact that:
$$
[\CO] = \sum_{i=0}^n [\CP_i] = \sum_{i=0}^n \frac {[\CO_{p_i}]}{\prod_{0\leq j \neq i \leq n} \left(1 - \frac {\lambda_j}{\lambda_i} \right)}
$$
in the algebraic $K$--theory ring of $\PP^n_A$. The above is nothing but the Thomason equivariant localization formula, which is a very interesting result even in $K$--theory. At the categorical level, it is made even more interesting by the presence of the various differentials that appear in \eqref{eqn:semi orthogonal}, which give rise to a semi-orthogonal decomposition of the category $\bCD$. 

The denominator of \eqref{pi as a fraction} equals  the Poincar\'e series for the equivariant local ring of $\PP^n_A$ at $p_i$.
This is not a coincidence, and the relation between the two objects can be made more precise.

\begin{proposition}
\label{open vs proj}
Consider a locally closed subset $S_i=\{z_0=\ldots=z_{i-1}=0,z_i\neq 0\}\subset \PP^n_{A}$.  
Then $\CP_i$ is quasi-isomorphic to the pushforward of $S^{\bullet}(\nu^{\vee}_{S_i})$, where $\nu_{S_i}$ is the normal bundle to $S_i$.
 \end{proposition}
 
\begin{remark}
The ordering of coordinates in the definition of $S_i$ agrees with the ordering of eigenvalues of $\CO(1)$ (that is, the weights of the torus action) on $\PP^n$. It is easy to see that the strata $S_i$ agree with the cells in the Bia\l ynicki-Birula decomposition \cite{BB1,BB2} 
of $\PP^n$ with respect to this torus action. Similar decompositions of equivariant derived categories with respect to 
Bia\l ynicki-Birula strata were studied in \cite{DanHL}, and we plan to study the relation between the categorical diagonalization framework and \cite{DanHL} in the future work.
\end{remark}
 
\begin{proof}
To simplify the notations, we will consider the case $n=1$ and omit all the grading shifts (which can be easily reconstructed since all maps are homogeneous). The construction \eqref{eqn:idempotent geometry} yields two different infinite complexes built from the sections $z_0,z_1:\CO\to \CO(1)$. The first has a form:
$$ 
\includegraphics{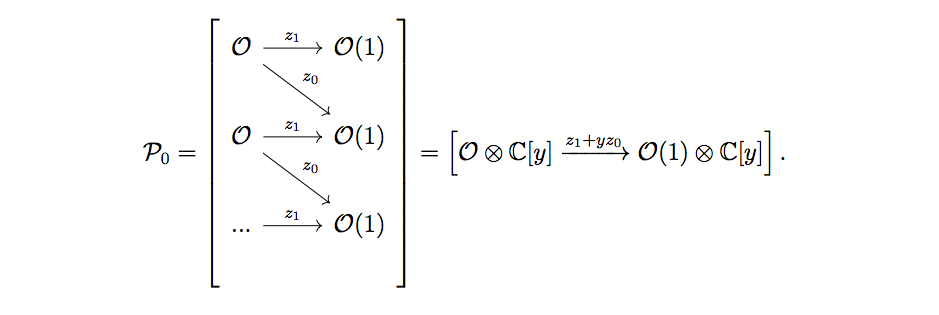}
$$
Here $y$ is a formal variable corresponding to the shift of the complex down by one unit. It can be made less formal by considering the projection $\pi:\PP^n\times \AA^1\to \PP^n$, so that
$$
\CP_0=p_*\left[\CO\xrightarrow{z_1+yz_0}\CO(1)\right]=p_*\CO_{\{z_1+yz_0=0\}}.
$$
The projection $p$ identifies the closed subset $\{z_1+yz_0=0\}\subset \PP^n\times \AA^1$ with  the open subset $S_0=\{z_0\neq 0\}\subset \PP^n$, so $\CP_0=\CO_{S_0}$. The second complex is more interesting. It has the form:
$$
\includegraphics{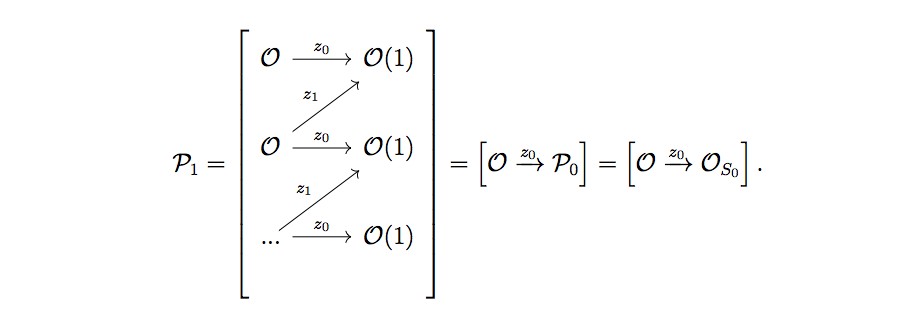}
$$
It is supported on $S_1=\PP^1\setminus S_0=\{z_0=0\}$ where the stalk of $\CO$ is isomorphic to $\CC[\frac{z_0}{z_1}]$ and the stalk of
$\CO_{S_0}$ is isomorphic to $\CC[\frac{z_0}{z_1},\frac{z_1}{z_0}]$, so the quotient is isomorphic to $\frac{z_1}{z_0}\cdot \CC[\frac{z_1}{z_0}].$
\end{proof} 

\begin{remark}
Note that 
$$
\includegraphics{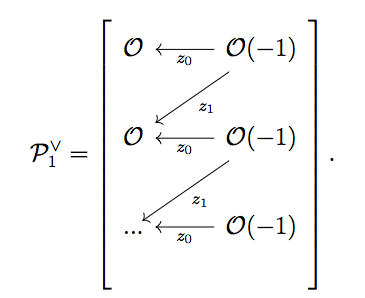}
$$
One can use similar arguments to formally match this complex with $\CO_{\{z_1\neq 0\}}\otimes \CO(-1)=\CO_{\{z_1\neq 0\}}$.
However, $\CP_1^{\vee}$ does not belong to the category $\CD^{\uparrow}$ since the gradings of its summands are unbounded.
\end{remark}

\begin{corollary}
The endomorphism ring of $\CP_i$ is isomorphic to the local ring of $\PP^n_{A}$ at a fixed point $p_i$.
\end{corollary}

\begin{proof}
We follow the proof of Proposition \ref{open vs proj}. Indeed, $\End(\CP_0)=H^0(S_0,\CO_{S_0})=\CC[\frac{z_1}{z_0}]$. On the other hand,
$$
\End(\CP_1)=\End\left[\CC\left[\frac{z_0}{z_1}\right]\rightarrow \CC\left[\frac{z_1}{z_0},\frac{z_0}{z_1}\right]\right]=\CC\left[\frac{z_0}{z_1}\right].
$$
One could also argue that 
$$
\End(\CP_1)=\End(\CP_1^{\vee})=\End(\CO_{\{z_1\neq 0\}})=\CC\left[\frac{z_0}{z_1}\right].
$$
The proof for general $n$ is analogous.
\end{proof}

\begin{remark}
Proposition \ref{open vs proj} shows that the endomorphism rings of the projectors can be interpreted as the rings of functions on certain open charts. This point of view will be important in the next section where we define some open charts on the flag Hilbert scheme and compute the rings of functions on them (up to a certain completion). By Conjecture \ref{conj:1} and the preceding discussion these rings match the homology of the categorified Jones-Wenzl projectors. 
\end{remark}

\begin{remark}
\label{rem:what if}
The equivariant localization formula makes sense when $\CD = D^b(\coh_T(X))$ for any local complete intersection $X$ acted on by a torus $T$:
\begin{equation}
\label{eqn:equivloc}
[\CO_X] = \sum_{p \in X^T} \frac {[\CO_p]}{\wedge^\bullet \left( \Tan_p^\vee X \right)} 
\end{equation}
As we have seen, when $X = \PP^n$ the above setup encodes categorical diagonalization as in Definition \ref{def:EH2} and \ref{def:EH3}. It would be very interesting to determine which problems in ``categorical linear algebra" are encoded by formula \eqref{eqn:equivloc} for more general schemes $X$. 

\end{remark}

\subsection{The relative case}
%\label{sub:relative}

For the remainder of this Section, we will generalize the objects \eqref{eqn:idempotent geometry} from $\PP^n_A$ to projective bundles $\PP \CV^\vee$ on an arbitrary base scheme $X$, as in Example \ref{ex:equivrelative}. We assume that both $X$ and $\CV$ are acted on by a torus $T$, and that we have a decomposition:
\begin{equation}
\label{eqn:bill}
\CO_X \cong \left[ \bigoplus_{x\in X^T} \CP_x, \text{ a certain differential}\right] \in \overline{D^b(\coh_T(X))}
\end{equation}
where the indexing set goes over the fixed points of $X$. We assume that the above is semi-orthogonal, in the sense that $\Hom(\CP_x, \CP_y) = 0$ whenever $x>y$ with respect to some total order. We wish to upgrade the decomposition \eqref{eqn:bill} to the projective bundle $\PP \CV^{\vee}$. 

\begin{proposition}
\label{prop:decomposition projective}

Let $n + 1 = \emph{rank }\CV$. There exist objects $\CP_i^x$ for all $i \in \{0,...,n\}$ and $x\in X^T$, such that we have a semi-orthogonal decomposition:
\begin{equation}
\label{eqn:decomposition}
\CO_{\PP \CV^{\vee}} \cong \left[ \bigoplus^{0\leq i \leq n}_{x\in X^T} \CP^i_x, \text{ a certain differential}\right] \in \overline{D^b(\coh_T(\PP \CV^\vee))}
\end{equation}
whenever the homological subtorus $\CC^* \subset T$ acts with distinct weights in the fibers $\CV|_x$ for all $x \in X^T$. We have $\Hom(\CP_x^i, \CP_y^j) = 0$ if $x>y$ or if $x=y$ and $i<j$.

\end{proposition}

%\begin{equation}
%\label{eqn:decomposition}
%\CO_{\PP \CE^{\vee}} \cong \left[ \begin{tikzcd}
%P_0^x \arrow{d}  \arrow[bend right]{ddd}  \\
%P_1^x[1] \arrow{d} \arrow[bend left]{dd} \\
%... \arrow{d} \\
%P_n^x[n]
%\end{tikzcd} \ , \ \delta^x_{ij}:P^x_i \rightarrow P^x_j[1] \right]
%\end{equation}

The object $\CP_x^i$ is precisely of the form \eqref{eqn:idempotent geometry} if one replaces $\CO$ with $\CP_x$, and $\lambda_0,...,\lambda_n$ are precisely the weights of the torus $T$ in the fiber $\CV|_x$.

\section{Local charts and fixed points of $\FH_n$}
\label{sec:local}

\subsection{Affine charts for Hilbert schemes}
\label{sub:affinecharts1}

Recall the action of $\torus$ on Hilbert schemes given by rescaling the $X$ and $Y$ matrices. The fixed points of this action on the Hilbert scheme are well-known. They are given by monomial ideals, which are indexed by partitions of $n$:
$$
\H_n^{\torus} = \{I_\lambda\}_{\lambda \vdash n}, \qquad I_\lambda = (x^{\lambda_1}, x^{\lambda_2}y,...)\subset \CC[x,y]
$$
Haiman described a set of affine charts on the Hilbert scheme, each of which is $\torus$ invariant and contains a single fixed point:
\begin{equation}
\label{eqn:chart0}
\H_n = \bigcup_{\lambda \vdash n} \oH_\lambda
\end{equation}
where:
\begin{equation}
\label{eqn:haimanchart}
\oH_\lambda = \Big \{I \text{ such that } \{x^ay^b\}_{(a,b) \in \lambda} \text{ is a basis of } \CC[x,y]/I\Big\}
\end{equation}
Here and throughout this paper, we identify a partition with its Young diagram, which is the set of $1 \times 1$ boxes in the first quadrant of the plane with coordinates $(a,b) \in \NN_0 \times \NN_0$, $b < \lambda_{a}$:

\begin{picture}(100,150)(-130,-20)
\label{fig0}

%\put(17,17){$1$}
%\put(15,57){$t^{-1}$}
%\put(17,97){$t^{-2}$}
%\put(57,17){$q$}
%\put(53,57){$qt^{-1}$}
%\put(97,17){$q^2$}
%\put(92,57){$q^2t^{-1}$}
%\put(137,17){$q^3$}

\put(0,0){\line(1,0){160}}
\put(0,40){\line(1,0){160}}
\put(0,80){\line(1,0){120}}
\put(0,120){\line(1,0){40}}

\put(0,0){\line(0,1){120}}
\put(40,0){\line(0,1){120}}
\put(80,0){\line(0,1){80}}
\put(120,0){\line(0,1){80}}
\put(160,0){\line(0,1){40}}

\end{picture}

\noindent For example, the above Young diagram corresponds to the partition $\lambda = (4,3,1)$. It would be very nice to have a clear description of the algebra of functions on each affine chart \eqref{eqn:chart0}, but this is not at all easy. Instead, Haiman's construction gives us a set of generators:
$$
\{f_1,...,f_{2n} \} \in \fm_\lambda/\fm^2_\lambda
$$
where $\fm_\lambda \in \CC[\oH_\lambda]$ denotes the maximal ideal of the fixed point $I_\lambda$. 

\subsection{Affine charts for flag Hilbert schemes}
\label{sub:affinecharts2}

The situation is somewhat better in the case of flag Hilbert schemes $\FH_n(*)$ for any $\star$, where one has affine coverings: 
\begin{equation}
\label{eqn:chart}
\FH_n(*) = \bigcup_{T \vdash n} \oFH_T(*)
\end{equation}
indexed by standard Young tableaux $T$ of size $n$. Recall that a standard Young tableau is a numbering of the boxes of a Young diagram of size $n$ with the numbers $1,...,n$ such that the numbers increase as we go up and right in the diagram. A covering \eqref{eqn:chart} is called \textbf{good} if all the charts are $\torus$ equivariant and it respects passage from $n+1$ to $n$:
$$
\includegraphics{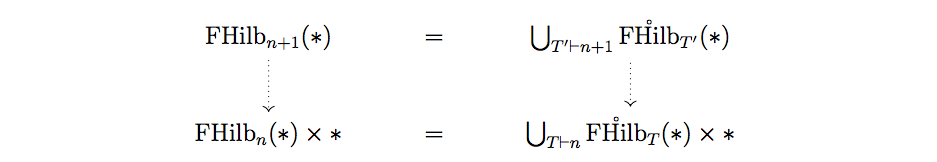}
$$
where the chart corresponding to any $T'$ maps to the chart corresponding to $T = T' \backslash \square_{n+1}$. Here, $\square_{n+1}$ denotes the box labeled $n+1$ in $T$, which must necessarily be an outer corner of $T$ and an inner corner of $T'$. Restricting the sheaf of dg algebras $\FH_n^\dg(*)$ to the affine charts \eqref{eqn:chart} gives rise to dg algebras:
\begin{equation}
\label{eqn:coordinaterings}
\oCA^\dg_T(*) = \CC\left[\oFH^\dg_T(*)\right]
\end{equation}

\begin{conjecture}
\label{conj:charts}

There exists a good covering whose coordinate rings \eqref{eqn:coordinaterings} satisfy:
\begin{equation}
\label{eqn:genrel}
\oCA_{T \cup \sq}(*) = \frac {\oCA_{T}(*)[*, f_{\sq_1},f_{\sq_2},...]}{(r_{\bsq_1},r_{\bsq_2},...)}
\end{equation}
where $\square_1,\square_2,...$ denote the inner corners of $T$ different from $\square$, and $\blacksquare_1,\blacksquare_2,...$ denote the outer corners of $T$ (except for the outer corner labeled $n$ in the case $* = \point$). The generators denoted by $*$ stand for the affine coordinates $\{x_{n+1},y_{n+1}\}$, $\{x_{n+1}\}$, $\emptyset$ when $* = \CC^2, \CC, \point$.

\end{conjecture}

We do not know how to define the generators $f_k$ and the relations $r_k$, but we know how to predict their characters with respect to the $\torus$ action. Specifically, for a box $\square = (a,b)$ in a Young diagram, we define its weight as:
\begin{equation}
\label{eqn:weight}
z_\square = q^a t^b
\end{equation}
When $\bsq$ is the box labeled by $i$ in a Young tableau $T$, we will write $z_\bsq = z_i$ for brevity. Then we expect that the generators and relations of \eqref{eqn:genrel} have equivariant weights 
\begin{equation}
\label{eqn:expectation}
\text{weight }f_{\bsq} = \frac {z_{\bsq}}{z_{\square}}, \qquad \text{weight }r_{\bsq} = \frac {z_{\bsq}}{z_{\square}}
\end{equation}
where $\sq$ is the corner that is being added in \eqref{eqn:genrel}. In the remainder of this Section, we will establish a weaker version of Conjecture \ref{conj:charts}, by constructing affine $\torus$ invariant open sets that contain the fixed points of $\FH_n(*)$, but are not required to cover it.

% Still, even this result will be enough to prove the following Proposition, since the graded Poincar\'e series of a ring is equal to that of its localizations.

\subsection{Defining the charts}
\label{sub:defcharts}

In this Section, we will define affine charts on the flag Hilbert scheme which only satisfy Conjecture \ref{conj:charts} on the local rings around the fixed points. $\FH_n$ will henceforth refer to either of $\FH_n(*)$ for $* \in \{\CC^2,\CC,\point\}$. 

\begin{definition}
\label{def:chart}

For any point $(X,Y,v) \in \FH_n$ and standard Young tableau $T$, consider the following algorithm to construct a basis $e_1 = v, e_2,...,e_n$ of $\CC^n$. Suppose $e_1,...,e_{k-1}$ have been constructed and the $k$-th box looks as in the following picture:

\begin{picture}(100,90)(-170,-5)
\label{pic:1}

%\put(17,17){$k'$}
\put(15,57){$i$}
\put(57,17){$i'$}
\put(53,57){$k$}
\put(0,0){\line(1,0){80}}
\put(0,40){\line(1,0){80}}
\put(0,80){\line(1,0){80}}
\put(0,0){\line(0,1){80}}
\put(40,0){\line(0,1){80}}
\put(80,0){\line(0,1){80}}

\end{picture}

\noindent Define the vector $e_k \in \Ker(\CC^n \twoheadrightarrow \CC^{k-1})$ by the following formula if $i>i'$:
\begin{equation}
\label{eqn:chartx}
X e_{i} = e_k + \sum_{j=i}^{k-1} x^j_i e_j 
\end{equation}
where $x^j_i$ are coefficients, and by the following formula if $i<i'$:
\begin{equation}
\label{eqn:charty}
Y e_{i'} = e_k + \sum_{j={i'}}^{k-1} y^j_{i'} e_j 
\end{equation}
where $y^j_{i'}$ are coefficients. If the process terminates after having constructed $e_n$, in a way such that $e_1,...,e_k$ form a basis of the quotient $\CC^n \twoheadrightarrow \CC^k$ for all $k$, then we set:
$$
(X,Y,v) \in \FH_T
$$

\end{definition}

In either \eqref{eqn:chartx} or \eqref{eqn:charty}, it is clear that the vector $e_k$ is unique, since the coefficients $x_i^j$ or $y_i^j$ are uniquely determined by the fact that $e_k$ vanishes in the quotient $\CC^n \twoheadrightarrow \CC^{k-1}$. The fact that such an $e_k$ exists at each step, and that the resulting collection of vectors forms a basis, is an open condition and therefore:
$$
\FH_T \subset \FH_n
$$
thus defined is an open subscheme. It is also an affine subscheme, simply because the basis $e_1,...,e_n$ is unique. We could therefore define $\FH_T$ alternatively as the affine space of matrices $X,Y$ of the form prescribed by \eqref{eqn:chartx} and \eqref{eqn:charty} in a fixed basis. It is also clear that the locus $\FH_T$ is $\torus$ invariant and that the only fixed point it contains is: 
$$
I_T = \left\{\CC^n = \bigoplus_{i=1}^n \CC \cdot e_i \quad \text{with} \quad  X\cdot e_i = e_{i \rightarrow}, \ Y\cdot e_i = e_{i\uparrow}, \ v = e_1 \right\}
$$
In the above formula, for any box $i \in T$ we write $i \rightarrow$ and $i \uparrow$ for the boxes immediately right and above $\sq$, respectively. If there is no box to the right or up of $\sq$, we naturally set $e_{i \rightarrow}$ or $e_{i\uparrow}$ equal to 0. The fact that the open sets of Definition \ref{def:chart} cover the whole of $\FH_n$ follows from the following principle:
\begin{equation}
\label{eqn:principletorus}
\textbf{any open torus invariant property which holds}
\end{equation}
$$
\textbf{near the fixed points of }\FH_n \textbf{ holds everywhere}
$$
This is because the set of points which do not enjoy said property is closed, torus invariant and contains no fixed points: any such set must be empty. One must be careful here, because the argument is a priori only true for projective varieties, such as $\FH_n(\point)$. However, it also applies to $\FH_n(\CC)$ and $\FH_n(\CC^2)$ because the torus $\torus$ contracts the affine directions $\CC$ and $\CC^2$ to the origin.

\subsection{The special coefficients}
\label{sub:special}

Note that the coefficients $x_i^i$ and $y_i^i$ in \eqref{eqn:chartx} and \eqref{eqn:charty} are precisely the eigenvalues of the matrices $(X,Y,v) \in \FH_n$. If we are in the case $* = \CC$ or $* = \point$, then we must set $y_i^i = 0$ or $x_i^i = y_i^i = 0$ in \eqref{eqn:chartx} and \eqref{eqn:charty}, respectively.

\begin{definition}
\label{def:specialcoeffs}

The coefficients $x_i^j$ and $y_i^j$ which appear in \eqref{eqn:chartx} and \eqref{eqn:charty} will be called \textbf{special coefficients}. We also apply this terminology to the case when $k$ is an outer corner of the Young diagram of $T$, but in that case \eqref{eqn:chartx} and \eqref{eqn:charty} hold with $e_k = 0$. 

\end{definition}

Note that the number of special coefficients corresponding to a standard Young tableau $T$ is:
\begin{equation}
\label{eqn:numberofgens}
\sum_{i=1}^n \# \left( \text{of inner corners of the Young diagram consisting of the boxes labeled } 1,...,i \right)
\end{equation}
Conjecture \ref{conj:charts} would suggest that the special coefficients generate the dg ring of functions $\oCA_T$ subject to a number of: 
\begin{equation}
\label{eqn:numberofrels}
\sum_{i=1}^n \# \left( \text{of outer corners of the Young diagram consisting of the boxes labeled } 1,...,i \right)
\end{equation}
However, this is not true, because this would entail that all coefficients $x_i^j$ and $y_i^j$ could be written as polynomials in the special coefficients. We partially salvage this in the next Subsection, when we will show that the previous sentence holds if we replace the word ``polynomials" by ``rational functions". In other words, some open subset of $\FH_T$ can be described by \eqref{eqn:numberofgens} generators and \eqref{eqn:numberofrels} relations.

\begin{example}
\label{ex:symmetric}
When $T= (n)$ and $* = \CC$, only relations \eqref{eqn:chartx} come into play:
$$
Xe_i = e_{i+1} + x_i e_i
$$
unless $i=1$, in which case we have:
$$
Y e_1 = \sum_{j=2}^n y_1^j e_j
$$
Therefore, the special coefficients are $\{x_i, y_1^j\}^{1\leq i \leq n}_{2\leq j \leq n}$. The number of these coefficients is $2n-1$, and it matches \eqref{eqn:numberofgens} minus 1, where the minus one stems from the fact that $y_1^1 = 0$ for $* = \CC$. The non-special coefficients are the $y_i^j$ with $i>1$, but they can be inferred from the special ones via the commutation relation $[X,Y]=0$, which in the case at hand reads:
\begin{equation}
\label{eqn:above}
y^j_i (x_i-x_j) = y^j_{i+1} - y^{j-1}_i
\end{equation}
for all $i<j$. Note that \eqref{eqn:above} is precisely \eqref{eqn:sym}.  We make the convention that $y_i^j=0$ for $j\leq i$. After solving for $y^j_i$ in terms of $\{x_i, y_1^j\}$, we obtain the inductive formulas for any $\delta>0$:
$$
y_i^{i+\delta} = y_1^{\delta+1} + \sum_{s=1}^{i-1} y_{i-s}^{i-s+\delta+1} (x_{i-s}-x_{i-s+\delta+1})
$$
The above relation also holds when $i+\delta=n+1$, in which case the left hand side is 0. We therefore obtain a relation among the special coefficients $\{x_i,y_1^j\}$ for all $\delta>0$. There are $n-1$ such relations, and their number matches \eqref{eqn:numberofrels} minus 1, where the minus one stems from the fact that $y_1^1 = 0$ for $* = \CC$.

\end{example}

\begin{example}
\label{ex:antisymmetric}

When $T= (1,...,1)$ and $* = \CC$, only relations \eqref{eqn:charty} come into play:
$$
Ye_i = e_{i+1}
$$
unless $i=1$, in which case we have:
$$
X e_1 = \sum_{j=1}^n x_1^j e_j
$$
Therefore, the special coefficients are $x_1^j$ for all $j$. Note that the commutation relation $[X,Y]=0$ implies that:
$$
x_i^{j-1} = x_{i+1}^j \qquad \forall i<j
$$
and therefore we conclude that $x_i^j = u_{j-i+1}$ for some variables $u_1,...,u_n$. Compare with \eqref{eqn:antisym}. 

\end{example}

%\begin{proposition} Formulas \eqref{eqn:poincare1}--\eqref{eqn:poincare3} hold for our choice of charts. \end{proposition}

%\begin{proof} In order for \eqref{eqn:slice} to be a $\torus$ equivariant equality, we must have: 
%$$
%\text{weight } e_i = z_i
%$$
%where $z_i$ is the weight of the box labeled $i$, as defined in \eqref{eqn:weight}. The coefficients of the matrices $X$ and $Y$ therefore have weights:
%$$
%\text{weight } x_{ji} = q \frac {z_i}{z_j}, \qquad \text{weight } y_{ji} = t \frac {z_i}{z_j}
%$$
%Assuming that $\FH_n(*)$ is a local complete intersection, which we will prove in the remainder of this Section, then the commutation relations have weights:
%\begin{equation}
%\label{eqn:commutator}
%\text{weight } [X,Y]_{ji} = q t \frac {z_i}{z_j}
%\end{equation}
%Meanwhile, \eqref{eqn:slice} requires the $j-$th coefficient of $X^aY^bv$ to be 1 when $j=i$ and $0$ when $j>i$. This imposes extra equations with weights:
%$$
%\frac {z_i}{z_j} \qquad \forall \ j > i
%$$
%We conclude that the Poincar\'e polynomial of the chart $\oFH_T(*)$ matches formulas \eqref{eqn:poincare1}--\eqref{eqn:poincare3}. In the latter formula, the reason why we place factors $(1-qtz_i/z_{i+1})$ in the denominator is because relation \eqref{eqn:commutator} does not need to be imposed for $j=i+1$. Indeed, the commutator of two strictly lower triangular matrices automatically has all 0 on the sub--diagonal. 

%\end{proof}

\subsection{Explicit local coordinates}
\label{sub:localcoord}

In this section, we will use the special coefficients to describe the neighborhood of the fixed point $I_T$ for any standard Young tableau $T$:
\begin{equation}
\label{eqn:localization}
\oFH_T := \left( \FH_n \right)_{\text{localized at }T}
\end{equation}
and the dg local ring $\oCA^\dg_T = \CC [\oFH^\dg_T ]$. In fact, we will actually describe an open subscheme of $\FH_T$ given by the non-vanishing of certain torus invariant functions. The resulting open subschemes also form a cover of $\FH_n$ because of the principle \eqref{eqn:principletorus}, so we abuse notation and use \eqref{eqn:localization} both for the local neighborhood and for the open subscheme $\oFH_T \subset \FH_T$.  

\begin{proposition}
\label{prop:chart} 

For any standard Young tableau $T \vdash n$, the complex $\CE_n$ of \eqref{eqn:complex} is:
\begin{equation}
\label{eqn:localcomplex}
\CE_n|_{\oFH_T} \stackrel{\qis}\cong \left[\bigoplus^{\bsq \text{ outer}}_{\text{corner of }T} \CO \cdot e_\bsq \stackrel{\psi}\longrightarrow \bigoplus^{\sq \text{ inner}}_{\text{corner of }T} \CO \cdot f_\sq \right]
\end{equation}
Theorem \ref{thm:complex} describes the map $\pi: \FH_{n+1} \rightarrow \FH_n$ as the projectivization of $H^0(\CE_n)$. Locally, this map takes the form:
$$
\pi^{-1}\left(\oFH_T \right) = \bigcup^{\sq \text{ inner}}_{\text{corner of }T} \oFH_{T \cup \sq}
$$
where $\oFH_{T \cup \sq} \subset \PP H^0\left( \CE^{\vee}_n|_{\oFH_T} \right)$ is the affine chart of \eqref{eqn:localcomplex} given by $f_\sq = 1$. We conclude \eqref{eqn:genrel}, where the generators are $f_{\sq'}$ for inner corners $\sq'\neq \sq$ and the relations are $r_{\bsq} = \psi(e_\bsq)$.

\end{proposition}

\begin{proof}

From each box in $T$, draw two lines of unit length, one going up and one to the right:

\begin{picture}(70,155)(-140,-5)
\label{fig}

\put(7,6){$1$}
\put(7,46){$3$}
\put(7,86){$4$}
\put(47,6){$2$}
\put(47,46){$5$}
\put(87,6){$6$}
\put(87,46){$8$}
\put(127,6){$7$}

\put(0,0){\line(1,0){160}}
\put(0,40){\line(1,0){160}}
\put(0,80){\line(1,0){120}}
\put(0,120){\line(1,0){40}}

\put(0,0){\line(0,1){120}}
\put(40,0){\line(0,1){120}}
\put(80,0){\line(0,1){80}}
\put(120,0){\line(0,1){80}}
\put(160,0){\line(0,1){40}}

\multiput(20,100)(8,0){5}{\color{red}{\line(1,0){4}}}
\multiput(60,60)(8,0){5}{\line(1,0){4}}
\multiput(60,20)(0,8){5}{\line(0,1){4}}
\multiput(100,60)(0,8){5}{\color{red}{\line(0,1){4}}}
\multiput(140,20)(0,8){5}{\color{red}{\line(0,1){4}}}

\linethickness{1mm}
\put(20,20){\line(1,0){120}}
\put(140,20){\color{red}{\line(1,0){40}}}
\put(20,20){\line(0,1){80}}
\put(20,100){\color{red}{\line(0,1){40}}}
\put(100,20){\line(0,1){40}}
\put(100,60){\color{red}{\line(1,0){40}}}
\put(60,60){\color{red}{\line(0,1){40}}}
\put(20,60){\line(1,0){40}}

\end{picture}

%\begin{picture}(70,155)(-140,-5)
%\label{fig}

%\put(7,6){$1$}
%\put(7,46){$3$}
%\put(7,86){$4$}
%\put(47,6){$2$}
%\put(47,46){$5$}
%\put(87,6){$6$}
%\put(87,46){$8$}
%\put(127,6){$7$}

%\put(0,0){\line(1,0){160}}
%\put(0,40){\line(1,0){160}}
%\put(0,80){\line(1,0){120}}
%\put(0,120){\line(1,0){40}}

%\put(0,0){\line(0,1){120}}
%\put(40,0){\line(0,1){120}}
%\put(80,0){\line(0,1){80}}
%\put(120,0){\line(0,1){80}}
%\put(160,0){\line(0,1){40}}

%\multiput(20,100)(8,0){5}{\line(1,0){4}}
%\multiput(60,60)(8,0){5}{\line(1,0){4}}
%\multiput(60,20)(0,8){5}{\line(0,1){4}}
%\multiput(100,60)(0,8){5}{\line(0,1){4}}
%\multiput(140,20)(0,8){5}{\line(0,1){4}}

%\linethickness{1mm}
%\put(20,20){\line(1,0){160}}
%\put(20,20){\line(0,1){120}}
%\put(100,20){\line(0,1){40}}
%\put(100,60){\line(1,0){40}}
%\put(60,60){\line(0,1){40}}
%\put(20,60){\line(1,0){40}}

%\end{picture}

\noindent The lines are of two types: thick or dotted, and black or red. The color of a line is determined by whether the line points to a box in $T$ or outside of $T$. The shape of a line is determined by the following rule: If $i>i'$ where $i'$ is the label of the box to the southeast (respectively northwest) of $i$, then we make the horizontal (respectively vertical) line starting at $i$ thick. All the boxes below and to the left of the diagram are thought to have label $0$ for the purpose of this rule, and all the boxes above and to the right of the diagram are thought to have label $\infty$. By definition:
\begin{equation}
\label{eqn:localcomplex2}
\CE_n = \Big[ qt\CT_{n} \stackrel{\Psi}\longrightarrow q \CT_n \oplus t \CT_n \oplus \CO \stackrel{\Phi}\longrightarrow \CT_n \Big]
\end{equation}
When we restrict the complex to the affine chart $\oFH_T$, we observe that the tautological bundles are already trivialized by the basis $e_1,...,e_n$ of Definition \ref{def:chart}:
$$
\CT_n |_{\oFH_T} = \CO \cdot e_1 \oplus ... \oplus \CO \cdot e_n
$$
Therefore, the middle term of \eqref{eqn:localcomplex2} has a basis which we will denote by $e_1,...,e_n, e_1',...,e_n',1$. We claim that the projection that forgets some of these basis vectors induces an isomorphism:
\begin{equation}
\label{eqn:localkernel}
\Ker \ \Phi|_{\oFH_T} \cong \bigoplus^{\text{red or dotted horizontal}}_{\text{lines from box }i} \CO \cdot e_i \bigoplus^{\text{red or dotted vertical}}_{\text{lines from box }i} \CO \cdot e_i'
\end{equation}
In other words, we claim that if one specifies rescaled basis vectors $c_i e_i$ and $d_i e_i'$ corresponding to the red and dotted edges, then there exist unique rescaled basis vectors $\gamma_i e_i$ and $\delta_i e_i'$ corresponding to the black thick edges, and a function $f$, such that:
$$
(X-x_{n+1})\left(\sum_{i\text{ dotted}} c_i e_i + \sum_{i\text{ thick}} \gamma_i e_i \right) + (Y-y_{n+1})\left(\sum_{i\text{ dotted}} d_i e_i + \sum_{i\text{ thick}} \delta_i e_i \right) + f e_1 = 0
$$
Any box $k$ has a unique black thick line going to the left or down. Assume without loss of generality that the black thick line from $k$ leads one step left to the box $i$. Then \eqref{eqn:chartx} implies that equating the coefficient of $e_k$ in the left hand side to 0 yields the equation:
$$
\gamma_i \in \sum_j (c_j \text{ or } d_j) \cdot \text{coefficients} + \sum_j (\gamma_j \text{ or }\delta_j) \cdot \ofm_T
$$
This system of equations can be solved in the localization $\oCA_T$, since its determinant is in $1+\ofm_T$. Therefore, we conclude that in the local chart $\oFH_T$, we have:
\begin{equation}
\label{eqn:zzz}
\CE_n|_{\oFH_T} \stackrel{\qis}\cong \left[\bigoplus_{i=1}^n \CO \cdot e_i \stackrel{\Psi}\longrightarrow \bigoplus^{\text{red or dotted horizontal}}_{\text{lines from box }i} \CO \cdot e_i \bigoplus^{\text{red or dotted vertical}}_{\text{lines from box }i} \CO \cdot e_i' \right]
\end{equation}
The Proposition will be proved once we show that projecting the two terms in the above complex to a certain subset of factors induces a quasi-isomorphism. Specifically, in the domain of $\Psi$ we consider the subspace spanned by basis vectors $e_\bsq$ corresponding to outer corners $\bsq$, and in the codomain of $\Psi$ we project onto one basis vector $f_\sq$ corresponding to each inner corner. The rule is that $f_\sq = e_i$ or $e_{i'}'$, depending on whether the number $i$ to the left of $\sq$ is bigger or smaller than the number $i'$ below $\sq$. In other words, the only $e_i$ or $e_i'$ we will consider in the codomain of $\Psi$ are the ones corresponding to thick red lines:
\begin{equation}
\label{eqn:zzzz}
\CE_n|_{\oFH_T} \stackrel{\qis}\cong \left[\bigoplus^{\bsq \text{ outer}}_{\text{corner of }T} \CO \cdot e_\bsq \stackrel{\psi}\longrightarrow \bigoplus^{\sq \text{ inner}}_{\text{corner of }T} \CO \cdot (e_{\sq \leftarrow} \text{ or } e'_{\sq \downarrow}) \right]
\end{equation}
In plain English, the claim is that for any $e_i$ where $i$ is not an outer corner of $T$, quotienting the codomain of \eqref{eqn:zzz} by the vector:
\begin{equation}
\label{eqn:quotientvector}
\sum^{\text{red or dotted horizontal}}_{\text{lines between boxes }\bar{jk}} (x_i^j - \delta_j^i x_{n+1}) e_j + \sum^{\text{red or dotted vertical}}_{\text{lines between boxes }\bar{jk}} (y_i^j - \delta_j^i y_{n+1}) e_j'
\end{equation}
will allow us to solve for one of the $e_j,e_{j}'$. The only basis vectors which remain unsolved for should be the ones that appear in the codomain of \eqref{eqn:zzzz}. For example, suppose we are trying to solve for $e_j$, where $j$ corresponds to a dotted horizontal edge $\bar{jk}$. Then with the notation in the following picture:

\begin{picture}(100,90)(-160,-5)

\put(10,10){$i$}
\put(10,65){$j$}
\put(65,10){$j'$}
\put(65,65){$k$}
\put(0,0){\line(1,0){80}}
\put(0,40){\line(1,0){80}}
\put(0,80){\line(1,0){80}}
\put(0,0){\line(0,1){80}}
\put(40,0){\line(0,1){80}}
\put(80,0){\line(0,1){80}}

%\multiput(20,20)(8,0){5}{\line(1,0){4}}
\multiput(20,60)(8,0){5}{\line(1,0){4}}
%\multiput(20,20)(0,8){5}{\line(0,1){4}}\\
%\multiput(60,20)(0,8){5}{\line(0,1){4}}

\linethickness{1mm}
%\put(20,20){\line(1,0){40}}
%\put(20,60){\line(1,0){40}}
%\put(20,20){\line(0,1){40}}
\put(60,20){\line(0,1){40}}

%\put(150,40){\text{ the vector } \eqref{eqn:quotientvector} \text{ equals } $e_j+\ofm_T$}

\end{picture}

\noindent let us observe that \eqref{eqn:chartx}--\eqref{eqn:charty} imply that $y_i^j  = x_i^{j'} = y_{j'}^k \in 1 + \ofm_T$. Then the vector \eqref{eqn:quotientvector} lies in $e_j + \ofm_T$, and $e_j$ can therefore be solved for in the localization $\oCA_T$.

\end{proof}

\subsection{Examples }
\label{sec:exprojectors}

In this subsection, we use the local geometry of the flag Hilbert scheme to describe the homology of 
categorified projectors on two and three strands.

\begin{example}
For the $S^2$ projector, we have
$$
Y=\left(\begin{matrix}
0 & 0\\
1 & 0\\
\end{matrix}\right).
$$
The commutation relation implies that \(X\) is of the form
$$
X=\left(\begin{matrix}
u_1 & 0\\
u_2 & u_1\\
\end{matrix}\right).
$$
One has $\deg(u_1)=q$, $\deg(u_2)=q/t$, so the Poincar\'e series equals 
$$
P(\oCA_T)=\frac{1}{(1-q)(1-q/t)}.
$$
\end{example}

%The $\fsl(N)$ differentials have the form $d_N(\xi_1)=u_1^N,\ d_N(\xi_2)=Nu_1^{N-1}u_2$.
%The Poincar\'e series for the homology of $d_N$ (with rational coefficients) was computed in
%\cite[Theorem 12]{GL}:
%$$
%\CP(S^2,d_N)=\frac{1-q^{2N}-q^{2N+2}t^2 + q^{2N+4}t^2 + q^{2N+4}t^3 - q^{4N+2}t^3} {(1-q^2)(1-q^4t^2)}.
%$$
%{\bf REGRADE!!!!}
%This agrees with the actual $\fsl(N)$ Khovanov-Rozansky homology of the unknot, computed by Cautis in
%\cite[Corollary 10.2]{Cautis}.

%By Corollary \ref{d11 sym}, we have
%$$
%\CP(S^2,d_{1|1})=(1+q??)/(1-q^4t^2).
%$$

\begin{example}

For the $\Lambda^2$ projector, we have
$$
X=\left(\begin{matrix}
x_1 & 0\\
1 & x_2\\
\end{matrix}\right),
Y=\left(\begin{matrix}
0 & 0\\
y_{21} & 0\\
\end{matrix}\right),
$$
and the commutation relation $(x_1-x_2)y_{21}=0$.
One has $\deg(x_1)=\deg(x_2)=q$, $\deg(y_{21})=t/q$, so the Poincar\'e series equals 
$$
P(\oCA_T)=\frac{1-t}{(1-q)^2(1-t/q)}=\frac{1}{(1-q)^2}+\frac{t/q}{(1-q)(1-t/q)}.
$$
%Let us compute the homology of $d_N$. We have 
%$$
%d_N(\xi_1)=x_1^N,\ d_N(\xi_2)=\sum_{i=0}^{N-1}x_1^{i}x_2^{N-1-i}.
%$$
%The zero Koszul homology is spanned by monomials in $x_1,x_2$ not divisible by $x_1^N$ or $x_2^{N-1}$,
%and $y_{21}$ multiplied by monomials in $x_1$ not divisible by $x_1^{N-1}$. To compute the first homology,
%note that $d_N(N\xi_1-x_1\xi_2)$ is divisible by $(x_1-x_2)$, so $d_N((N\xi_1-x_1\xi_2)y_{21})$ vanishes in the algebra
%and hence generates $H^1(d_N)$. 
%Therefore we get:
%$$
%\CP(\Lambda^2,d_N)=\frac{(1-q^{N})(1-q^{N-1})}{(1-q)^2}+\frac{t(1-q^{N-1})}{q(1-q)(1-t/q)}+\frac{q^{N}t(1-q^{N-1})}{q(1-q)(1-t/q)}.
%$$

\end{example}

%\subsection{Projector to $S^3$}

\begin{example}
\label{ex:Sym n}
For the $S^3$ projector, we have
$$
X=\left(\begin{matrix}
x_1 & 0 & 0\\
x_2 & x_1 & 0\\
x_3 & x_2 & x_1\\
\end{matrix}\right),
Y=\left(\begin{matrix}
0 & 0 & 0\\
1 & 0 & 0\\
0 & 1 & 0\\
\end{matrix}\right),
$$
One has $\deg(x_1)=q$, $\deg(x_2)=q/t$, $\deg x_3=q/t^2$, so the Poincar\'e series equals 
$$
P(\oCA_T)=\frac{1}{(1-q)(1-q/t)(1-q/t^2)}.
$$
More generally, for the \(S^n\) projector, we have 
$$
P(\oCA_T)=\prod_{i=1}^n{(1-qt^{1-i})^{-1}}.
$$

\end{example}

%\subsection{Projector to $\Lambda^3$}

\begin{example}
\label{ex:Lambda n}
For the $\Lambda^3$ projector, we have
$$
X=\left(\begin{matrix}
x_1 & 0 & 0\\
1   & x_2 & 0\\
0 & 1 & x_3\\
\end{matrix}\right),
Y=\left(\begin{matrix}
0 & 0 & 0\\
y_{21} & 0 & 0\\
y_{31} & y_{32} & 0\\
\end{matrix}\right),
$$
and the commutation relations $(x_1-x_2)y_{21}=(x_2-x_3)y_{32}=0$ and
$$
y_{21}-y_{32}=(x_{1}-x_{3})y_{31}.
$$
Note that one can eliminate $y_{32}$ using the last equation.
One has $\deg(x_1)=\deg(x_2)=\deg(x_3)=q$, $\deg(y_{21})=\deg(y_{32})=t/q$, $\deg(y_{31})=t/q^2$,  so the Poincar\'e series equals 
$$
P(\oCA_T)=\frac{(1-t)^2}{(1-q)^3(1-q/t)(1-q/t^2)}.
$$

More generally, for the \(\Lambda^n\) projector, we have 
 $$\oCA_T= \frac {\CC[x_1,\ldots,x_n, y_{i,j}]_{i>j}}{y_{i,j}(x_i-x_j) - (y_{i-1,j} - y_{i,j+1})} $$
 and 
 $$
P(\oCA_T)=\frac{(1-t)^{n-1}}{(1-q)^n}\prod_{i=1}^{n-1}(1-qt^{-i})^{-1}.
$$
(which can also be seen directly from Proposition \ref{prop:poincare}.)
\end{example}

\begin{example}
For the hook-shaped projector with $(z_1,z_2,z_3)=(1,t,q)$, we have
$$
X=\left(\begin{matrix}
x_1 & 0 & 0\\
x_{21}   & x_1 & 0\\
1 & x_{32} & x_3\\
\end{matrix}\right),
Y=\left(\begin{matrix}
0 & 0 & 0\\
1 & 0 & 0\\
0 & y_{32} & 0\\
\end{matrix}\right),
$$
with commutation relations
$$
x_{32}=x_{21}y_{32},\ (x_1-x_3)y_{32}=0,
$$
In this case $\deg(x_1)=\deg(x_3)=q$, $\deg(x_{21})=q/t$, $\deg(x_{32})=t$, $\deg(y_{32})=t^2/q$, so the Poincar\'e series equals 
$$
P(\oCA_T)=\frac{(1-t^2)}{(1-q)^2(1-q/t)(1-t^2/q)}.
$$

For the other hook-shaped projector we have $(z_1,z_2,z_3)=(1,q,t)$, so
$$
X=\left(\begin{matrix}
x_1 & 0 & 0\\
1   & x_2 & 0\\
0 & x_{32} & x_3\\
\end{matrix}\right),
Y=\left(\begin{matrix}
0 & 0 & 0\\
y_{21} & 0 & 0\\
1 & y_{32} & 0\\
\end{matrix}\right),
$$
with commutation relations $(x_1-x_2)y_{21}=(x_2-x_3)y_{32}=0$ and 
$$
x_1-x_3+y_{32}=x_{32}y_{21}.
$$
In this case $\deg(x_1)=\deg(x_2)=\deg(x_3)=q$, $\deg(x_{32})=q^2/t$, $\deg(y_{21})=t/q$, $\deg(y_{32})=q$, so the Poincar\'e series equals 
$$
P(\oCA_T)=\frac{(1-t)(1-q^2)}{(1-q)^3(1-t/q)(1-q^2/t)}.
$$
\end{example}

\subsection{Poincare series}

In general,  Proposition~\ref{prop:chart} can be used to show

\begin{proposition}
\label{prop:poincare}

For any standard Young tableau of size $n$, the bigraded Poincar\'e series of graded algebras $\oCA_T(*)$ are given by the following formulas:
\begin{equation}
\label{eqn:poincare1}
P(\oCA_T(\CC^2)) = (1-q)^{-n} (1-t)^{-n} \prod_{i=1}^{n} \frac{1}{1-z_i^{-1}} \prod_{1\leq i<j \leq n} \zeta \left(\frac {z_i}{z_j}\right)
\end{equation}
\begin{equation}
\label{eqn:poincare2}
P(\oCA_T(\CC)) = (1-q)^{-n} \prod_{i=1}^{n}\frac{1}{1-z_i^{-1}} \prod_{1\leq i<j \leq n} \zeta \left(\frac {z_i}{z_j}\right)
\end{equation}
\begin{equation}
\label{eqn:poincare3}
P(\oCA_T(\point)) = \prod_{i=1}^{n}\frac{1}{1-z_i^{-1}}  \prod_{i=2}^{n}\frac{1}{1-q t  z_i/z_{i+1}} \prod_{1\leq i<j \leq n} \zeta \left(\frac {z_i}{z_j}\right)
\end{equation}
where
$$
\zeta(x) = \frac {(1 - x)(1 - q t x)}{(1 - q x)(1 - t x)}
$$
and $z_i$ denotes the weight of the \(i\)th box in the standard Young tableau $T$.
\end{proposition}

\begin{proof} We will prove \eqref{eqn:poincare1}, and leave the other two formulas as exercises for the interested reader. In order to prove the formula by induction, one needs to compute the following quotient for any standard Young tableau $T$ of size $n$ and an inner corner $\square \in T$:
$$
\frac {P(\oCA_{T \cup \square} (\CC^2))}{P(\oCA_T(\CC^2))} = \frac 1{(1-q)(1-t)(1-z_{\square}^{-1})} \prod_{i=1}^{n} \zeta \left(\frac {z_i}{z_{\square}}\right)
$$
Since $T$ is a Young tableau, it is easy to show that the product of $\zeta$'s can be simplified to:
$$
\frac {P(\oCA_{T \cup \square} (\CC^2))}{P(\oCA_T(\CC^2))} = \frac 1{(1-q)(1-t)} \cdot \frac {\prod_{\bsq \text{inner corner of }T}^{\bsq \neq \sq} (1-z_{\bsq}/z_{\sq})}{\prod_{\bsq \text{outer corner of }T} (1-z_{\bsq}/z_{\sq})}
$$
The above formula follows from \eqref{eqn:genrel} and the relation \eqref{eqn:expectation} for the weights. \end{proof}

If we pass to the decategorified setting by substituting \(t=q^{-1}\), we see that the Poincar\'e series depends only on the Young diagram of \(T\): 
\begin{corollary}
\label{cor:decat ps}
$ \displaystyle
P(\oCA_T(\CC))|_{t=q^{-1}}=\frac{1}{\prod_{\square\in \lambda}(1-q^{h(\square)})}.
$
\end{corollary}

\begin{proof}
If we let  $\zeta_{q,q^{-1}}(x) = \zeta(x)|_{t=q^{-1}}$, then clearly
$
\zeta_{q,q^{-1}}(x)=\zeta_{q,q^{-1}}(x^{-1}).
$
It follows that the  function $\prod_{i<j}\zeta_{q,q^{-1}}(z_i/z_j)$
is actually symmetric in $z_i$, hence depends only of the shape of the Young diagram.
Let us choose the permutation of $z_i$ such that
$$
(z_1,\ldots,z_n)=(1,q,\ldots,q^{\lambda_1-1},q^{-1},\ldots,q^{\lambda_1-2},\ldots),
$$
Given $z_i$ on the vertical boundary of $\lambda$ and $z_j$ on the horizonal boundary such that $i<j$,
one can consider the box $\square$ in the same row as $z_i$ and in the same column as $z_j$. 
We get $(1-qz_i/z_j)=(1-q^{h(\square)})$, where $h(\square)$ denotes the hook-length of $\square$.
One can check after all telescopic cancellations 
$$
\prod_{i<j}\zeta_{q,q^{-1}}(z_i/z_j)=\frac{(1-q)^n\prod_{j>1}(1-z_i^{-1})}{\prod_{\square\in \lambda}(1-q^{h(\square)})},
$$
so
$$
P(\oCA_T)=\prod_{i=1}^{n}\frac{1}{(1-q)(1-z_i^{-1})}\prod_{i<j}\zeta_{q,q^{-1}}(z_i/z_j)=\frac{1}{\prod_{\square\in \lambda}(1-q^{h(\square)})}.
$$
\end{proof}

To compute the full endomorphism ring %\(HHH\) 
of the projector \(P_T\), we should tensor with $\wedge^{\bullet}\CT_n^{\vee}$.  
When we restrict to the affine chart $\FH_T \subset \FH_n$  the vector space $\CC^n$ is endowed with a preferred basis $e_1,...,e_n$, which more abstractly means that the tautological bundle is trivialized:
$$
\CT_n|_{\FH_T} \cong \CO \cdot e_1 \oplus...\oplus \CO \cdot e_n
$$
The basis vectors are indexed by boxes $\sq$ in the Young diagram of $T$, and the torus $\torus$ acts on the basis vector $e^\sq$ by the character $z_\sq = q^at^b$ for any box $\sq = (a,b)$. We conclude that:
$$
\wedge^{\bullet}\CT_n^{\vee}|_{\FH_T} \cong \wedge(\xi_1,\ldots,\xi_n)
$$
where the equivariant weights of the symbols $\xi_\sq$ are given by $z_\sq^{-1} = q^{-a}t^{-b}$. In particular, Conjecture~\ref{conj:homfly} implies that $\End(P_T)$ should be  the tensor product of the homology on the ``bottom row'' with an exterior algebra. 

The theorems stated in the introduction can be easily deduced from the results above.

\begin{proof} (Of Theorem \ref{thm:symm antisymm} ) The observations in Examples \ref{ex:Sym n} and \ref{ex:Lambda n}, together with the remark above, show that  the expressions on the right-hand side of equations 
\eqref{eqn:antisym} and \eqref{eqn:sym} agree with 
\( \CA_T(\CC) \otimes \left( \wedge^{\bullet}\CT_n^{\vee}|_{\oFH_T(\CC)} \right)\). 
On the other hand, these expressions agree with the known homology of the symmetric projector (computed by Hogancamp in \cite{Hog}) and the antisymmetric projector (computed by Abel and Hogancamp in \cite{AbHog}.) 
\end{proof}

\begin{proof} (of Theorem \ref{thm:hecke projector}). 
From Corollary~\ref{cor:decat ps}, we see that 
$$P\left(\oCA_T(\CC) \otimes \left( \wedge^{\bullet}\CT_n^{\vee}|_{\oFH_T(\CC)} \right) \right)= 
\prod_{\square\in \lambda}\frac{1-aq^{b(\square)-a(\square)}}{(1-q^{h(\square)})}.$$
This right-hand side is a well-known formula for the \(\lambda\)-colored HOMFLY-PT polynomial of the unknot, which is by definition the Markov trace of the Jones-Wenzl projector \(p_\lambda \in H_n\). 

\end{proof}

\section{Differentials and $\fgl_N$ homology}
\label{sec:differentials}

\subsection{Spectral sequence for $\fgl_N$ homology}
\label{sub:spectral diff}

By \cite{RasDiff}, for each $N$ there exists a spectral sequence starting at the HOMFLY-PT homology and converging to $\fsl_N$ homology of a given knot.  More precisely, for a given braid $\sigma$ one can construct a complex of Soergel bimodules as described in Subsection \ref{sub:khr}. The Hochschild homology of this complex coincides with the HOMFLY-PT homology of the closure of $\sigma$. Given a polynomial \(p \in \CC[x]\),  we can construct an additional differential \(d_-\) which acts on Soergel bimodules, as we now describe. 

Recall that the simple Soergel bimodule can be written as $B_i=R \otimes_{R^{i,i+1}} R$.
Denote $u_j=x_j\otimes 1, v_j=1\otimes x_j$ for all $j$, and
$$
U_{i,i+1}:=\frac{\CC[u_1,\ldots,u_n,v_1,\ldots,v_n]}{(u_i+u_{i+1}-v_i-v_{i+1},u_j-v_j,j\notin \{i,i+1\})}, 
$$
then
$$
B_i\cong \left[U_{i,i+1}\xrightarrow{(v_i-u_i)(v_i-u_{i+1})}U_{i,i+1}\right].
$$
Given a polynomial $p\in \CC[x]$, consider the difference
$$
W_{i,i+1}:=p(u_i)+p(u_{i+1})-p(v_i)-p(v_{i+1})=p(u_i)+p(u_{i+1})-p(v_i)-p(u_i+u_{i+1}-v_i)\in U_{i,i+1}.
$$
Remark that $W_{i,i+1}$ is divisible by $(v_i-u_i)(v_i-u_{i+1})$: indeed, $W_{i,i+1}$ vanishes if $v_i=u_i$ or $v_i=u_{i+1}$.
Let $p_{i,i+1}=W_{i,i+1}/(v_i-u_i)(v_i-u_{i+1})$.  We use $p_{i,i+1}$ to define an additional differential (denoted by $d_{-}$ in \cite{RasDiff}) which acts {\em backwards}:
\begin{equation}
\label{bs with potential}
B_i^{(p)}:=\left[U_{i,i+1}\xtofrom[d_{-}:=p_{i,i+1}]{(v_i-u_i)(v_i-u_{i+1})}U_{i,i+1}\right].
\end{equation}
Note that the total complex $(B_i^{(p)},d_++d_-)$ is not a chain complex but a matrix factorization with potential $W_{i,i+1}$.
% which is usually referred to as a {\em matrix factorization with potential $p(x)$}. Alternatively, one could quotient $S_{i,i+1}$ by the ideal generated by $p(u_i),p(v_i)$ and get a genuine chain complex. 

It is proved in \cite{RasDiff} that this additional differential $d_{-}$ can be naturally extended to Bott-Samuelson bimodules (tensor products of $B_i$), and to Rouquier complexes. One can also prove \cite{Becker} that $d_{-}$ can be correctly defined on general Soergel bimodules as well.
For $p'(x)=x^N$, this differential is usually denoted by $d_N$, and the homology of the total differential is isomorphic to $\fgl_N$ Khovanov-Rozansky homology \cite{KhR1}. The desired spectral sequence is then induced by $d_N$ on $\HHH(\sigma)$. 

%\begin{remark}

%In fact, the differential $d_N$ acts on resolutions of Soergel bimodules by free bimodules. %\textcolor{red}{Eugene wants more details here}.

%\end{remark}
In the present section, we wish to present a more geometric viewpoint of this construction. Given $N$, we define the so-called $\fsl_N$ dg category $(\SBim_n,d_N)$, where the objects are Soergel bimodules equipped with the ``internal differential'' $d_N$. This is a subcategory of the category of matrix factorizations with potential $x^N$. There is a monoidal functor: 
$$
K^b(\SBim_n) \to \left( K^b(\SBim_n), d_N \right)
$$
which is given by endowing complexes of Soergel bimodules with the differential $d_N$.

%\begin{remark}
%The differential $d_N$ is the special case $p(x) = x^N$ of the construction of \cite{RasDiff}. For any polynomial $p(x) \in \CC[x]$, \loccit constructs a homology theory that also gives rise to a knot invariant, in which the homology of the unknot is isomorphic to $\CC[x]/p(x).$
%\end{remark}

\subsection{Sections and schemes}
\label{sub:sections schemes}

On the geometric side, we have a remarkable family of dg schemes closely related to $\FH^\dg_n = \FH_n^\dg(\CC)$. Namely, let $s$ be an arbitrary section of the tautological bundle $\CT_n$. It defines a contraction map: 
\begin{equation}
\label{eqn:contraction}
d_s:\wedge^{\bullet}\CT_n^{\vee}\to \wedge^{\bullet-1}\CT_n^{\vee}
\end{equation}
Recall the construction \eqref{eqn:iota wedge}:
$$
\widetilde{\iota}_*(\sigma) = \iota_*(\sigma) \otimes \wedge^\bullet \CT_n^\vee
$$
which is naturally a sheaf of dg modules on $\tot_{\FH_n^\dg} \CT_n[1]$. If we endow the exterior power with the differential \eqref{eqn:contraction}, we obtain:
$$
\left( \widetilde{\iota}_*(\sigma), d_s \right)
$$
which is naturally a sheaf of dg modules on the dg scheme:
$$
\tot_{\FH_n^\dg} (\CT_n[1],s) := \text{the sheaf of dg algebras } \left( \wedge^\bullet \CT_n^\vee, d_s \right) \text{ on }\FH_n^\dg.
$$
To construct sections $s$ of the tautological bundle $\CT_n$, recall that its fibers are given by:
$$
\CT_n|_{I_n\subset \ldots\subset \CC[x,y]} = \CC[x,y]/I_n.
$$
Therefore every polynomial $f\in \CC[x,y]$ defines a section $s_f \in \Gamma(\FH_n, \CT_n)$ for all $n$, and these sections are all compatible with each other: 
$$
\includegraphics{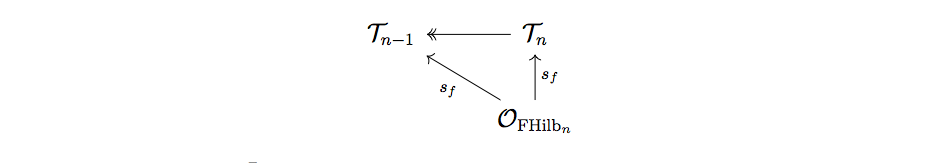}
$$
The morphism $\FH_{n} \stackrel{\pi}\longrightarrow \FH_{n-1} \times \CC$ therefore induces a map:
$$
\tot_{\FH_n^\dg} (\CT_n[1],s_f) \stackrel{\pi_f}\longrightarrow \tot_{\FH_{n-1}^\dg \times \CC} (\CT_{n-1}[1],s_f)
$$
and so one has a commutative diagram of maps of dg schemes:
$$
\includegraphics{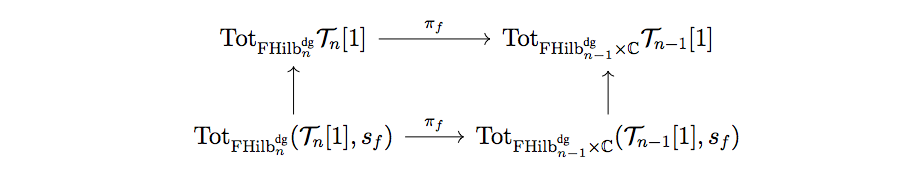}
$$
where the vertical maps are simply induced by the map of dg algebras $\wedge^\bullet \CT_n^\vee \rightarrow \left( \wedge^\bullet \CT_n^\vee, d_s \right)$. Note that the dg scheme $\tot_{\FH_n^\dg} (\CT_n[1],s_f)$ is $\torus$ equivariant if and only if $f$ is an equivariant section of $\CT_n$. It is not hard to see that the only such equivariant sections are $f(x,y) = x^Ny^M$ for some $(N,M) \in \NN_0 \times \NN_0$. We denote the corresponding section by $s_{N|M}$.

%\textcolor{red}{Here I'm not convinced that such a map of dg algebras exists. Indeed, it's perfectly conceivable for an algebra $B$ to have a structure of module over an algebra $A$, without this structure coming from an algebra map $A \rightarrow B$. This would be unfortunate for us, since it would mean we could not define the operation:
%$$
%- \bigotimes_{\wedge^\bullet \CT_n^\vee} \left( \wedge^\bullet \CT_n^\vee, d_s \right)
%$$
%and then I'm not sure how to define the left downward map in \eqref{eqn:conj differentials}}

\begin{remark}
In \cite[Section 7]{GORS}, the differentials were parametrized by copies of the defining representation of $S_n$ in the rational Cherednik algebra, which can be considered as a noncommutative deformation of $\CC[x_1,\ldots,x_n,y_1,\ldots,y_n]$. One can check that such a copy naturally corresponds to a section of $\CT_n$, in particular, $f\in \CC[x,y]$ corresponds to
the subspace $\Span(f(x_i,y_i))_{1\leq i \leq n}$.
\end{remark}

\subsection{The commutative tower}
\label{sub:conj differentials}

We conjecture that the differential $d_N$ in the Soergel category is closely related to the section $f=x^N$ of the tautological bundle on the flag Hilbert scheme. More precisely, we propose the following:

\begin{conjecture}
\label{conj:differentials}
There is a map $\iota_N : (\SBim_n,d_N) \to (Z_n(\CC),s_N)$ in the sense of Definition \ref{def:morphism}. The corresponding functors fit into the commutative diagram:
\begin{equation}
\label{eqn:conj differentials}
\includegraphics{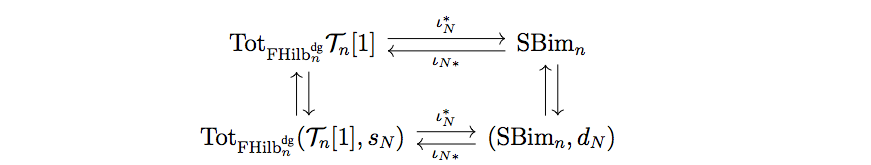}
\end{equation}
Furthermore, there is a tower of commuting squares connected with $\pi_N,\Tr,I$ akin to \eqref{eq:tower}.
\end{conjecture}

\begin{remark}
We expect that the general differential on $\SBim_n$ corresponding to the polynomial $p(x)$ in the right hand side, corresponds to replacing $s_N$ by $s_{p(x)}$ in the left hand side.
\end{remark}

The conjecture is true for $n=1$. Indeed, $\FH_1=\FH_1^\dg = \CC$, so: 
$$
\tot_{\FH_1^\dg} (\CT_n[1],s_N) = S_{\CC[x]}^\bullet \left(\CC[x] \stackrel{x^N}\longrightarrow \CC[x] \right) \cong \spec \ \CC[x]/(x^N).
$$
The Soergel category $\SBim_1$ has a unique $\CC[x]$ bimodule, namely $\1 = \CC[x,y]/(x-y)$, and the corresponding object in the dg category $(\SBim_1,d_N)$ is given by:
$$
\1 = \left[ \CC[x,y]\xtofrom[x-y]{(W(x)-W(y))/(x-y)}\CC[x,y] \right] 
$$
where $W(x)=\frac{x^{N+1}}{N+1}$. One can eliminate $y$ and rewrite the above
$$
\1 = \left[ \CC[x]\xrightarrow{W'(x)=x^N}\CC[x] \right]
$$
from where it is clear that the categories $\tot_{\FH_1^\dg} (\CT_n[1],s_N)$ and $(\SBim_1,d_N)$ are equivalent.
 
\subsection{Differentials in affine charts}
\label{sub:affine diff}

Recall the affine charts $\FH_T \subset \FH_n$ defined in Subsection \ref{sub:defcharts}. In each of these, the vector space $\CC^n$ is endowed with a preferred basis $e_1,...,e_n$, which more abstractly means that the tautological bundle is trivialized:
$$
\CT_n|_{\FH_T} \cong \CO \cdot e_1 \oplus...\oplus \CO \cdot e_n
$$
The basis vectors are indexed by boxes $\sq$ in the Young diagram of $T$, and the torus $\torus$ acts on the basis vector $e^\sq$ by the character $z_\sq = q^at^b$ for any box $\sq = (a,b)$. We conclude that:
$$
\wedge^{\bullet}\CT_n^{\vee}|_{\FH_T} \cong \wedge(\xi_1,\ldots,\xi_n)
$$
where the equivariant weights of the symbols $\xi_\sq$ are given by $z_\sq^{-1} = q^{-a}t^{-b}$. Recall from Subsection \ref{sub:sections schemes} that to any polynomial $f\in \CC[x,y]$, we may associate a section of the tautological bundle given by:
\begin{equation}
\label{eqn:diff chart}
s_f|_{(X,Y,v)} = f(X,Y)v \in \CT_n|_{(X,Y,v)}
\end{equation}
We may dualize the above section to obtain $s_f:\CT_n^\vee \rightarrow \CO$, and in local coordinates this takes the form: 
\begin{equation}
\label{eqn:dual diff chart}
s_f(\xi_i)=[f(X,Y)v]_{i} = f(X,Y)_{i1}
\end{equation}
The local rings of the dg scheme $\tot_{\FH_n^\edg} (\CT_n[1],s_f)$ is then given by the Koszul complex associated with the first column of the matrix $f(X,Y)$.

\begin{lemma}
Suppose that $f=x^Ny^M$ and the diagram of $T$ contains the box with coordinates $(N,M)$.
Then the dg algebra $\wedge_{\FH_n^\edg} (\CT_n^\vee,s_f)$ is contractible in the local chart $\oFH_T$. 
\end{lemma}

\begin{proof}
Suppose that $\sq = (N,M)$in $T$. Using \eqref{eqn:chartx}--\eqref{eqn:charty}, one can prove that $(X^NY^M)(v)\in e_\sq + \fm_T$, where $\fm_T$ is the maximal ideal in the local ring $\oCA_T = \CC[\oFH_T]$. Therefore, $s_f(\xi_\sq)=1$ is invertible in \eqref{eqn:dual diff chart}, and this implies that the Koszul complex of $s_f$ is contractible.
\end{proof}

\begin{corollary}
\label{cor: kill projectors}
Suppose that  the diagram of $T$ has more than $N$ columns.
Then the homology of the categorified projector $\CP_T$ with respect to $d_N$ vanishes.
\end{corollary}

\begin{remark}
In \cite[Theorem 4]{Becker} it is proved that $(B_w,d_N)\cong 0$, if the Robinson-Shensted tableau of $w$ has more than $N$ columns.
One can prove that Soergel bimodules $B_w$ with this property generate a tensor subcategory of $\SBim_n$, and all categorified projectors $\CP_T$ belong to this subcategory, provided that $T$ has more than $N$ columns. Therefore  $(\CP_T,d_N)\cong 0$ in agreement with Corollary \ref{cor: kill projectors}. 
\end{remark}

For $T=(1,\ldots,1)$, the differential corresponding to $x^N$ can be written very explicitly. 
 
\begin{proposition}
In the chart $\oFH_{(1,\ldots,1)}$ the differential $d_N$ is given by the equation
\begin{equation}
\label{eqn:dN}
d_N(\xi_1+z\xi_2+\ldots+z^{n-1}\xi_{n-1})=(u_1+zu_2+\ldots+z^{n-1}u_{n})^{N}\ \mod z^n, 
\end{equation}
where $u_1,\ldots,u_n$ are local coordinates and $z$ is a formal parameter. 
\end{proposition}

\begin{proof} Indeed, in the chart $\oFH_{(1,\ldots,1)}$ one has $X=u_1+Bu_2+\ldots+B^{n-1}u_n$, where $B$ is the  $n\times n$ Jordan block. Clearly, $B^n=0$ and the first column of $X^N$ contains first $n$ coefficients of the polynomial $(u_1+zu_2+\ldots+z^{n-1}u_{n})^{N}$. 
\end{proof}

As a corollary, we get the following result. 

\begin{proposition}
Assuming Conjecture \ref{conj:differentials},
the $\fsl_N$ homology of the $n$-th symmetric categorified Jones-Wenzl projector is isomorphic to
the Koszul homology of the differential \eqref{eqn:dN}.
\end{proposition}

This description of $d_N$ indeed agrees with the ones in \cite{GL,GOR,GORS}, and the homology is quite involved. Indeed, its Poincar\'e series for $n\to\infty$ deforms the character of the $(2,2N+1)$ minimal model for the Virasoro algebra.  Extensive computer experiments \cite{GL,GOR} support this conjecture for $N=2$ and $N=3$. See also \cite{Hog2} for recent developments for $N=2$.

The homology of all projectors on two and three strands with respect to $d_N$ were described in \cite{GL}. One can check that they agree with the general framework of this paper.

%On the other hand, the homology of $d_{1|1}$ is easy to compute:

%\begin{corollary}
%\label{d11 sym}
%The homology of the differential $d_{1|1}$ is isomorphic to $\mathbb{C}[\xi_1,u_n]$. 
%\end{corollary}

%\section{The differentials}
%\label{sec:differentials}

%\section{Examples of projectors}
%\label{sec:examples}

%\subsection{Action of symmetric group}
%
%There is a natural action of the symmetric group on complexes of Soergel bimodules which permutes both $x$- and $y$-variables. 
%This action preserves the bimodule $B_{w}$, but acts nontrivially on morphisms. As one applies the Hochschild homology to the
%antisymmetric projector, this action preserves $u_i$ and permutes $x_i$:
%$$
%\sigma(x_1,\ldots,x_n,u_2,\ldots,u_n)=\sigma(x_{\sigma(1)},\ldots,x_{\sigma(n)},u_2,\ldots,u_n).
%$$
%Let us describe the action of the symmetric group in terms of the matrices $X$ and $Y$.

%\begin{theorem}
%Consider the matrix $T_i=I+(x_i-x_{i+1})E_{i,i+1}$. Then 
%$$
%T_i^{-1}XT_i=s_i(X),\ T_i^{-1}YT_i=s_i(Y),
%$$
%where $s_i=(i,i+1)$ is the simple reflection in $S_n$ and $s_i(Y(x,u)):=Y(s_i(x,u))$. 
%In other words, $T_i$ generate the action of $S_n$ on $\CA(T)$.
%\end{theorem}

%\begin{proof}
%Both statements can be checked by a direct computation. Note that the matrix $T_i^{-1}YT_i$ may, a priori,
%contain some entries at positions $(i,i),(i,i+1),(i+1,i+1)$, but all these vanish since the commutation relation
%$[X,Y]=0$ implies $(x_i-x_{i+1})y_{i+1,i}=0$.
%\end{proof}

%\section{Examples of local algebras}

\section{Appendix}
\label{sec:appendix}

\subsection{Dg algebras}
\label{sub:dg algebras}

A vector space $V$ will be called dg (short for ``differential graded") if it comes endowed with a grading:
$$
V = \bigoplus_{n\in \ZZ} V^i
$$
and a differential $d : V^\bullet \rightarrow V^{\bullet+1}$ such that $d^2=0$. A vector $v \in V$ is called homogeneous if $v\in V^i$ for some integer $i$. If this is the case, then we will write $\deg v = i$. 

\begin{definition}
\label{def:dg algebra}

A \textbf{dg algebra} $A^\bullet$ is a dg vector space concentrated in non-positive degrees ($A^n=0$ for $n>0$), which is endowed with a multiplication that preserves the grading:
$$
A^i \cdot A^j \subset A^{i+j} \qquad \forall \ i,j \in \NN_0
$$
and the differential via the graded Leibniz rule:
\begin{equation}
\label{eqn:leibniz}
d(a \cdot a') = (d a) \cdot a' + (-1)^{\deg a} a \cdot (da') \qquad \forall \ a,a'\in A
\end{equation}
We impose the usual axioms on the dg algebra $A^\bullet$, such as associativity and unit $1 \in A^0$. 

\end{definition}

All the dg algebras in this paper will be commutative, in the sense that:
\begin{equation}
\label{eqn:dg commutative}
a \cdot a' = (-1)^{(\deg a)(\deg a')} a' \cdot a \qquad \forall \ a,a'\in A^\bullet
\end{equation}
We will write $H^0(A)$ for the 0--th cohomology of $A^\bullet$, which is a usual commutative algebra. All the dg algebras studied in this paper will be finitely generated over $H^0(A)$. 

\begin{definition}
\label{def:dg module}

A \textbf{dg module} $M^\bullet$ for a dg algebra $A^\bullet$ is a dg vector space  $M^\bullet$ with a map:
$$
A^\bullet \otimes M^\bullet \longrightarrow M^\bullet
$$
which is associative, preserves the grading, and satisfies the graded Leibniz rule (i.e. \eqref{eqn:leibniz} with $a'$ replaced by $m$). Note that all the cohomologies $H^i(M^\bullet)$ are modules for $H^0(A^\bullet)$. 

\end{definition}

When the grading will not be particularly crucial, we may simplify notation by writing $A = A^\bullet$ and $M = M^\bullet$. We will only studied the derived category $A$--modules:
$$
A\text{--Mod} = \Big\{\text{dg modules }M \curvearrowleft A \Big\}/\text{quasi--isomorphism}
$$
When the dg algebra $A$ is finitely generated over $H^0(A)$, we will call an object of $A\text{--Mod}$ \textbf{finitely presented} if all its cohomologies have this property over $H^0(A)$. Then we write:
$$
A\text{--mod} \subset A\text{--Mod}
$$
for the full subcategory of finitely presented modules. The category of dg modules behaves much like that of usual modules, but with certain particular features. First of all is the existence of the grading shift:
$$
M^\bullet[1] = M^{\bullet+1}
$$
Given two $A$--modules $M$ and $M'$, one can define the space of degree preserving homomorphisms between them as $\Hom_A(M,M')$. But it is more naturally to consider instead:
\begin{equation}
\label{eqn:hom dg}
\Hom^\bullet_A(M,M') = \bigoplus_{n\in \ZZ} \Hom_A(M,M'[n])
\end{equation}
which is actually a dg vector space with respect to:
$$
d(f) = d \circ f - (-1)^n f \circ d \qquad \forall \ f:M \rightarrow M'[n]
$$
The spaces \eqref{eqn:hom dg} make $A$--Mod and $A$--mod into \textbf{dg categories}, which just means a category whose Hom spaces are dg vector spaces. We may inquire about the ordinary categories:
\begin{equation}
\label{eqn:homotopy category}
H^0(A\text{--Mod}) \qquad \text{and} \qquad H^0(A\text{--mod})
\end{equation}
whose Hom spaces are, by definition, the 0--th cohomologies of \eqref{eqn:hom dg}. Because the zero--cycles of \eqref{eqn:hom dg} are degree and differential preserving maps $f:M \rightarrow M'$, while the zero--boundaries are homotopies between such maps, we conclude that \eqref{eqn:homotopy category} is nothing but the homotopy category of $A$--modules. So the dg category $A$--mod supersedes the homotopy category.

\subsection{Symmetric and exterior algebras}
\label{sub:sym ext}

There will be two main examples of dg algebras, both associated to a vector space $V$. The first is the \textbf{symmetric algebra}:
\begin{equation}
\label{eqn:symm}
S V = \bigoplus_{d=0}^\infty S^d V
\end{equation}
concentrated in degree 0 and with trivial differential, and the \textbf{exterior algebra}:
\begin{equation}
\label{eqn:ext}
\wedge V = \bigoplus_{d=0}^\infty \wedge^d V
\end{equation}
situated in degrees $...,-2,-1,0$ and with trivial differential. By definition, the spaces \eqref{eqn:symm} and \eqref{eqn:ext} are quotients of the tensor algebra of $V$ by the relations $v\otimes v' \mp v' \otimes v$. Therefore, they are both particular cases of the symmetric algebra of a dg vector space:
\begin{equation}
\label{eqn:symm ext}
S V^\bullet := \left(\bigoplus_{n=0}^\infty V^\bullet \otimes ... \otimes V^\bullet \right) \Big/ \left( v \otimes v' - (-1)^{(\deg v)(\deg v')} v' \otimes v \right)
\end{equation}
which inherits the differential from $V^\bullet$:
$$
d(v_1 \otimes ... \otimes v_k) = \sum_{i=1}^k (-1)^{\deg v_1 + ... + \deg v_{i-1}} \cdot v_1 \otimes ... \otimes v_{i-1} \otimes d(v_i) \otimes v_{i+1} \otimes ... \otimes v_k
$$
By the very definition, \eqref{eqn:symm ext} is a commutative dg algebra, which is concentrated in non-positive degrees as long as the original dg vector space $V^\bullet$ is. In particular, when the dg vector space is concentrated in degree 0 (respectively -1), we obtain \eqref{eqn:symm} (respectively \eqref{eqn:ext}.

\begin{example}
\label{ex:important}

A particularly important case of the construction \eqref{eqn:symm ext} is when:
$$
V^\bullet = \left[ M \stackrel{s}\longrightarrow N \right]
$$
is concentrated in degrees $-1$ and 0. Then we have:
$$
S V^\bullet = \left[ ... \stackrel{d_s}\longrightarrow \wedge^2 M \otimes S N \stackrel{d_s}\longrightarrow M \otimes S N \stackrel{d_s}\longrightarrow S N \right]
$$
in degrees $...,-2,-1,0$, with differential given by:
\begin{equation}
\label{eqn:ds}
d_s(m_1 \wedge... \wedge m_k \otimes n) = (-1)^{k-1} \sum_{i=0}^k m_1 \wedge ... \wedge m_{i-1} \wedge m_{i+1} \wedge ... \wedge m_k \otimes s(m_i)n
\end{equation}
for all $m_1,...,m_k \in M$ and $n \in S N$.

\end{example}

More generally, suppose that $A$ is a dg algebra and $M$ is a dg module for $A$. Define:
$$
S_{A} M^\bullet = SM^\bullet \Big/ \left( am \otimes m' - m \otimes a m' \right)
$$
which will also be a dg module for $A$. The formalism above, as well as Example \ref{ex:important}, apply.

\subsection{Affine dg schemes}
\label{sub:affine dg schemes}

Dg schemes can be defined as spectra of dg algebras with respect to the \'{e}tale topology, as detailed in \cite{B}. We will not need the full theory, and instead follow the original definition of Kontsevich. 

\begin{definition}
\label{def:dg scheme}

If $X$ is an scheme with structure sheaf $\CO_X$, an \textbf{affine dg} scheme supported on $X$ is a sheaf $\CA$ of dg algebras, concentrated in non-positive degrees, such that $\CO_X = H^0(\CA)$. 

\end{definition}

We will write $\spec \ \CA$ for the affine dg scheme associated to $\CA$, to match this situation with that of usual schemes. Philosophically, the approach of Definition \ref{def:dg scheme} can be summarized by saying that we ignore topological subtleties of dg schemes, and simply endow them with the topology coming from $\CO_X$. The natural definition of quasi-coherent sheaves is: 
$$
\qcoh(\spec \ \CA) = \CA\text{--Mod} = \frac {\Big\{ \CP \in \qcoh(X) \text{ endowed with a dg module structure for }\CA \Big\}}{\text{quasi--isomorphism}}
$$
All of the dg schemes in this paper will be of finite type, meaning that $\CA$ is finitely generated over $\CO_X = H^0(\CA)$. Since this is the case, it is natural to define coherent-sheaves as the full subcategory:
$$
\CA\text{--mod} = \coh(\spec \ \CA) \subset \qcoh(\spec \ \CA)
$$
consisting of dg modules whose cohomology groups are coherent sheaves over $\CO_X = H^0(\CA)$.

\begin{example}
\label{ex:koszul}

Suppose that $\CA = S_X[\CN \stackrel{s}\rightarrow \CO_X]$ is the Koszul complex associated to a coherent sheaf $\CN$ and a co-section $s$. Explicitly, we have:
$$
\CA = \left[ ... \stackrel{d_s}\longrightarrow \wedge^2 \CN \stackrel{d_s}\longrightarrow \CN \stackrel{d_s}\longrightarrow \CO_X \right]
$$
The structure sheaf $\CO_X$ situated in degree 0, as in Example \ref{ex:important}, upgraded to the situation of modules. If the co-section $s$ is regular, then it is well-known that the Koszul complex is acyclic, and the dg algebra $\CA$ becomes isomorphic to the usual commutative algebra $\CO_X/s$. In this case, the dg scheme is simply the subscheme of $X$ cut out by the section $s$.

However, in general it may be that the section $s$ is not regular (for example, $s$ could be 0). In this case, the dg algebra $\CA = \wedge^\bullet \CN$ has 0 differential but non-trivial grading. Explicitly:
$$
\CA\text{--mod} = \frac {\Big\{\text{graded coherent }\CO_X \curvearrowright \CP^\bullet \text{ together with } \CN \otimes \CP^\bullet \stackrel{\lambda}\rightarrow \CP^{\bullet-1} \text{ such that } \lambda \circ \lambda = 0\Big\}}{\text{quasi--isomorphism}}
$$
In particular, if $\CN \cong \CO_X^{\oplus n}$ is a free module, the choice of the datum $\lambda$ corresponds to $n$ commuting degree $-1$ endomorphisms of $\CP$. 

\end{example}

\begin{example}
\label{ex:affine dg}

In general, the affine dg schemes we will encounter will combine the previous example with the case of polynomial rings over ordinary algebras. Specifically, we will have:
$$
\CA = S_X[\CM \stackrel{s}\longrightarrow \CN] = \left[ ... \stackrel{d_s}\longrightarrow \wedge^2 \CM \otimes S_X\CN \stackrel{d_s}\longrightarrow \CM \otimes S_X\CN \stackrel{d_s}\longrightarrow S_X\CN \right]
$$
where $\CM \stackrel{s}\rightarrow \CN$ is a map of coherent sheaves of $X$. The differential $d_s$ is given by \eqref{eqn:ds}, and the grading has $\wedge^i\CM \otimes S\CN$ sitting in degree $-i$. But note that there is an extra grading on the algebra $\CA$, given by placing $\wedge^i \CM \otimes S^j \CN$ in degree $i+j$. We will write this as:
$$
\CA^{\bullet,*} = \bigoplus_{i,j \geq 0} \CA^{-i,i+j} = \bigoplus_{i,j \geq 0} \wedge^i \CM \otimes S^j \CN
$$
Since the $* = i+j$ grading is preserved by the differential $d_s$, it descends to a grading on the cohomology groups. For example, when the morphism $s$ is regular (i.e. when the Koszul complex $\CA$ is acyclic in negative degrees), the $\bullet$ grading collapses, and the $*$ grading matches the usual polynomial grading on the symmetric power $S^*_X (\CN/\CM)$. 

\end{example}

\subsection{Projective dg bundles}
\label{sub:proj dg}

We do not wish to define projective dg schemes in complete generality, but instead focus on projectivizations of dg vector bundles $\CV^\bullet$ on a space $X$.

\begin{definition}
\label{def:prog dg space}

A projective dg bundle is defined through its category of coherent sheaves:
$$
\coh(\proj \ S_X \CV^\bullet) = \frac {\left\{\text{graded }S^*_X\CV^\bullet \text{ dg modules} \right\}}{\left( S^{*}\CV^\bullet/S^{* > 0}\CV^\bullet \right) \cong 0}
$$

\end{definition}

Let us make two remarks: first of all, an object in $\coh(\proj \ S_X \CV^\bullet)$ has two gradings. The first comes from the power $*$ of the symmetric power, and the second comes from the dg grading on $\CV^\bullet$. Secondly, the difference between a projectivization and the affine cone $\spec \ S_X\CV^\bullet$ is the same as in the classical case: there is, in the derived category of the former, an additional quasi-isomorphism between the structure sheaf of the zero section and the zero module.

\begin{example}
\label{ex:proj dg}

As in Example \ref{ex:proj dg}, let us study the case when $\CV^\bullet = [\CM \stackrel{s}\longrightarrow \CN]$ is a two step complex of vector bundles, concentrated in degrees $-1$ and $0$. In this case, we have a map:
\begin{equation}
\label{eqn:support}
\includegraphics{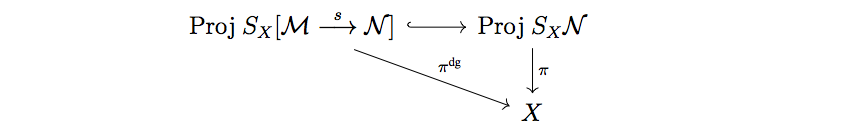}
\end{equation}
where the map $\pi$ is an actual projective bundle since $\CN$ is a vector bundle on $X$. The symbol $\hookrightarrow$ emulates closed embeddings of schemes, because we tautologically have:
\begin{equation}
\label{eqn:coherent dg}
\coh \left( \proj \ S_X [\CM \stackrel{s}\longrightarrow \CN] \right) \cong 
\end{equation}
$$
\cong \Big\{\text{coherent sheaves on }\proj \ S_X  \CN \text{ endowed with a dg action of }\wedge^\bullet[\pi^*\CM(-1) \stackrel{s}\rightarrow \CO_{\proj \ S_X  \CN}] \Big\}
$$
With this in mind, we think of $\proj \ S_X [\CM \stackrel{s}\longrightarrow \CN]$ as the dg subscheme of $\proj \ S_X  \CN$ cut out by the cosection $s$ of the vector bundle $\pi^*\CM(-1)$.

\end{example}

Our main Example \ref{ex:proj dg} should be interpreted as a dg version of the familiar notion of projective bundles $\proj \ S_X\CV \stackrel{\pi}\longrightarrow X$, where $\CV$ is a rank $n$ locally free sheaf of $X$. In this case, recall the following formulas:
$$
\pi_*(\CO(k)) = S^k\CV \qquad \quad \qquad \text{ concentrated in degree 0}
$$
$$
\pi_*(\CO(-k)) = S^{k-n} \CV^\vee \otimes \wedge^{\text{top}} \CV^\vee \quad \text{concentrated in degree }n-1
$$
for all $k\in \NN$, where $\pi_*$ denotes the derived pull-back. The second equality follows from the first one, together with \textbf{relative Serre duality}:
\begin{equation}
\label{eqn:serre scheme}
R^\bullet \pi_*(\CA) = R^{\bullet - n + 1} \pi_{*}(\CA^\vee \otimes \wedge^\text{top} \CV(-n))^\vee
\end{equation}
for all $\CA \in D^b(\coh(\proj \ S_X\CV))$. We now prove a similar formula in the dg setting

\begin{proposition}
\label{prop:canonical} 

In the notation of Example \ref{ex:proj dg}, suppose $\text{rank }\CM = m$ and $\text{rank }\CN = n$. Then:
\begin{equation}
\label{eqn:serre dg scheme}
R^\bullet \pi^\edg_*(\CA) = R^{\bullet - n + m + 1} \pi^\edg_{*}\left(\CA^\vee \otimes \frac {\wedge^\emph{top}\CN(-n)}{\wedge^\emph{top}\CM(-m)}\right)^\vee
\end{equation}
for all $\CA \in D^b(\coh(\proj \ S_X [\CM \stackrel{s}\longrightarrow \CN] ))$

\end{proposition}

\begin{proof} Implicitly in equation \eqref{eqn:coherent dg}, one has the equation:
$$
R^\bullet \pi^\dg_*(\CA) = R^\bullet \pi_* \left(\CA \otimes \wedge^\bullet[\pi^*\CM(-1) \stackrel{s}\rightarrow \CO] \right)
$$
Applying \eqref{eqn:serre scheme} to the right hand side, we obtain
$$
R^\bullet \pi^\dg_*(\CA) = R^{\bullet-n+1} \pi_{*} \left(\CA^\vee \otimes \wedge^\bullet \left[\pi^*\CM(-1) \stackrel{s}\rightarrow \CO \right]^\vee \otimes \wedge^\text{top} \CN^\vee(-n) \right)^\vee
$$
It is easy to see that $\wedge^\bullet \left[\pi^*\CM(-1) \stackrel{s}\rightarrow \CO\right]^\vee = \wedge^{\bullet+m} \left[\pi^*\CM(-1) \stackrel{s}\rightarrow \CO\right] \otimes \wedge^{\text{top}}\CM^\vee(m)$, hence:
$$
R^\bullet \pi^\dg_*(\CA) = R^{\bullet-n+m+1} \pi_{*} \left(\CA^\vee \otimes \wedge^\bullet \left[\pi^*\CM(-1) \stackrel{s}\rightarrow \CO\right] \otimes \frac {\wedge^{\text{top}}\CN(-n)}{\wedge^{\text{top}}\CM(-m)} \right)^\vee
$$
which equals $R^{\bullet - n + m + 1} \pi^\dg_{*}\left(\CA^\vee \otimes \frac {\wedge^\text{top}\CN(-n)}{\wedge^\text{top}\CM(-m)}\right)^\vee$ by another application of \eqref{eqn:coherent dg}. 

\end{proof}

\end{document}